\documentclass[a4paper,reqno,10.5pt, oneside]{amsart}

\usepackage{amssymb}
\usepackage{amstext}
\usepackage{amsmath}
\usepackage{amscd}
\usepackage{amsthm}
\usepackage{amsfonts}
\usepackage{enumerate}
\usepackage{graphicx}
\usepackage{latexsym}
\usepackage{mathrsfs}
\usepackage{mathtools}
\usepackage{here}

\usepackage{tikz}
\usepackage{tikz-cd}
\usetikzlibrary{arrows}
\usetikzlibrary{calc,decorations.markings}
\usetikzlibrary{patterns,decorations.pathreplacing}

\tikzset{->-/.style={decoration={
  markings,
  mark=at position .5 with {\arrow{>}}},postaction={decorate}}}

\usepackage{hyperref}

\usepackage{lscape}
\usepackage{comment}

\newtheorem{theorem}{Theorem}[section]

\newtheorem{corollary}[theorem]{Corollary}
\newtheorem{lemma}[theorem]{Lemma}
\newtheorem{proposition}[theorem]{Proposition}
\newtheorem{definition-proposition}[theorem]{Definition-Proposition}

\newtheorem{question}[theorem]{Question}

\theoremstyle{definition}
\newtheorem{definition}[theorem]{Definition}

\newtheorem{example}[theorem]{Example}

\theoremstyle{remark}
\newtheorem{remark}[theorem]{Remark}

\makeatletter
  
  \@addtoreset{equation}{section}
\makeatother

\newcommand{\rme}{\mathrm{e}}

\newcommand{\rmg}{\mathrm{g}}

\newcommand{\rmw}{\mathrm{w}}

\newcommand{\rmH}{\mathrm{H}}

\newcommand{\rmZ}{\mathrm{Z}}

\newcommand{\calI}{\mathcal{I}}

\newcommand{\calM}{\mathcal{M}}

\newcommand{\calP}{\mathcal{P}}

\newcommand{\calZ}{\mathcal{Z}}

\newcommand{\fkm}{\mathfrak{m}}

\newcommand{\fkp}{\mathfrak{p}}

\newcommand{\fkS}{\mathfrak{S}}

\newcommand{\ZZ}{\mathbb{Z}}
\newcommand{\QQ}{\mathbb{Q}}
\newcommand{\RR}{\mathbb{R}}

\newcommand{\sfA}{\mathsf{A}}

\newcommand{\ttI}{\mathtt{I}}

\newcommand{\sfM}{\mathsf{M}}
\newcommand{\sfN}{\mathsf{N}}
\newcommand{\sfP}{\mathsf{P}}

\newcommand{\sfT}{\mathsf{T}}

\newcommand{\height}{\operatorname{ht}}

\newcommand{\Spec}{\operatorname{Spec}}

\newcommand{\depth}{\operatorname{depth}}

\newcommand{\rank}{\operatorname{rank}}
\newcommand{\Ext}{\operatorname{Ext}}

\newcommand{\Hom}{\operatorname{Hom}}

\newcommand{\GL}{\operatorname{GL}}
\newcommand{\SL}{\operatorname{SL}}

\newcommand{\Cl}{\operatorname{Cl}}

\newcommand{\End}{\operatorname{End}}
\newcommand{\gldim}{\mathrm{gl.dim}}

\newcommand{\add}{\mathsf{add}}
\newcommand{\CM}{\mathsf{CM}}

\newcommand{\mc}{\mathsf{mod}}
\newcommand{\refl}{\mathsf{ref}}

\begin{document}

\title[Semi-steady NCCRs via regular dimer models]{Semi-steady non-commutative crepant resolutions \\ via regular dimer models}
\author[Y. Nakajima]{Yusuke Nakajima} 

\address[Y. Nakajima]{Kavli Institute for the Physics and Mathematics of the Universe (WPI), UTIAS, The University of Tokyo, Kashiwa, Chiba 277-8583, Japan} 
\email{yusuke.nakajima@ipmu.jp}



\maketitle

\begin{abstract} 
A consistent dimer model gives a non-commutative crepant resolution (= NCCR) of a $3$-dimensional Gorenstein toric singularity. 
In particular, it is known that a consistent dimer model gives a class of NCCRs called steady if and only if it is homotopy equivalent to a regular hexagonal dimer model. 
Inspired by this result, we detect another nice property on NCCRs that characterizes square dimer models. 
We call such NCCRs semi-steady NCCRs, and study their properties. 
\end{abstract}


\section{\bf Introduction} 

\subsection{Overview and Motivations} 

The notion of non-commutative crepant resolution (= NCCR) was introduced by Van den Bergh \cite{VdB2} (see also \cite{VdB1}). 
It is an algebra derived equivalent to crepant resolutions for some singularities, and it gives another perspective on Bondal-Orlov conjecture \cite{BO} and Bridgeland's theorem \cite{Bri}. 
For example, an NCCR of a quotient singularity is given by the skew group algebra (see e.g., \cite{VdB2,Iya,IW2}), and 
if a given quotient singularity is $d$-dimensional Gorenstein with $d\le 3$, the skew group algebra is derived equivalent to crepant resolutions of such a singularity \cite{BKR,KV}. 
NCCRs are also related with Cohen-Macaulay representation theory. 
Indeed, cluster tilting modules (or subcategories) give a framework to study modules giving NCCRs (see e.g., \cite{DH,Iya,IR,IW2,IW4}), and the present paper follows this viewpoint. 
Here, we recall the definition of NCCRs \cite{VdB2}. (For further details concerning terminology used in this introduction, see later sections.) 

\begin{definition}
\label{NCCR_def}
Let $R$ be a Cohen-Macaulay normal domain, and $M$ be a non-zero reflexive $R$-module. Let $\Lambda\coloneqq\End_R(M)$. 
We say that $\Lambda$ is a \emph{non-commutative crepant resolution} (= \emph{NCCR}) of $R$ 
or $M$ \emph{gives an NCCR} of $R$ if $\Lambda$ is a \emph{non-singular $R$-order}, that is, 
$\gldim\,\Lambda_\fkp=\dim\,R_\fkp$ for all $\fkp\in\Spec R$ and $\Lambda$ is a maximal Cohen-Macaulay $R$-module. 
\end{definition}

We refer to \cite[Proposition~2.17]{IW2} for several conditions that are equivalent to $\Lambda$ is a non-singular $R$-order. 
Also, the existence of NCCRs and their properties have been studied in several papers e.g., \cite{Bro,BLVdB,BIKR,Dao,DFI,HN,IU2,IW1,IW2,IW4,Leu2,SpVdB,Wem} and references therein. 

\medskip

One of the interesting families of NCCRs is given by dimer models, and we will mainly discuss such NCCRs in this paper. 
A \emph{dimer model} is a finite bipartite graph on the real two-torus. 
We define the quiver with potential $(Q, W_Q)$ as the dual of a dimer model, and we then define the Jacobian algebra $\calP(Q, W_Q)$ which is 
the path algebra with certain relations arising from the potential $W_Q$. 
If a dimer model satisfies the consistency condition (see Definition~\ref{def_consistent}), then the center of $\calP(Q, W_Q)$ is a $3$-dimensional Gorenstein toric singularity and $\calP(Q, W_Q)$ gives an NCCR of such a singularity.  
Conversely, for every $3$-dimensional Gorenstein toric singularity $R$, there exists a consistent dimer model giving $R$ as the center of $\calP(Q, W_Q)$. 
Thus, every $3$-dimensional Gorenstein toric singularity admits an NCCR. For more details, see e.g., \cite{Bro,IU2,Boc4} and Section~\ref{NCCR_via_dimer}. 

Although an NCCR does not necessarily exist for a given singularity in general, the existence of an NCCR shows that 
it has at worst log-terminal singularities \cite{DITW} (see also \cite{StVdB,DITV}). 
Furthermore, if we impose several conditions on NCCRs, then we have more concrete singularities. 
Indeed, Iyama and the author introduced the notion of steady NCCRs and splitting NCCRs in \cite{IN}, and studied singularities admitting steady splitting NCCRs. Here, we note the definition of steady and splitting NCCRs. 

\begin{definition}
\label{def_steady}
Let $R$ be a Cohen-Macaulay normal domain, and $M$ be a non-zero reflexive $R$-module. 
\begin{enumerate}[\rm (1)]
\item We say that $M$ is \emph{steady} if $M$ is a \emph{generator} (that is, $R\in\add_RM$) and $\End_R(M)\in\add_RM$ holds. 
We say that an NCCR $\End_R(M)$ is a \emph{steady NCCR} if $M$ is steady.
\item We say that $M$ is \emph{splitting} if $M$ is a finite direct sum of rank one reflexive modules. 
We say that an NCCR $\End_R(M)$ is a \emph{splitting NCCR} if $M$ is splitting. 
\end{enumerate}
\end{definition}

Using these notions, we see that the existence of a steady splitting NCCR characterizes quotient singularities associated with finite abelian groups (see \cite[Theorem~3.1]{IN}). 
Restricting this result to NCCRs arising from consistent dimer models, we have the following theorem. We note that NCCRs arising from consistent dimer models are always splitting, thus we may not mention the splittingness in this theorem. 

\begin{theorem}[{see \cite[Corollary~1.7]{IN}}]
\label{dimer} 
Let $\Gamma$ be a consistent dimer model, $R$ be the $3$-dimensional complete local Gorenstein toric singularity associated with $\Gamma$, 
and $k$ be an algebraically closed field of characteristic zero. Then, the following conditions are equivalent.
\begin{enumerate}[\rm (1)]
\item $R$ is a quotient singularity associated with a finite abelian group $G\subset\SL(3, k)$ $($i.e., $R=S^G$ where $S=k[[x_1,x_2, x_3]]$$)$. In particular, the cone defining $R$ is simplicial and hence $R$ is $\QQ$-factorial. 
\item $\Gamma$ is homotopy equivalent to a regular hexagonal dimer model $($i.e., each face of a dimer model is a regular hexagon$)$. 
\item $\Gamma$ gives a steady NCCR of $R$.
\end{enumerate}
\end{theorem}

In this way, we could characterize a regular hexagonal dimer model, which is a typical dimer model as we will see below, using the nice class of NCCRs. Thus, we then ask the following. 

\begin{question}
\label{que_dimer}
Can we characterize other dimer models by using NCCRs ? 
\end{question}

Since a dimer model is a bipartite graph on the real two-torus, the universal cover of it gives rise to the one on the Euclidean plane, 
hence dimer models are closely related with tilings of the Euclidean plane. 
A \emph{tiling} (or \emph{tessellation}) is a covering of the Euclidean plane using one or more polygons without overlaps and gaps. 
A \emph{regular tiling} is a tiling that is made up of congruent regular polygons and edge-to-edge. 
Here, edge-to-edge means any two polygons intersect precisely along a common edge, 
or have precisely one common point which is a vertex of a polygon, or have no common points. 
It is well-known that regular polygons giving regular tilings are only equilateral triangles, squares, or regular hexagons (see e.g., \cite{GS}). 

\begin{definition}
We say that a dimer model $\Gamma$ is \emph{regular} if the underlying cell decomposition of the universal cover of $\Gamma$ is 
homotopy equivalent to a regular tiling. 
\end{definition}

Since we can not realize a regular tiling consisting of equilateral triangles as a dimer model, 
a dimer model is regular if and only if it is homotopy equivalent to a square dimer model or a regular hexagonal dimer model. 
Thus, we next consider nice properties on NCCRs that characterize square dimer models. 
In this paper, in order to give a partial answer to Question~\ref{que_dimer}, we will introduce the notion of semi-steady NCCRs, which is weaker than the steadiness. 
We then study basic properties of semi-steady NCCRs, 
and as a result we show that the semi-steadiness actually characterizes NCCRs arising from square dimer models (see Theorem~\ref{main_intro} below). 
We also mention that we have several examples of semi-steady NCCRs even if a given singularity is not toric. 
(see Example~\ref{ex_Ansing} and \ref{cAn-sing}). 

\subsection{Semi-steady non-commutative crepant resolutions}
\label{subsec_intro}
In this subsection, we introduce the notion of semi-steady NCCRs. 
Let $R$ be a Cohen-Macaulay normal domain. Since non-singular $R$-orders are closed under Morita equivalence 
(see e.g., \cite[Lemma~2.13]{IW2}), 
we assume that a module $M=\bigoplus^n_{i=1}M_i$ giving an NCCR is \emph{basic}, that is, $M_i$'s are mutually non-isomorphic. 
In addition, since we will discuss memberships of additive closures, we assume that $R$ is complete local by \cite[Proposition~2.26]{IW2}. 
In particular, the Krull-Schmidt condition holds in our situation. 

We now recall that if $M=\bigoplus_{i=0}^nM_i$ is a steady module, then it implies $\rme_i\End_R(M)\cong\Hom_R(M_i,M)\in\add_RM$ for any $i$, where $\rme_i$ is the idempotent corresponding to the summand $M_i$. 
On the other hand, the semi-steadiness allows $\Hom_R(M_i,M)$ to be in $\add_RM^*$ as follows. 
Here, $M^*$ denotes the $R$-dual of $M$.  

\begin{definition}
\label{def_semisteady}
Let $M=\bigoplus_{i=0}^nM_i$ be the indecomposable decomposition of a reflexive $R$-module $M$. 
We say that $M=\bigoplus_{i=0}^nM_i$ is \emph{semi-steady} if $M$ is a generator and 
$\Hom_R(M_i,M)\in\add_RM$ or $\add_RM^*$ for all $i=0,\cdots,n$. 
In addition, we say that an NCCR $\End_R(M)$ is a \emph{semi-steady NCCR} if $M$ is semi-steady. 
\end{definition}

We note that the condition ``$M$ is a generator" can be obtained from the condition that $\Hom_R(M_i,M)\in\add_RM$ or $\add_RM^*$ for all $i=0,\cdots,n$ 
in many cases (see Lemma~\ref{prop_semisteady}). 
Also, we may change the later condition to $\Hom_R(M,M_i)\in\add_RM$ or $\add_RM^*$ for all $i=0,\cdots,n$ when $R$ is a normal domain (see Lemma~\ref{dual_def}). 
Further, we can easily see that a steady module is a semi-steady module. 
In particular, the next lemma follows from the definition and \cite[Lemma~2.5(b)]{IN}. 

\begin{lemma}
\label{semisteady_iff_steady}
Let $M$ be a reflexive $R$-module. Then, $M$ is steady if and only if $M$ is semi-steady and $\add_RM=\add_RM^*$ holds. 
\end{lemma}

Considering semi-steady NCCRs, we can characterize square dimer models as follows. 
(For more details regarding terminologies, see Section~\ref{NCCR_via_dimer}.) 

\begin{theorem}[{see Theorem~\ref{main_thm} for more precise version}]
\label{main_intro}
Let $\Gamma$ be a consistent dimer model. 
Suppose that $R$ is the $3$-dimensional complete local Gorenstein toric singularity associated with $\Gamma$. 
Then, the following conditions are equivalent.
\begin{enumerate}[\rm (1)]
\item $\Gamma$ is homotopy equivalent to a square dimer model (i.e., each face of the dimer model is a square).
\item $\Gamma$ is isoradial and gives a semi-steady NCCR of $R$ that is not steady. 
\end{enumerate}
When this is the case, we also see that the toric singularity $R$ corresponding to such a dimer model 
is the one associated with a parallelogram. 
\end{theorem}

Thus, we immediately have the following corollary by combining Theorem~\ref{dimer} and \ref{main_intro}. 

\begin{corollary}[{see Corollary~\ref{main_cor}}] 
\label{main_cor_intro}
With the notation as above, the following conditions are equivalent.
\begin{enumerate}[\rm (1)]
\item The dimer model $\Gamma$ is isoradial and gives a semi-steady NCCR of $R$.
\item The dimer model $\Gamma$ is homotopy equivalent to a regular dimer model. 
\end{enumerate}
\end{corollary}

The content of this paper is the following. 
First, we observe some basic properties of semi-steady modules in Section~\ref{sec_semisteady}. 
The remarkable thing is that a singularity admitting a semi-steady NCCR has the typical class group (see Theorem~\ref{class_semi_steady}). 
Since the main purpose of this paper is to investigate NCCRs arising from dimer models, 
we review some basic results regarding toric singularities and dimer models in Section~\ref{NCCR_via_dimer}. 
In particular, we explain that how to construct splitting NCCRs using consistent dimer models. 
After that, we prove Theorem~\ref{main_intro} in Section~\ref{sec_semisteady_dimer}. 
In Section~\ref{subsec_ex}, we give several examples of semi-steady NCCRs arising from regular dimer models. 


\subsection*{Notations and Conventions} 
Throughout this paper, we will assume that $k$ is an algebraically closed field of characteristic zero, 
and a commutative noetherian ring $R$ is complete local, thus the Krull-Schmidt condition holds (see the beginning of subsection~\ref{subsec_intro}). 

In this paper, all modules are left modules, and we denote 
by $\mc R$ the category of finitely generated $R$-modules, 
by $\add_RM$ the full subcategory consisting of direct summands of finite direct sums of copies of $M\in \mc R$. 
We suppose that $M=\bigoplus^n_{i=0}M_i$ always denotes the indecomposable decomposition of an $R$-module $M$. 
When we consider a composition of morphism, $fg$ means we firstly apply $f$ then $g$. 
With this convention, $\Hom_R(M, X)$ is an $\End_R(M)$-module and $\Hom_R(X, M)$ is an $\End_R(M)^{\rm op}$-module. 
Similarly, when we consider a quiver, a path $ab$ means $a$ then $b$. 

In addition, we denote by $\Cl(R)$ the class group of $R$. 
When we consider a divisorial ideal (rank one reflexive $R$-module) $I$ as an element of $\Cl(R)$, we denote it by $[I]$.


\section{\bf Basic properties of semi-steady NCCRs} 
\label{sec_semisteady}

In this section, we present some basic properties of semi-steady modules. 

We start this section with preparing some notions used in this paper. 
We denote the $R$-dual functor by $(-)^*:=\Hom_R(-, R) : \mc R\rightarrow \mc R$. 
We say that $M\in \mc R$ is \emph{reflexive} if the natural morphism $M\rightarrow M^{**}$ is an isomorphism. 
We denote by $\refl R$ the full subcategory of reflexive $R$-modules. 
For $M\in\mc R$, we define the depth of $M$ as 
\[
\depth_RM:=\mathrm{inf} \{ i\ge0 \mid \Ext^i_R(R/\fkm, M)\neq 0\}, 
\]
where $\fkm$ is the maximal ideal of $R$. 
We say that $M$ is a \emph{maximal Cohen-Macaulay} (= \emph{MCM}) $R$-module if $\depth_RM={\rm dim}R$ or $M=0$. 
Furthermore, we say that $R$ is a \emph{Cohen-Macaulay ring} (= \emph{CM ring}) if $R$ is an MCM $R$-module. 
We denote by $\CM R$ the full subcategory of MCM $R$-modules. 

\medskip

Before moving to basic properties of semi-steady modules, we note some comments concerning the definition of semi-steady modules. 

\begin{lemma}
\label{prop_semisteady}
Let $R$ be a normal domain. 
Suppose that $M=\bigoplus_{i=0}^nM_i\in\refl R$ satisfies $\Hom_R(M_i,M)\in\add_RM$ or $\Hom_R(M_i,M)\in\add_RM^*$ for all $i$. 
Then, $M$ is a generator if one of the following conditions is satisfied. 
  \begin{enumerate}[$\bullet$]
   \item $R$ contains a field of characteristic zero, 
   \item $M$ has a rank one reflexive module as a direct summand. 
  \end{enumerate}
In particular, if $M$ is splitting, then $M$ is a generator. 
\end{lemma}

\begin{proof}
First, if $R$ contains a field of characteristic zero, we have that $R\in\add_R\End_R(M)$ by \cite[5.6]{Aus3}. 
Thus, we have that $R\in\add_R\Hom_R(M_i,M)$ for some $i$, and hence $R\in\add_RM$ or $\add_RM^*$. 
If $R\in\add_RM^*$, then we have that $R=R^*\in\add_RM^{**}=\add_RM$. 

Next, we suppose that $I$ is a rank one reflexive $R$-module such that $I\in\add_R M$. 
Then, we have that $R\cong\Hom_R(I, I)\in\add_RM$ or $\add_RM^*$. 
\end{proof}

The following lemma is basic, and useful to investigate semi-steady modules. 

\begin{lemma}
\label{basic_lem_ref}
Let $R$ be a normal domain. For any $M, N\in\refl R$, we have that 
\[
\Hom_R(M,N)^*\cong\Hom_R(N,M).
\]
\end{lemma}

\begin{proof}
Consider a natural morphism $\varphi:M^*\otimes_RN\rightarrow\Hom_R(M,N)$ ($\varphi(f\otimes y)(x)=f(x)y$ for any $x\in M, y\in N$), 
and this induces 
\[
\varphi^*:\Hom_R(M,N)^*\rightarrow (M^*\otimes_RN)^*\cong\Hom_R(N,M^{**})\cong\Hom_R(N,M). 
\]
We easily see that $\varphi_\fkp^*$ is an isomorphism for any $\fkp\in\Spec R$ with $\height\fkp=1$, 
and hence $\varphi^*$ is also an isomorphism since both are reflexive (see e.g., \cite[Lemma~5.11]{LW}). 
\end{proof}

Next, we discuss the latter condition of the definition of semi-steady modules. 

\begin{lemma}
\label{dual_def} 
Suppose that $R$ is a normal domain. For a reflexive $R$-module $M=\bigoplus_{i=0}^nM_i$, we have that 
$\Hom_R(M_i,M)\in\add_RM$ or $\add_RM^*$ holds if and only if 
$\Hom_R(M,M_i)\in\add_RM$ or $\add_RM^*$ holds.  
\end{lemma}

\begin{proof}
If $\Hom_R(M_i,M)\in\add_RM$ (resp. $\add_RM^*$), then we have $\Hom_R(M_i,M)^*\in\add_RM^*$ (resp. $\add_RM$). 
By Lemma~\ref{basic_lem_ref}, we have $\Hom_R(M_i,M)^*\cong\Hom_R(M,M_i)\in\add_RM^*$ (resp. $\add_RM$). 
By the duality, the converse also holds. 
\end{proof}

In what follows, we show basic properties of semi-steady modules (see also \cite[Lemma~2.5]{IN}). 
We remark that the converse of Lemma~\ref{basic_lem}(a) is not true (see Example~\ref{ex3}). 

\begin{lemma}
\label{basic_lem}
Suppose that $R$ is a normal domain and $M=\bigoplus_{i=0}^nM_i\in\refl R$ is semi-steady. 
Then, we have the following. 
\begin{enumerate}[\rm(a)]
\item We have that $\add_R\End_R(M)=\add_R(M\oplus M^*)$.
\item $M^*$ is also a semi-steady $R$-module. 
\end{enumerate}
\end{lemma}

\begin{proof}
(a) Since $M$ is a generator, we have that $M, M^*\in\add_R\End_R(M)$. 
In addition, we have that $\End_R(M)\in\add_R(M\oplus M^*)$ by the definition of semi-steady module. 

(b) Clearly, $M^*$ is a generator. By Lemma~\ref{basic_lem_ref}, we have an isomorphism 
\[
\Hom_R(M_i^*,M^*)\cong\Hom_R(M,M_i)^{**}\cong\Hom_R(M,M_i)\cong\Hom_R(M_i,M)^*.
\]
Therefore, $\Hom_R(M_i^*,M^*)\in\add_RM$ or $\add_RM^*$ for all $i$. 
\end{proof}

Further, we discuss the number of direct summands in semi-steady modules. 

\begin{lemma}
\label{number_lem}
Suppose that $R$ is a normal domain and $M=\bigoplus_{i=0}^nM_i$ is a basic semi-steady module that is not steady. 
We define the sets of subscripts $\ttI\coloneqq\{i \mid \Hom_R(M_i,M)\in \add_RM\}$ and $\ttI^*\coloneqq\{i \mid \Hom_R(M_i,M)\in \add_RM^*\}$. 
Then, we have the following. 
\begin{enumerate}[\rm(a)]
\item Let $\calI$ (resp. $\calI^*$) be the number of elements in $\ttI$ (resp. $\ttI^*$). 
Then, we have that $\calI=\calI^*$. 
\item $n+1$ $($$=$ the number of direct summands in $M$$)$ is an even number.  
\end{enumerate}
\end{lemma}

\begin{proof}
(a) First, we have that $M\not\cong M^*$, because $M$ is not steady (see Lemma~\ref{semisteady_iff_steady}). 
Thus, there exists a direct summand $M_s\in\add_RM$ such that $M_s\not\in\add_RM^*$, 
and hence we have that $M_s^*\in\add_RM^*$ and $M_s^*\not\in\add_RM$. 
In the description 
\[
\End_R(M)\cong
\begin{pmatrix}
\Hom_R(M_0,M_0)&\Hom_R(M_0,M_1)&\cdots&\Hom_R(M_0,M_n) \\
\Hom_R(M_1,M_0)&\Hom_R(M_1,M_1)&\cdots&\Hom_R(M_1,M_n) \\
\vdots&\vdots&\ddots&\vdots \\
\Hom_R(M_n,M_0)&\Hom_R(M_n,M_1)&\cdots&\Hom_R(M_n,M_n) \\
\end{pmatrix}, 
\]
the number of rows in which $M_s$ (resp. $M_s^*$) appears is $\calI$ (resp. $\calI^*$). 
By Lemma~\ref{basic_lem_ref}, the number of columns in which $M_s$ (resp. $M_s^*$) appears is $\calI^*$ (resp. $\calI$). 
Since $M$ is basic, we have that $\calI=\calI^*$. 

(b) Since $n+1=\calI+\calI^*$, this follows from (a). 
\end{proof}

We remark that if an $R$-module $M$ satisfying the assumption in Lemma~\ref{number_lem} is splitting, then the definition of $\ttI$ and $\ttI^*$ can be 
replaced by $\ttI\coloneqq\{i \mid \Hom_R(M_i,M)\cong M\}$ and $\ttI^*\coloneqq\{i \mid \Hom_R(M_i,M)\cong M^*\}$ because $M$ is basic and $\rank_R\Hom_R(M_i,M)=\rank_RM$ for all $i$. 

\medskip

Next, we consider the class group $\Cl(R)$. 
We know that by \cite[Proposition~2.8]{IN} the class group of a CM normal domain having a steady splitting NCCR is a finite abelian group. 
Thus, we consider a CM normal domain having a semi-steady splitting NCCR that is not steady. 

\begin{theorem}
\label{class_semi_steady}
Let $R$ be a CM normal domain and assume that every rank one reflexive $R$-module, 
whose class in $\Cl(R)$ is a torsion element, is an MCM $R$-module (e.g., $R$ is a toric singularity). 
Suppose that $M=\bigoplus_{i=0}^nM_i$ is a basic $R$-module giving a semi-steady splitting NCCR that is not steady. 
Then, $\Cl(R)\cong\ZZ\times A$ where $A$ is the torsion subgroup and the order of $A$ is equal to $\frac{n+1}{2}$. 
In particular, $\Cl(R)$ contains a torsion element if and only if $n\neq 1$.
\end{theorem}

\begin{proof}
Let $M_0=R$. We define the set
\[
\calM\coloneqq\{ \, [M_0], [M_1],\cdots, [M_n]\, \}. 
\]
We know that $\Cl(R)$ is generated by $[M_0],\cdots,[M_n]$ (see \cite[Proposition~2.8(a)]{IN}). 

First, we assume that $\Cl(R)$ is a finite group. 
For any rank one reflexive module $N$, we consider $\End_R(M\oplus N)$. 
Since $\Cl(R)$ is finite, $[\End_R(M\oplus N)]$ is a torsion element in $\Cl(R)$, thus $\End_R(M\oplus N)\in\CM R$ by the assumption. 
By \cite[Proposition~4.5]{IW2}, this implies $N\in\add_RM$, hence we have that $\Cl(R)=\calM$. 
Thus, we see that $M$ is steady by \cite[Theorem~3.1]{IN}, and hence we conclude $\Cl(R)$ is not a finite group. 

Next we show that the rank of the free part of $\Cl(R)$ is one. Let $\ttI, \ttI^*$ be the sets as in Lemma~\ref{number_lem}. 
If $\Cl(R)$ contains $\ZZ^2$, then we can take two elements in $\calM$ generating $\ZZ^2$. 
Let $[M_1], [M_2]$ be such generators. Note that these are not torsion.
Since $M$ is semi-steady, we have that $\Hom_R(M_1,M)\cong M$ or $M^*$. 
If $1\in\ttI$ holds (i.e., $\Hom_R(M_1,M)\cong M$), then $-[M_1]\in\calM$. 
Further, since $M\cong\Hom_R(M_1,M)\cong\Hom_R(M_1,\Hom_R(M_1,M))$, we also have that $-2[M_1]\in\calM$.
By repeating this argument, we have that $-t[M_1]\in\calM$ for any integer $t\ge 1$. 
Since the number of elements in $\calM$ is finite and $[M_1]$ is torsion-free, this is a contradiction, thus we have that $1\in\ttI^*$. 
Similarly, we also have that $2\in\ttI^*$. 
Therefore, we have that $\Hom_R(M_1,M_2), \Hom_R(M_2,M_1)\in\add_RM^*$, and this also implies $\Hom_R(M_1,M_2), \Hom_R(M_2,M_1)\in\add_RM$. 
By this observation, we may write $[M_2]-[M_1]=[M_s]$ for some $s\in[0,n]$. 
If $s\in\ttI$,  then we have that 
\[
[\Hom_R(M_s,\Hom_R(M_2,M_1))]=[\Hom_R(M_2,M_1)]-[M_s]=2[M_1]-2[M_2]\in\calM. 
\]
If $s\in\ttI^*$,  then we have that 
\[
[\Hom_R(\Hom_R(M_2,M_1),M_s)]=[M_s]-[\Hom_R(M_2,M_1)]=2[M_2]-2[M_1]\in\calM. 
\]
In any case, we have that $t[M_1]-t[M_2]\in\calM$ for any non-zero integer $t$ by repeating the above argument. 
Since $[M_1], [M_2]$ are torsion-free and generators of $\ZZ^2$, this contradicts the finiteness of $\calM$. 
Therefore, we conclude $\Cl(R)\cong\ZZ\times A$ where $A$ is the torsion subgroup. 

Finally, we show that the order of $A$ is equal to $\calI=\frac{n+1}{2}$. 
(Recall that $\calI$ is the number of elements in $\ttI$, and it is the same as that of elements in $\ttI^*$.) 
Let $[M_1]$ be a torsion-free element generating the free part of $\Cl(R)$. 
Clearly, $0\in \ttI$ holds. Further, we see that $1\in\ttI^*$ by the same argument as above. 
Next, for a subscript $i\in\ttI$, we may write $\Hom_R(M_i,M_j)\cong M_k$ for some $j, k\in[0,n]$. 
In this situation, we have the following claim: 
\begin{equation}
\label{claim}
\text{If $j\in\ttI^*$, then we have that $k\in\ttI^*$}. 
\end{equation}
This follows from an isomorphism 
\begin{equation*}
\begin{split}
\Hom_R(M_k,M)&\cong\Hom_R(\Hom_R(M_i,M_j),M) \\
&\cong\Hom_R(\Hom_R(M_i,M_j),\Hom_R(M_i,M))\cong\Hom_R(M_j,M)\cong M^*.
\end{split}
\end{equation*}
Since $1\in\ttI^*$, we especially have that 
\begin{equation}
\label{bijection}
[M_1]-[M_i]=[M_k], 
\end{equation}
for some $i\in\ttI, k\in\ttI^*$, and easy to see that this equation induces a bijection between $\ttI$ and $\ttI^*$. (Note that $0\in\ttI$ corresponds to $1\in\ttI^*$.)  
Thus, the torsion subgroup $A$ is generated by $[M_i]$'s with $i\in\ttI$. 

Let $N$ be a rank one reflexive module, whose class is $[N]=\sum_{i\in\ttI}t_i[M_i]$. 
Since $[M_i]$'s are torsion elements, we may assume $t_i\in\ZZ_{\ge0}$. 
For $i\in\ttI$, we see that $M\cong\Hom_R(M_i,M)\cong\Hom_R(M_i,\Hom_R(M_i,M))$, 
and hence we may write $[M_1]-2[M_i]=[M_\ell]$ for some $\ell\in[0,n]$. Furthermore, we have that $\ell\in\ttI^*$ using the claim (\ref{claim}). 
By repeating this argument, we have that $[M_1]-[N]=[M_m]$ with $m\in\ttI^*$. 
A bijection induced by (\ref{bijection}) asserts that there exists $i^\prime\in\ttI$ such that $[N]=[M_{i^\prime}]$. 
Therefore, we have that $A=\{[M_i] \mid i\in\ttI \}$, especially $|A|=\calI$. 
\end{proof}

We give some examples of semi-steady NCCRs below. 
In particular, semi-steady NCCRs are well understood for the two dimensional case (see Proposition~\ref{twodim}). 

\begin{example}
\label{ex_Ansing}
Consider the $3$-dimensional simple singularity $R=k[[x,y,u,v]]/(x^2+y^{2n}+u^2+v^2)$ of type $A_{2n-1}$. 
It is well known that $R$ is of finite CM representation type (see e.g., \cite[Chapter~12]{Yos}), and the finitely many MCM $R$-modules are 
$R$, two modules $I$, $I^*$ with rank one, and $(n-1)$ modules $N_1,\cdots,N_{n-1}$ with rank two. 
Then, modules giving NCCRs of $R$ are only $R\oplus I$ and $R\oplus I^*$ (see \cite[Proposition~2.4]{BIKR}, \cite[Example~3.6]{Dao}). 
We easily see that they are semi-steady, but not steady. 
\end{example}

\begin{example}
\label{cAn-sing}
We consider a complete local $cA_n$-singularity $R=k[[x,y,u,v]]/(f-uv)$ where $f\in\fkm=(x,y)$. 
Let $f=f_1\cdots f_n$ be a decomposition of $f$ into prime elements in $k[[x,y]]$. 
(Note that some elements $f_i$ might be the same element.) 
We consider a subset $I\subset\{1,\cdots,n\}$, and set $f_I=\prod_{i\in I}f_i$. 
Further, we define the ideal $T_I\coloneqq (u,f_I)\subset R$. 
For each $\omega\in\fkS_n$, we consider the maximal flag which is a sequence of subsets: 
\[
I_1^\omega=\{\omega(1)\}\subset I_2^\omega=\{\omega(1),\omega(2)\}\subset\cdots\subset I_{n-1}^\omega=\{\omega(1),\omega(2),\cdots,\omega(n-1)\}. 
\]
If $f_i\not\in\fkm^2$ for all $i$, then modules giving NCCRs of $R$ are precisely 
$$T^\omega\coloneqq R\oplus\bigoplus^{n-1}_{j=1}T_{I_j^\omega}$$ where $\omega\in\fkS_n$ (see \cite[Theorem~5.1]{IW4}) 
and clearly all NCCRs are splitting. 

Furthermore, by using results in \cite[Section~5]{IW4}, we can show the following: 
\begin{enumerate}[(a)]
\item $R$ has a steady NCCR if and only if $f=f_1^n$. In this case, maximal flags are only 
\[
\{1\}\subset\{1, 1\}\subset\cdots\subset\{\underbrace{1,1,\cdots,1}_{n-1}\}, 
\]
and this gives a unique steady splitting NCCR. Further, $R$ is isomorphic to the invariant subring under the action of the cyclic group 
generated by $\mathrm{diag}(1,\zeta_n,\zeta^{-1}_n)$ where $\zeta_n$ is a primitive $n$-th root of unity, 
and it is the polynomial extension of a $2$-dimensional $A_{n-1}$-singularity. 
\item $R$ has a semi-steady NCCR that is not steady if and only if $f=f_1^af_2^a$ where $n=2a$. 
In this case, the following two maximal flags give semi-steady NCCRs that are not steady. 
\[
\{1\}\subset\{1,2\}\subset\{1,1,2\}\subset\{1,1,2,2\}\subset\cdots\subset\{\underbrace{1,1,\cdots,1}_{a},\underbrace{2,2,\cdots,2}_{a-1}\}, 
\]
\[
\{2\}\subset\{1,2\}\subset\{1,2,2\}\subset\{1,1,2,2\}\subset\cdots\subset\{\underbrace{1,1,\cdots,1}_{a-1},\underbrace{2,2,\cdots,2}_{a}\}. 
\]
\end{enumerate}
\end{example}

\begin{proposition}
\label{twodim}
Let $R$ be a $2$-dimensional complete local normal domain containing an algebraically closed field of characteristic zero. 
Then, the following conditions are equivalent.
\begin{enumerate}[\rm (1)]
\item $R$ is a quotient singularity associated with a finite group $G\subset\GL(2, k)$ (i.e., $R=S^G$ where $S=k[[x_1, x_2]]$). 
\item $R$ has a steady NCCR.
\item $R$ has a semi-steady NCCR.
\item $R$ has an NCCR. 
\item $R$ is of finite CM representation type, that is, $R$ has only finitely many non-isomorphic indecomposable MCM $R$-modules.
\end{enumerate}
When this is the case, modules giving NCCRs of $R$ are additive generators of $\CM R$.
\end{proposition}

\begin{proof}
(2)$\Rightarrow$(3)$\Rightarrow$(4) is clear. Therefore the assertion follows from \cite[Proposition~4.2]{IN}. 
\end{proof}

\section{\bf NCCRs arising from dimer models}
\label{NCCR_via_dimer}

In this section, we present several results concerning dimer models. 
In particular, we will show that a splitting non-commutative crepant resolution of a $3$-dimensional Gorenstein toric singularity 
is obtained from a consistent dimer model. 
For more results regarding dimer models, we refer to \cite{Boc4} and references quoted in this section. 

\subsection{Preliminaries on toric singularities}
\label{toric_pre}
We start this subsection with recalling some basic facts concerning toric singularities. For more details, see e.g., \cite{BG2,CLS}. 

Let $\sfN\cong\ZZ^d$ be a lattice, and $\sfM\coloneqq\Hom_\ZZ(\sfN, \ZZ)$ be the dual lattice of $\sfN$. 
Let $\sfN_\RR\coloneqq\sfN\otimes_\ZZ\RR$ and $\sfM_\RR\coloneqq\sfM\otimes_\ZZ\RR$. 
We denote an inner product by $\langle\;,\;\rangle:\sfM_\RR\times\sfN_\RR\rightarrow\RR$. 
In addition, let 
\[
\sigma\coloneqq\mathrm{Cone}(v_1, \cdots, v_n)=\RR_{\ge 0}v_1+\cdots +\RR_{\ge 0}v_n
\subset\sfN_\RR 
\]
be a strongly convex rational polyhedral cone generated by $v_1, \cdots, v_n\in\ZZ^d$. 
Suppose that this system of generators is minimal. 
For each generator, we define the linear form $\lambda_i(-)\coloneqq\langle-, v_i\rangle$, 
and denote $\lambda(-)\coloneqq(\lambda_1(-),\cdots,\lambda_n(-))$. 
We consider the dual cone $\sigma^\vee$: 
\[
\sigma^\vee\coloneqq\{x\in\sfM_\RR \mid \langle x,y\rangle\ge0 \text{ for all } y\in\sigma \}. 
\]
Then, we consider the $\fkm$-adic completion of a \emph{toric singularity} 
\[
R\coloneqq k[[\sigma^\vee\cap\sfM]]=k[[t_1^{a_1}\cdots t_d^{a_d}\mid (a_1, \cdots, a_d)\in\sigma^\vee\cap\sfM]], 
\] 
where $\fkm$ is the irrelevant maximal ideal. In our setting, $R$ is a $d$-dimensional CM normal domain, 
and it is known that $R$ is Gorenstein if and only if there exists $x\in\sigma^\vee\cap\ZZ^d$ such that $\lambda_i(x)=1$ for all $i=1, \cdots, n$ 
(see e.g., \cite[Theorem~6.33]{BG2}). 

For each $\mathbf{u}=(u_1, \cdots, u_n)\in\RR^n$, we define 
\[
\mathbb{T}(\mathbf{u})\coloneqq\{x\in\sfM\cong\ZZ^d \mid (\lambda_1(x), \cdots, \lambda_n(x))\ge(u_1, \cdots, u_n)\}. 
\]
Then, we define the divisorial ideal $T(\mathbf{u})$ generated by all monomials whose exponent vector is in $\mathbb{T}(\mathbf{u})$. 
Clearly, we have that $T(\mathbf{u})=T(\ulcorner \mathbf{u}\urcorner)$ 
where $\ulcorner \mathbf{u}\urcorner=(\ulcorner u_1\urcorner, \cdots, \ulcorner u_n\urcorner)$, 
thus we will assume $\mathbf{u}\in\ZZ^n$ in the rest of this paper. 
In general, a divisorial ideal of $R$ takes this form. 
In addition, for $\mathbf{u}, \mathbf{u}^\prime\in\ZZ^n$, $T(\mathbf{u})\cong T(\mathbf{u}^\prime)$ as an $R$-module 
if and only if there exists $y\in\sfM$ such that $u_i=u_i^\prime+\lambda_i(y)$ for all $i=1, \cdots, n$ (see \cite[Corollary~4.56]{BG2}). 
Thus, we have the exact sequence: 
\[
0\rightarrow\ZZ^d\xrightarrow{\lambda(-)}\ZZ^n\rightarrow\Cl(R)\rightarrow0, 
\]
we especially have the following. 

\begin{lemma}
\label{cl_toric}
The class group $\Cl(R)$ is isomorphic to $\ZZ^n/\lambda(\ZZ^d)$. 
In particular, the rank of the free part of $\Cl(R)$ is $n-d$. 
\end{lemma}

In this paper, we will investigate $3$-dimensional Gorensitein toric singularities, 
thus we can take the hyperplane $z=1$ so that generators $v_1,\cdots,v_n$ lie on this (i.e., the third coordinate of $v_i$ is $1$). 
Hence, we have the lattice polygon $\Delta\subset\RR^2$ on this hyperplane. 
Conversely, for a given lattice polygon $\Delta$ in $\RR^2$, we define the cone $\sigma_\Delta\subset\RR^3$ whose section on the hyperplane $z=1$ is $\Delta$. 
Then, the toric singularity $R=k[[\sigma_\Delta^\vee\cap\ZZ^3]]$ associated with such a cone is Gorenstein in dimension three. 
In the rest of this paper, we call $R$ obtained by the above manner the \emph{toric singularity associated with $\Delta$}, and call $\Delta$ the \emph{toric diagram} of $R$. 
We note that unimodular transformations of $\Delta$ in $\RR^2$ do not change the associated toric singularity up to isomorphism, 
thus we will discuss toric diagrams up to unimodular transformations.

\subsection{Dimer models and quivers with potentials} 
\label{subsec_dimer}

A \emph{dimer model} (or \emph{brane tiling}) is a polygonal cell decomposition of the real two-torus $\sfT\coloneqq\RR^2/\ZZ^2$, 
whose nodes and edges form a finite bipartite graph. 
Therefore, we color each node either black or white, and each edge connects a black node to a white node. 
For a dimer model $\Gamma$, we denote the set of nodes (resp. edges, faces) of $\Gamma$ by $\Gamma_0$ (resp. $\Gamma_1$, $\Gamma_2$). 
We also obtain the bipartite graph $\widetilde{\Gamma}$ on $\RR^2$ induced via the universal cover $\RR^2\rightarrow\sfT$, 
hence we call $\widetilde{\Gamma}$ the universal cover of a dimer model $\Gamma$. 
For example, the left hand side of Figure~\ref{ex_quiver4a} is a dimer model where the outer frame is the fundamental domain of the torus $\sfT$, 
and this is a regular dimer model. 

As the dual of a dimer model $\Gamma$, we define the quiver $Q_\Gamma$ associated with $\Gamma$. 
Namely, we assign a vertex dual to each face in $\Gamma_2$, an arrow dual to each edge in $\Gamma_1$.  
The orientation of arrows is determined so that the white node is on the right of the arrow. 
For example, the right hand side of Figure~\ref{ex_quiver4a} is the quiver obtained from the dimer model on the left. 
(Note that common numbers are identified in this figure.) 
Sometimes we simply denote the quiver $Q_\Gamma$ by $Q$. 
We denote the set of vertices by $Q_0$ and the set of arrows by $Q_1$. 
We consider the set of oriented faces $Q_F$ as the dual of nodes on a dimer model $\Gamma$. 
The orientation of faces is determined by its boundary, that is, faces dual to white (resp. black) nodes are oriented clockwise (resp. anti-clockwise). 
Therefore, we decompose the set of faces as $Q_F=Q^+_F\sqcup Q^-_F$ where $Q^+_F$, $Q^-_F$ denote the set of faces oriented clockwise and
that of faces oriented anti-clockwise respectively. 

\begin{figure}[H]
\begin{center}
\begin{tikzpicture}
\node (DM) at (0,0) 
{\scalebox{0.65}{
\begin{tikzpicture}
\node (P1) at (1,1){$$}; \node (P2) at (3,1){$$}; \node (P3) at
(3,3){$$}; \node (P4) at (1,3){$$};
\draw[line width=0.05cm]  (P1)--(P2)--(P3)--(P4)--(P1);\draw[line width=0.05cm] (-0.5,1)--(P1)--(1,-0.5); \draw[line width=0.05cm]  (4.5,1)--(P2)--(3,-0.5);
\draw[line width=0.05cm]  (-0.5,3)--(P4)--(1,4.5);\draw[line width=0.05cm]  (3,4.5)--(P3)--(4.5,3);
\filldraw  [ultra thick, fill=black] (1,1) circle [radius=0.16] ;\filldraw  [ultra thick, fill=black] (3,3) circle [radius=0.16] ;
\draw  [ultra thick,fill=white] (3,1) circle [radius=0.16] ;\draw  [ultra thick, fill=white] (1,3) circle [radius=0.16] ;

\draw[line width=0.05cm]  (0,0) rectangle (4,4);
\end{tikzpicture}
} }; 

\node (QV) at (5,0) 
{\scalebox{0.65}{
\begin{tikzpicture}

\node (Q1) at (2,2){$0$};\node (Q2a) at (0,2){$1$}; \node(Q2b) at (4,2){$1$};\node (Q3a) at (0,0){$2$};
\node(Q3c) at (4,4){$2$};\node(Q3b) at (4,0){$2$};\node(Q3d) at (0,4){$2$};\node (Q4a) at (2,0){$3$};
\node (Q4b) at (2,4){$3$};

\draw[lightgray, line width=0.05cm]  (P1)--(P2)--(P3)--(P4)--(P1);\draw[lightgray, line width=0.05cm] (-0.5,1)--(P1)--(1,-0.5); 
\draw[lightgray, line width=0.05cm]  (4.5,1)--(P2)--(3,-0.5);
\draw[lightgray, line width=0.05cm]  (-0.5,3)--(P4)--(1,4.5);\draw[lightgray, line width=0.05cm]  (3,4.5)--(P3)--(4.5,3);
\filldraw  [ultra thick, draw=lightgray, fill=lightgray] (1,1) circle [radius=0.16] ;\filldraw  [ultra thick, draw=lightgray, fill=lightgray] (3,3) circle [radius=0.16] ;
\draw  [ultra thick, draw=lightgray,fill=white] (3,1) circle [radius=0.16] ;\draw  [ultra thick, draw=lightgray,fill=white] (1,3) circle [radius=0.16] ;

\draw[->, line width=0.064cm] (Q1)--(Q2a);\draw[->, line width=0.064cm] (Q2a)--(Q3a);\draw[->, line width=0.064cm] (Q3a)--(Q4a);
\draw[->, line width=0.064cm] (Q4a)--(Q1);\draw[->, line width=0.064cm] (Q2a)--(Q3d);\draw[->, line width=0.064cm] (Q3d)--(Q4b);
\draw[->, line width=0.064cm] (Q4b)--(Q1);\draw[->, line width=0.064cm] (Q1)--(Q2b);\draw[->, line width=0.064cm] (Q2b)--(Q3b);
\draw[->, line width=0.064cm] (Q3b)--(Q4a);\draw[->, line width=0.064cm] (Q2b)--(Q3c);\draw[->, line width=0.064cm] (Q3c)--(Q4b);
\end{tikzpicture}
} }; 
\end{tikzpicture}
\caption{Dimer model and the associated quiver}
\label{ex_quiver4a}
\end{center}
\end{figure}
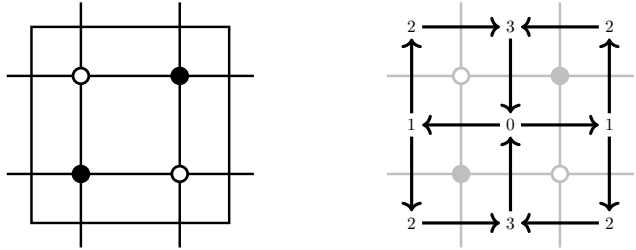

We define the maps $h, t:Q_1\rightarrow Q_0$ sending an arrow $a\in Q_1$ to the head of $a$ and the tail of $a$ respectively. 
A \emph{nontrivial path} is a finite sequence of arrows $a=a_1\cdots a_r$ with $h(a_\ell)=t(a_{\ell+1})$ for $\ell=1, \cdots r-1$. 
We define the \emph{length of path} $a=a_1\cdots a_r$ as $r \,(\ge 1)$, and denote by $Q_r$ the set of paths of length $r$. 
We consider each vertex $i\in Q_0$ as a trivial path $e_i$ of length $0$ where $h(e_i)=t(e_i)=i$. 
We extend the maps $h, t$ to the maps on paths, that is, $t(a)=t(a_1), h(a)=h(a_r)$ for a path $a=a_1\cdots a_r$. 
We say that a path $a$ is a \emph{cycle} if $h(a)=t(a)$. 
In addition, we denote the opposite quiver of $Q$ by $Q^{\rm op}$. 
That is, $Q^{\rm op}$ is obtained from $Q$ by reversing all arrows. 
Hence, we obtain the opposite quiver associated with the original dimer model by replacing white nodes by black nodes and vice versa. 

For a quiver $Q$, the \emph{complete path algebra} is defined as 
\[
\widehat{kQ}\coloneqq\prod_{r\ge 0}kQ_r 
\]
where $kQ_r$ is the vector space with a basis $Q_r$. The multiplication is defined as 
$a\cdot b=ab$ (resp. $a\cdot b=0$) if $h(a)=t(b)$ (resp. $h(a)\neq t(b)$) for paths $a, b$. We extend this multiplication linearly. 
Further, we set $\fkm_Q\coloneqq \prod_{r\ge 1}kQ_r$. 
For a subset $U\subseteq\widehat{kQ}$, we define the \emph{$\fkm_Q$-adic closure} of $U$ as 
$\overline{U}\coloneqq \bigcap_{n\ge 0}(U+\fkm_Q^n)$. 

Next, we define a potential. 
We denote by $[kQ, kQ]$ the $k$-vector space generated by all commutators in $kQ$ 
and set the vector space $kQ_{\mathrm{cyc}}\coloneqq kQ/[kQ, kQ]$, thus $kQ_{\mathrm{cyc}}$ has a basis consists of cycles in $Q$. 
We denote by $(kQ_{\mathrm{cyc}})_r$ the subspace of $kQ_{\mathrm{cyc}}$ spanned by cycles of length at least $r$. 
We call an element $W\in (kQ_{\mathrm{cyc}})_2$ a \emph{potential}, 
and call a pair $(Q,W)$ a \emph{quiver with potential} (= \emph{QP}). 
 
For each face $f\in Q_F$, we associate the \emph{small cycle} $\omega_f\in (kQ_{\mathrm{cyc}})_2$ obtained as the product of arrows 
around the boundary of $f$. 
For the quiver $Q$ associated with a dimer model, we define the potential $W_Q$ as 
\[
W_Q\coloneqq \sum_{f\in Q^+_F}\omega_f-\sum_{f\in Q^-_F}\omega_f. 
\]

For each face $f\in Q_F$, we choose an arrow $a\in \omega_f$ and consider $h(a)$ as the starting point of the small cycle $\omega_f$. 
Then, we may write $e_{h(a)}\omega_fe_{h(a)}\coloneqq a_1\cdots a_ra$ with some path $a_1\cdots a_r$. 
We define the partial derivative of $\omega_f$ with respect to $a$ by $\partial \omega_f/\partial a\coloneqq a_1\cdots a_r$. 
Extending this derivative linearly, we also define $\partial W_Q/\partial a$ for any $a\in Q_1$.
Then, we consider the closure of the two-sided ideal $J(W_Q)\coloneqq \langle\overline{\partial W_Q/\partial a\,|\,a\in Q_1}\rangle$. 
We define the \emph{complete Jacobian algebra} of a dimer model as 
\[
\calP(Q, W_Q)\coloneqq \widehat{kQ}/J(W_Q).
\]

We say that a node of a dimer model is \emph{bivalent} if the number of edges incident to that node is two. 
In the rest, we assume that our dimer model has no bivalent nodes. 
If there are bivalent nodes, we remove them as shown in \cite[Figure~5.1]{IU1}, because this operation does not change the Jacobian algebra up to isomorphism. 

\subsection{Consistency condition and NCCRs}
\label{sec_consist}

In this subsection, we impose the extra condition so-called ``consistency condition" on dimer models. 
Under this assumption, a dimer model gives an NCCR of a $3$-dimensional Gorenstein toric singularity (see Theorem~\ref{NCCR1}). 

We need the notion of zigzag paths to introduce the consistency condition. 

\begin{definition}
We say that a path on a dimer model $\Gamma$ is a \emph{zigzag path} if 
it makes a maximum turn to the right on a white node and a maximal turn to the left on a black node.  
\end{definition}

We also consider the lift of a zigzag path to the universal cover $\widetilde{\Gamma}$. 
(Note that a zigzag path on the universal cover is either periodic or infinite in both directions.) 
For example, zigzag paths of the dimer model given in Figure~\ref{ex_quiver4a} are shown in Figure~\ref{zz_4a}. 
By using this notion, we introduce the consistency condition. 
In the literature, there are several conditions that are equivalent to the following definition (see \cite{Boc1,IU1}). 

\begin{definition}[{see \cite[Definition~3.5]{IU1}}]
\label{def_consistent}
We say that a dimer model is \emph{consistent} if 
\begin{enumerate}[\rm(1)]
\item there is no homologically trivial zigzag path, 
\item no zigzag path on the universal cover has a self-intersection, 
\item no pair of zigzag paths on the universal cover intersect each other in the same direction more than once. 
That is, if a pair of zigzag paths $(z,w)$ on the universal cover has two intersections 
$a_1$, $a_2$ and $z$ points from $a_1$ to $a_2$, then $w$ point from $a_2$ to $a_1$. 
\end{enumerate}
Here, we remark that two zigzag paths are said to \emph{intersect} if they share an edge (not a node). 
\end{definition} 

We also introduce isoradial dimer models which are stronger than consistent ones. 

\begin{definition}[{\cite[Theorem~5.1]{KS},  see also \cite{Duf,Mer}}] 
\label{def_isoradial}
We say that a dimer model $\Gamma$ is \emph{isoradial} (or \emph{geometrically consistent}) if 
\begin{enumerate}[\rm(1)]
\item every zigzag path is a simple closed curve, 
\item any pair of zigzag paths on the universal cover share at most one edge. 
\end{enumerate}
\end{definition}

By Figure~\ref{zz_4a} below, we see that the dimer model given in Figure~\ref{ex_quiver4a} is isoradial, thus it is consistent in particular. 
In general, we can easily see that regular dimer models are isoradial. 

\medskip

\begin{figure}[H]
\begin{center}
{\scalebox{0.9}{
\begin{tikzpicture}
\node (ZZ1) at (0,0) 
{\scalebox{0.5}{
\begin{tikzpicture}
\coordinate (P1) at (1,1); \coordinate (P2) at (3,1); \coordinate (P3) at (3,3); \coordinate (P4) at (1,3);
\draw[thick]  (0,0) rectangle (4,4);
\draw[line width=0.05cm]  (P1)--(P2)--(P3)--(P4)--(P1);\draw[line width=0.05cm] (-0.5,1)--(P1)--(1,-0.5); \draw[line width=0.05cm]  (4.5,1)--(P2)--(3,-0.5);
\draw[line width=0.05cm]  (-0.5,3)--(P4)--(1,4.5);\draw[line width=0.05cm]  (3,4.5)--(P3)--(4.5,3);

\filldraw  [ultra thick, fill=black] (1,1) circle [radius=0.18] ;\filldraw  [ultra thick, fill=black] (3,3) circle [radius=0.18] ;
\draw  [ultra thick,fill=white] (3,1) circle [radius=0.18] ;\draw  [ultra thick,fill=white] (1,3) circle [radius=0.18] ;
\draw[->, line width=0.18cm, rounded corners, color=red] (-0.5,1)--(P1)--(P4)--(P3)--(3,4.5); 
\draw[->, line width=0.18cm, rounded corners, color=red] (3,-0.5)--(P2)--(4.5,1); 
\end{tikzpicture} }}; 

\node (ZZ2) at (3.3,0) 
{\scalebox{0.5}{
\begin{tikzpicture}
\coordinate (P1) at (1,1); \coordinate (P2) at (3,1); \coordinate (P3) at (3,3); \coordinate (P4) at (1,3);
\draw[thick]  (0,0) rectangle (4,4);
\draw[line width=0.05cm]  (P1)--(P2)--(P3)--(P4)--(P1);\draw[line width=0.05cm] (-0.5,1)--(P1)--(1,-0.5); \draw[line width=0.05cm]  (4.5,1)--(P2)--(3,-0.5);
\draw[line width=0.05cm]  (-0.5,3)--(P4)--(1,4.5);\draw[line width=0.05cm]  (3,4.5)--(P3)--(4.5,3);
\filldraw  [ultra thick, fill=black] (1,1) circle [radius=0.18] ;\filldraw  [ultra thick, fill=black] (3,3) circle [radius=0.18] ;
\draw  [ultra thick,fill=white] (3,1) circle [radius=0.18] ;\draw  [ultra thick,fill=white] (1,3) circle [radius=0.18] ;
\draw[->, line width=0.18cm, rounded corners, color=red] (1,-0.5)--(P1)--(-0.5,1); 
\draw[->, line width=0.18cm, rounded corners, color=red] (4.5,1)--(P2)--(P3)--(P4)--(1,4.5);
\end{tikzpicture} }} ;  

\node (ZZ3) at (6.6,0) 
{\scalebox{0.5}{
\begin{tikzpicture}
\coordinate (P1) at (1,1); \coordinate (P2) at (3,1); \coordinate (P3) at (3,3); \coordinate (P4) at (1,3);
\draw[thick]  (0,0) rectangle (4,4);
\draw[line width=0.05cm]  (P1)--(P2)--(P3)--(P4)--(P1);\draw[line width=0.05cm] (-0.5,1)--(P1)--(1,-0.5); \draw[line width=0.05cm]  (4.5,1)--(P2)--(3,-0.5);
\draw[line width=0.05cm]  (-0.5,3)--(P4)--(1,4.5);\draw[line width=0.05cm]  (3,4.5)--(P3)--(4.5,3);
\filldraw  [ultra thick, fill=black] (1,1) circle [radius=0.18] ;\filldraw  [ultra thick, fill=black] (3,3) circle [radius=0.18] ;
\draw  [ultra thick,fill=white] (3,1) circle [radius=0.18] ;\draw  [ultra thick,fill=white] (1,3) circle [radius=0.18] ;
\draw[->, line width=0.18cm, rounded corners, color=red] (1,4.5)--(P4)--(-0.5,3); 
\draw[->, line width=0.18cm, rounded corners, color=red] (4.5,3)--(P3)--(P2)--(P1)--(1,-0.5);
\end{tikzpicture} }}; 

\node (ZZ4) at (9.9,0) 
{\scalebox{0.5}{
\begin{tikzpicture}
\coordinate (P1) at (1,1); \coordinate (P2) at (3,1); \coordinate (P3) at (3,3); \coordinate (P4) at (1,3);
\draw[thick]  (0,0) rectangle (4,4);
\draw[line width=0.05cm]  (P1)--(P2)--(P3)--(P4)--(P1);\draw[line width=0.05cm] (-0.5,1)--(P1)--(1,-0.5); \draw[line width=0.05cm]  (4.5,1)--(P2)--(3,-0.5);
\draw[line width=0.05cm]  (-0.5,3)--(P4)--(1,4.5);\draw[line width=0.05cm]  (3,4.5)--(P3)--(4.5,3);
\filldraw  [ultra thick, fill=black] (1,1) circle [radius=0.18] ;\filldraw  [ultra thick, fill=black] (3,3) circle [radius=0.18] ;
\draw  [ultra thick,fill=white] (3,1) circle [radius=0.18] ;\draw  [ultra thick,fill=white] (1,3) circle [radius=0.18] ; 
\draw[->, line width=0.18cm, rounded corners, color=red] (-0.5,3)--(P4)--(P1)--(P2)--(3,-0.5) ; 
\draw[->, line width=0.18cm, rounded corners, color=red] (3,4.5)--(P3)--(4.5,3) ;
\end{tikzpicture} }} ;

\end{tikzpicture}
}}
\caption{Examples of zigzag paths}
\label{zz_4a}
\end{center}
\end{figure}
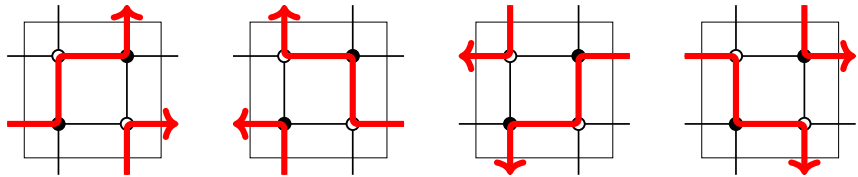

Next, we introduce the notion of perfect matchings. 
In general, every dimer model does not necessarily have a perfect matching. 
If a dimer model is consistent, then it has a perfect matching and every edge is contained in some perfect matchings (see e.g., \cite[Proposition~8.1]{IU2}). 

\begin{definition}
A \emph{perfect matching} (or \emph{dimer configuration}) on a dimer model $\Gamma$ is a subset $\sfP$ of $\Gamma_1$ such that each node is 
the end point of precisely one edge in $\sfP$. 
A \emph{perfect matching} on $\widetilde{\Gamma}$ is also defined naturally via the universal cover $\RR^2\rightarrow\sfT$.  
\end{definition}

For each edge contained in a perfect matching on $\Gamma$, we give the orientation from a white node to a black node. 
We fix a perfect matching $\sfP_0$. 
For any perfect matching $\sfP$, the difference of two perfect matchings $\sfP-\sfP_0$ forms a $1$-cycle, 
and hence we consider such a $1$-cycle as an element in the homology group $\rmH_1(\sfT)\cong\ZZ^2$. 
Then, we obtain finitely many elements in $\ZZ^2$ corresponding to perfect matchings on $\Gamma$, 
and define the lattice polygon $\Delta$ as the convex hull of them. 
We call $\Delta$ the \emph{perfect matching polygon} (or \emph{characteristic polygon}) of $\Gamma$. 
Although this lattice polygon depends on a choice of a fixed perfect matching, it is determined up to translations. 
We say that a perfect matching $\sfP$ is \emph{extremal} if the lattice point corresponding to the $1$-cycle $\sfP-\sfP_0$ lies at a vertex of $\Delta$. 
If a dimer model is consistent, then there exists a unique extremal perfect matching corresponding to a vertex of $\Delta$ (see e.g., \cite[Corollary~4.27 ]{Bro}, \cite[Proposition~9.2]{IU2}). 
Thus, we can give a cyclic order to extremal perfect matchings along the corresponding vertices of $\Delta$ in the anti-clockwise direction. 
In addition, we say that two extremal perfect matchings are \emph{adjacent} if they are adjacent with respect to a given cyclic order. 
For example, $\sfP_1,\cdots,\sfP_4$ shown in Figure~\ref{pm_4a} are extremal perfect matchings on the dimer model given in Figure~\ref{ex_quiver4a} 
corresponding to vertices $(1,0),(0,1),(-1,0),(0,-1)$ respectively, where $\sfP_0$ is a fixed perfect matching. 

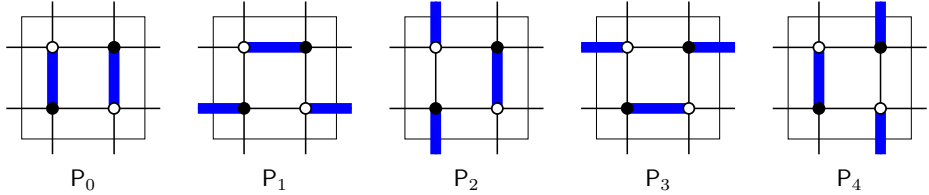
\begin{figure}[H]
\begin{center}
{\scalebox{0.9}{
\begin{tikzpicture} 
\node at (0,-1.5) {$\sfP_0$};
\node at (2.8,-1.5) {$\sfP_1$}; \node at (5.6,-1.5) {$\sfP_2$}; \node at (8.4,-1.5) {$\sfP_3$}; \node at (11.2,-1.5) {$\sfP_4$};

\node (PM0) at (0,0) 
{\scalebox{0.45}{
\begin{tikzpicture}
\node (P1) at (1,1){$$}; \node (P2) at (3,1){$$}; \node (P3) at (3,3){$$}; \node (P4) at (1,3){$$};
\draw[thick]  (0,0) rectangle (4,4);
\draw[line width=0.05cm]  (P1)--(P2)--(P3)--(P4)--(P1);\draw[line width=0.05cm] (-0.5,1)--(P1)--(1,-0.5); \draw[line width=0.05cm]  (4.5,1)--(P2)--(3,-0.5);
\draw[line width=0.05cm]  (-0.5,3)--(P4)--(1,4.5);\draw[line width=0.05cm]  (3,4.5)--(P3)--(4.5,3);
\draw[line width=0.35cm,color=blue] (P1)--(P4);\draw[line width=0.35cm,color=blue] (P2)--(P3);
\filldraw  [ultra thick, fill=black] (1,1) circle [radius=0.18] ;\filldraw  [ultra thick, fill=black] (3,3) circle [radius=0.18] ;
\draw  [ultra thick,fill=white] (3,1) circle [radius=0.18] ;\draw  [ultra thick,fill=white] (1,3) circle [radius=0.18] ;
\end{tikzpicture} }}; 

\node (PM1) at (2.8,0) 
{\scalebox{0.45}{
\begin{tikzpicture}
\node (P1) at (1,1){$$}; \node (P2) at (3,1){$$}; \node (P3) at (3,3){$$}; \node (P4) at (1,3){$$};
\draw[thick]  (0,0) rectangle (4,4);
\draw[line width=0.05cm]  (P1)--(P2)--(P3)--(P4)--(P1);\draw[line width=0.05cm] (-0.5,1)--(P1)--(1,-0.5); \draw[line width=0.05cm]  (4.5,1)--(P2)--(3,-0.5);
\draw[line width=0.05cm]  (-0.5,3)--(P4)--(1,4.5);\draw[line width=0.05cm]  (3,4.5)--(P3)--(4.5,3);
\draw[line width=0.35cm,color=blue] (P3)--(P4);\draw[line width=0.35cm,color=blue] (P1)--(-0.5,1);\draw[line width=0.35cm,color=blue] (P2)--(4.5,1);
\filldraw  [ultra thick, fill=black] (1,1) circle [radius=0.18] ;\filldraw  [ultra thick, fill=black] (3,3) circle [radius=0.18] ;
\draw  [ultra thick,fill=white] (3,1) circle [radius=0.18] ;\draw  [ultra thick,fill=white] (1,3) circle [radius=0.18] ;
\end{tikzpicture} }}; 

\node (PM2) at (5.6,0) 
{\scalebox{0.45}{
\begin{tikzpicture}
\node (P1) at (1,1){$$}; \node (P2) at (3,1){$$}; \node (P3) at (3,3){$$}; \node (P4) at (1,3){$$};
\draw[thick]  (0,0) rectangle (4,4);
\draw[line width=0.05cm]  (P1)--(P2)--(P3)--(P4)--(P1);\draw[line width=0.05cm] (-0.5,1)--(P1)--(1,-0.5); \draw[line width=0.05cm]  (4.5,1)--(P2)--(3,-0.5);
\draw[line width=0.05cm]  (-0.5,3)--(P4)--(1,4.5);\draw[line width=0.05cm]  (3,4.5)--(P3)--(4.5,3);
\draw[line width=0.35cm,color=blue] (P3)--(P2);\draw[line width=0.35cm,color=blue] (P4)--(1,4.5);\draw[line width=0.35cm,color=blue] (P1)--(1,-0.5);
\filldraw  [ultra thick, fill=black] (1,1) circle [radius=0.18] ;\filldraw  [ultra thick, fill=black] (3,3) circle [radius=0.18] ;
\draw  [ultra thick,fill=white] (3,1) circle [radius=0.18] ;\draw  [ultra thick,fill=white] (1,3) circle [radius=0.18] ;
\end{tikzpicture} }} ;  

\node (PM3) at (8.4,0) 
{\scalebox{0.45}{
\begin{tikzpicture}
\node (P1) at (1,1){$$}; \node (P2) at (3,1){$$}; \node (P3) at (3,3){$$}; \node (P4) at (1,3){$$};
\draw[thick]  (0,0) rectangle (4,4);
\draw[line width=0.05cm]  (P1)--(P2)--(P3)--(P4)--(P1);\draw[line width=0.05cm] (-0.5,1)--(P1)--(1,-0.5); \draw[line width=0.05cm]  (4.5,1)--(P2)--(3,-0.5);
\draw[line width=0.05cm]  (-0.5,3)--(P4)--(1,4.5);\draw[line width=0.05cm]  (3,4.5)--(P3)--(4.5,3);
\draw[line width=0.35cm,color=blue] (P2)--(P1);\draw[line width=0.35cm,color=blue] (P4)--(-0.5,3);\draw[line width=0.35cm,color=blue] (P3)--(4.5,3);
\filldraw  [ultra thick, fill=black] (1,1) circle [radius=0.18] ;\filldraw  [ultra thick, fill=black] (3,3) circle [radius=0.18] ;
\draw  [ultra thick,fill=white] (3,1) circle [radius=0.18] ;\draw  [ultra thick,fill=white] (1,3) circle [radius=0.18] ;
\end{tikzpicture} }}; 

\node (PM4) at (11.2,0) 
{\scalebox{0.45}{
\begin{tikzpicture}
\node (P1) at (1,1){$$}; \node (P2) at (3,1){$$}; \node (P3) at (3,3){$$}; \node (P4) at (1,3){$$};
\draw[thick]  (0,0) rectangle (4,4);
\draw[line width=0.05cm]  (P1)--(P2)--(P3)--(P4)--(P1);\draw[line width=0.05cm] (-0.5,1)--(P1)--(1,-0.5); \draw[line width=0.05cm]  (4.5,1)--(P2)--(3,-0.5);
\draw[line width=0.05cm]  (-0.5,3)--(P4)--(1,4.5);\draw[line width=0.05cm]  (3,4.5)--(P3)--(4.5,3);
\draw[line width=0.35cm,color=blue] (P4)--(P1);\draw[line width=0.35cm,color=blue] (P3)--(3,4.5);\draw[line width=0.35cm,color=blue] (P2)--(3,-0.5);
\filldraw  [ultra thick, fill=black] (1,1) circle [radius=0.18] ;\filldraw  [ultra thick, fill=black] (3,3) circle [radius=0.18] ;
\draw  [ultra thick,fill=white] (3,1) circle [radius=0.18] ;\draw  [ultra thick,fill=white] (1,3) circle [radius=0.18] ;
\end{tikzpicture} }} ;

\end{tikzpicture}
}}
\caption{Extremal perfect matchings}
\label{pm_4a}
\end{center}
\end{figure}

Then, we discuss a relationship between the perfect matching polygon and zigzag paths. 
Since we can consider a zigzag path $z$ as a $1$-cycle on $\sfT$, 
it determines the homology class $[z]\in\rmH_1(\sfT)\cong\ZZ^2$. We call this element $[z]\in\ZZ^2$ the \emph{slope} of $z$. 
If a dimer model is consistent, a zigzag path does not have a self-intersection, and hence the slope of each zigzag path is a primitive element. 
Then, we have the following correspondence.  

\begin{proposition}[{see e.g., \cite[Section~9]{IU2},\cite[Corollary~2.9]{Boc3}}] 
\label{zigzag_sidepolygon}
There exists a one to one correspondence between the set of slopes of zigzag paths on a consistent dimer model and 
the set of primitive side segments of the perfect matching polygon. 
Precisely, let $v,v^\prime\in\ZZ^2$ be end points of a primitive side segment, then there exists a zigzag path whose slope 
coincides with $v-v^\prime$. 

Moreover, zigzag paths having the same slope arise as the difference of two extremal perfect matchings that are adjacent.
\end{proposition}

Furthermore, by this correspondence, we can also give a cyclic order to the set of slopes of zigzag paths. 
Thus, we say that a pair of zigzag paths have \emph{adjacent slopes} if their slopes are adjacent with respect to a given cyclic order. 
This cyclic order is essential in the definition of  properly ordered dimer models written below. 
It is known that a dimer model is properly ordered if and only if it is consistent (see \cite[Proposition~4.4]{IU1}). 

\begin{definition}[{see \cite[Section~3.1]{Gul}}] 
\label{def_properly}
We say that a dimer model is \emph{properly ordered} if 
\begin{enumerate}[\rm(1)]
\item there is no homologically trivial zigzag path, 
\item no zigzag path on the universal cover has a self-intersection, 
\item no pair of zigzag paths with the same slope have a common node, 
\item for any node on the dimer model, the natural cyclic order on the set of zigzag paths touching that node 
coincides with the cyclic order determined by their slopes.
\end{enumerate}
\end{definition} 

We can also characterize isoradial dimer models in terms of slopes of zigzag paths. 

\begin{proposition}[{see \cite[Propostion~3.12]{Bro}}]
\label{slope_isoradial}
A dimer model is isoradial if and only if the following conditions hold.
\begin{enumerate}[\rm(1)]
\item No zigzag path on the universal cover has a self-intersection,
\item Let $z$ and $z^\prime$ be zigzag paths on the universal cover. If $[z], [z^\prime]\in\rmH_1(\sfT)$ are linearly independent, 
then they intersect in precisely one arrow. 
\item Let $z$ and $z^\prime$ be zigzag paths on the universal cover. If $[z], [z^\prime]\in\rmH_1(\sfT)$ are linearly dependent, 
then they do not intersect. 
\end{enumerate}
\end{proposition}

\medskip

In the rest of this subsection, we present a construction of modules giving NCCRs of $3$-dimensional Gorenstein toric singularities. 

By the dual point of view, we consider a perfect matching as a function on $Q_1$. 
Namely, for each arrow $a\in Q_1$ and each perfect matching $\sfP$, we define the perfect matching function: 
\begin{eqnarray}
\label{pm_function}
\sfP(a)=\left \{\begin{array}{ll}
1&\text{if the edge corresponding to $a$ is in $\sfP$}\\
0&\text{otherwise}. \\
\end{array} \right. 
\end{eqnarray}
When we consider the oppositely directed arrow $a^*\in Q^{\rm op}$ for $a\in Q_1$, we define $\sfP(a^*)=-\sfP(a)$. 

\medskip

Let $\Gamma$ be a consistent dimer model, whose perfect matching polygon is $\Delta$. 
We consider the $3$-dimensional Gorenstein toric singularity $R$ associated with $\Delta$. 
That is, the toric diagram of $R$ is the perfect matching polygon $\Delta$. 
Let $\sfP_1, \cdots, \sfP_n$ be the extremal perfect matchings on $\Gamma$ ordered cyclically.  
For $i, j\in Q_0$, let $a_{ij}$ be a path from $i$ to $j$ (i.e., $h(a_{ij})=j$ and $t(a_{ij})=i$). 
We define the divisorial ideal of $R$ associated with $a_{ij}$ as 
\[
T_{a_{ij}}\coloneqq T(\sfP_1(a_{ij}), \cdots, \sfP_n(a_{ij})). 
\]
This ideal depends on only the starting point $i$ and the ending point $j$, whereas a path is not unique. 
Namely, let $a_{ij}, b_{ij}$ be paths from $i$ to $j$, then we have that $T_{a_{ij}}\cong T_{b_{ij}}$ (see e.g., \cite[Lemma~3.7]{Nak}). 
Thus, we simply denote it by $T_{ij}$. 
Using this divisorial ideal, we obtain an NCCR of $R$ as follows. 

\begin{theorem}[{see e.g., \cite{Bro,IU2,Boc2}}]
\label{NCCR1}
Suppose that $(Q, W_Q)$ is the QP associated with a consistent dimer model $\Gamma$ and $\calP(Q, W_Q)$ is the complete Jacobian algebra. 
Let $R\coloneqq \rmZ(\calP(Q, W_Q))$ be the center of $\calP(Q, W_Q)$. 
Then, $R$ is a $3$-dimensional complete local Gorenstein toric singularity, whose toric diagram coincides with the perfect matching polygon of $\Gamma$. 
Furthermore, we have that 
\[
\calP(Q, W_Q)\cong\End_R(\bigoplus_{j\in Q_0}T_{ij}), 
\]
for each vertex $i\in Q_0$ and this is a splitting NCCR of $R$. 
\end{theorem}

\begin{remark}
\label{rem_NCCR}
Here, we give a few more remarks on Theorem~\ref{NCCR1}: 
\begin{enumerate}[\rm(a)]
\item Since $T^i\coloneqq\bigoplus_{j\in Q_0}T_{ij}$ contains $R\cong T_{ii}$ as a direct summand for any fixed vertex $i\in Q_0$, 
we have that $T_{ij}\in\CM R$ for any $i, j \in Q_0$. 
Furthermore, we see that $T^i$ is basic (i.e., $T_{ij}$'s are mutually non-isomorphic). 

\item An isomorphism in Theorem~\ref{NCCR1} can be established by sending each arrow $j\rightarrow k$ in $Q$ to 
an irreducible morphism $T_{ij}\rightarrow T_{ik}$ in $\End_R(T^i)$. 
Here, we say that a morphism $T_{ij}\rightarrow T_{ik}$ is \emph{irreducible} in $\End_R(T^i)$ 
if it does not factor through $T_{i\ell}$ with $\ell\neq j,k$. 
Evidently, irreducible morphisms from $T_{ij}$ to $T_{ik}$ generate $\Hom_R(T_{ij},T_{ik})$ as an $R$-module. 

\item Let $e_i$ be the idempotent corresponding to $i\in Q_0$. Then, $$T^i\cong\Hom_R(T_{ii},\bigoplus_{j\in Q_0}T_{ij})\cong e_i\calP(Q, W_Q).$$
Furthermore, since $T_{ij}^* \cong T_{ji}$, we have that 
\[
\calP(Q, W_Q)\cong\End_R(T^i)\cong\End_R((T^i)^*)\cong\calP(Q^{\rm op}, W_{Q^{\rm op}}). 
\]
\end{enumerate}
\end{remark}

In this manner, we obtain a $3$-dimensional complete local Gorenstein toric singularity $R$ and its splitting NCCR from a consistent dimer model. 
On the other hand, for every $3$-dimensional Gorenstein toric singularity $R$ associated with $\Delta$,  
there exists a consistent dimer model whose perfect matching polygon coincides with $\Delta$ (see \cite{Gul,IU2}). 
Thus, by combining these results, we have the following corollary. 
We remark that a consistent dimer model giving an NCCR of $R$ is not unique in general. 

\begin{corollary}
Every $3$-dimensional Gorenstein toric singularity admits a splitting NCCR which is constructed from a consistent dimer model. 
\end{corollary}


\section{\bf Semi-steady NCCRs arising from dimer models}
\label{sec_semisteady_dimer}

In the previous section, we saw that every $3$-dimensional complete local Gorenstein toric singularity admits NCCRs.  
In this section, we study splitting NCCRs arising from consistent dimer models that are semi-steady, 
and discuss a relationship with regular dimer models. 

First, we note a basic property of semi-steady NCCRs arising from consistent dimer models. 

\begin{lemma}
\label{lem_semi}
Let $R$ be a $3$-dimensional complete local Gorenstein toric singularity. 
If a consistent dimer model $\Gamma$ gives a semi-steady NCCR of $R$, 
then there exists a generator $M$ such that $\End_R(M)\cong\calP(Q_\Gamma,W_{Q_\Gamma})$ 
and $e_i\calP(Q_\Gamma,W_{Q_\Gamma})\cong M$ or $M^*$ for any $i\in Q_0$. 
In particular, for all $i\in Q_0$, $e_i\calP(Q_\Gamma,W_{Q_\Gamma})$ gives a semi-steady NCCR of $R$. 
\end{lemma}

\begin{proof}
By Theorem~\ref{NCCR1}, we have a basic splitting generator $M$ such that $\calP(Q_\Gamma,W_{Q_\Gamma})\cong\End_R(M)$, 
and there exists a one-to-one correspondence between direct summands in $M$ and vertices in $Q_\Gamma$. 
Thus, we may write $M=\bigoplus_{i\in(Q_\Gamma)_0}M_i$. 
Then, for each idempotent $e_i$ corresponding to a vertex $i\in(Q_\Gamma)_0$, 
we have that $e_i\calP(Q_\Gamma,W_{Q_\Gamma})\cong\Hom_R(M_i,M)$. 
By the definition of semi-steady module, we have that $\Hom_R(M_i,M)\in\add_RM$ or $\add_RM^*$ for any $i$. 
Since $M$ is basic, we have the assertion by the maximality of modules giving NCCRs (see \cite[Proposition~4.5]{IW2}). 
The last assertion follows from Lemma~\ref{basic_lem}(b). 
\end{proof}

Now, we state the main theorem in this paper. 

\begin{theorem}
\label{main_thm}
Let $R$ be a $3$-dimensional complete local Gorenstein toric singularity, $\Gamma_1,\cdots,\Gamma_n$ be consistent dimer models associated with $R$. 
Then, the following conditions are equivalent. 
\begin{enumerate}[\rm (1)]
\item $R$ is a toric singularity associated with a parallelogram (i.e., the toric diagram of $R$ is a parallelogram). 
\item There exists a consistent dimer model $\Gamma_i$ that is homotopy equivalent to a square dimer model. 
\item There exists an isoradial dimer model $\Gamma_i$ giving a semi-steady NCCR of $R$ that is not steady. 
\end{enumerate}
When this is the case, an isoradial dimer model $\Gamma$ gives a semi-steady NCCR of $R$ that is not steady if and only if 
$\Gamma$ is homotopy equivalent to a square dimer model. 
\end{theorem}

\begin{remark}
Even if $R$ is a toric singularity associated with a parallelogram, 
there exists a consistent dimer model that does not give a semi-steady NCCR of $R$ (see Example~\ref{ex_semisteady}). 
On the other hand, a consistent dimer model associated with a quotient singularity by a finite abelian group is unique 
(up to homotopy equivalence), and it is homotopy equivalent to a regular hexagonal dimer model, and gives a steady NCCR. (see Theorem~\ref{dimer}). 
\end{remark}


\begin{proof}[Proof of Theorem~\ref{main_thm}]

To show $(1)$$\Rightarrow$$(2)$, we construct a consistent dimer model whose perfect matching polygon coincides with the toric diagram of $R$. 
There are several methods for constructing it (see e.g., \cite{Gul,IU2}). 
To achieve our purpose, the operation in \cite{HV} is effective. 
In what follows, we will construct a consistent dimer model giving the parallelogram shown in Figure~\ref{ex_FIA} by using such an operation. 
(We can easily generalize this method for other parallelograms.) 

\begin{figure}[H]
\begin{center}
\begin{tabular}{c}
\begin{minipage}{0.38\hsize}
\begin{center}
{\scalebox{0.5}{
\begin{tikzpicture}
\draw [step=1,thin, gray] (-3,-3) grid (4,3);
\filldraw  [ultra thick, fill=black] (0,0) circle [radius=0.1] ;
\node at (-0.3,-0.3){{\huge$0$}};

\filldraw  [ultra thick, fill=black] (3,0) circle [radius=0.1] ; \filldraw  [ultra thick, fill=black] (0,2) circle [radius=0.1] ; 
\filldraw  [ultra thick, fill=black] (-2,0) circle [radius=0.1] ; \filldraw  [ultra thick, fill=black] (1,-2) circle [radius=0.1] ;
\filldraw  [ultra thick, fill=black] (2,-1) circle [radius=0.1] ; \filldraw  [ultra thick, fill=black] (-1,1) circle [radius=0.1] ;

\draw[line width=0.07cm]  (3,0)--(0,2)--(-2,0)--(1,-2)--(3,0) ;
\end{tikzpicture}
} }
\end{center}
\caption{}
\label{ex_FIA}
\end{minipage}

\begin{minipage}{0.38\hsize}
\begin{center}
{\scalebox{0.5}{
\begin{tikzpicture}
\draw [step=1,thin, gray] (-3,-3) grid (4,3);
\filldraw  [ultra thick, fill=black] (0,0) circle [radius=0.1] ;
\node at (-0.3,-0.3){{\huge$0$}};

\filldraw  [ultra thick, fill=black] (3,0) circle [radius=0.1] ; \filldraw  [ultra thick, fill=black] (0,2) circle [radius=0.1] ; 
\filldraw  [ultra thick, fill=black] (-2,0) circle [radius=0.1] ; \filldraw  [ultra thick, fill=black] (1,-2) circle [radius=0.1] ;
\filldraw  [ultra thick, fill=black] (2,-1) circle [radius=0.1] ; \filldraw  [ultra thick, fill=black] (-1,1) circle [radius=0.1] ;

\draw[->, dashed, line width=0.1cm,color=green] (2.5,-0.5)--++(-45:1); 
\draw[->, dashed, line width=0.1cm,color=red] (1.5,-1.5)--++(-45:1);
\draw[->, line width=0.1cm,color=red] (-0.5,1.5)--++(135:1); 
\draw[->, line width=0.1cm,color=green] (-1.5,0.5)--++(135:1);
\draw[->, line width=0.1cm,color=blue] (1.5,1)--++(0.6,0.9); 
\draw[->, dashed, line width=0.1cm,color=blue] (-0.5,-1)--++(-0.6,-0.9);
\draw[line width=0.07cm]  (3,0)--(0,2)--(-2,0)--(1,-2)--(3,0) ;
\end{tikzpicture}
} }
\end{center}
\caption{}
\label{ex_FIA_orthogonal}
\end{minipage}

\end{tabular}
\end{center}
\end{figure}


\noindent \textbf{Hanany-Vegh algorithm for a parallelogram \cite{HV}:} 

\begin{enumerate}[\rm (a)]
\item Consider primitive vectors orthogonal to each primitive side segments of the given polygon (see Figure~\ref{ex_FIA_orthogonal}). 
\item Consider curves on the two-torus $\sfT$ whose homology classes coincide with the above vectors, and write such curves on $\sfT$ according to the following rules: 
 \begin{enumerate}[\rm (b-1)] 
 \item They induce a cell decomposition of $\sfT$.
 \item Each curve intersects with other curves transversely and has a finite number of intersections. 
 \item No three curves intersect in the same point.
 \item Tracing along each curve, we see that its intersections with other curves occur with alternating orientations. 
 (For example, it is crossed from right to left and then left to right.) 
 \end{enumerate}
We call a resulting figure an \emph{admissible position} (see Figure~\ref{ex_FIA_homology}). 
\item After these processes, we have three kinds of quadrangles that are oriented clockwise, anti-clockwise and alternately: 
 \begin{center}
 {\scalebox{0.9}
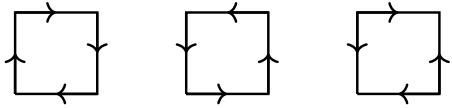
{
\begin{tikzpicture} 

\node (CC1) at (0,0) 
{\scalebox{1.2}{
\begin{tikzpicture}
\coordinate (P1) at (0.5,0); \coordinate (P2) at (1,0.5); \coordinate (P3) at (0.5,1); \coordinate (P4) at (0,0.5);   

\draw[line width=0.03cm]  (0,0)--(1,0)--(1,1)--(0,1)--(0,0); 
\draw[->, line width=0.03cm]  (0,0)--(P4); \draw[->, line width=0.03cm]  (1,0)--(P1); 
\draw[->, line width=0.03cm]  (1,1)--(P2); \draw[->, line width=0.03cm]  (0,1)--(P3); 
\end{tikzpicture} }}; 

\node (CC2) at (2.5,0) 
{\scalebox{1.2}{
\begin{tikzpicture}
\draw[line width=0.03cm]  (0,0)--(1,0)--(1,1)--(0,1)--(0,0); 
\draw[->, line width=0.03cm]  (0,0)--(P1); \draw[->, line width=0.03cm]  (1,0)--(P2); 
\draw[->, line width=0.03cm]  (1,1)--(P3); \draw[->, line width=0.03cm]  (0,1)--(P4); 
\end{tikzpicture} }} ;  

\node (CC3) at (5,0) 
{\scalebox{1.2}{
\begin{tikzpicture}
\draw[line width=0.03cm]  (0,0)--(1,0)--(1,1)--(0,1)--(0,0); 
\draw[->, line width=0.03cm]  (1,0)--(P1); \draw[->, line width=0.03cm]  (1,0)--(P2); 
\draw[->, line width=0.03cm]  (0,1)--(P3); \draw[->, line width=0.03cm]  (0,1)--(P4); 
\end{tikzpicture} }}; 

\end{tikzpicture}
}}
\end{center}
 
\item Draw white (resp. black) nodes in quadrangles oriented clockwise (resp. anti-clockwise). 
\item Connect white nodes to black ones facing each other across intersections of curves. 
\item Then, we obtain a square dimer model shown in Figure~\ref{ex_FIA_dimer}. We can check that this is isoradial, thus consistent in particular.  
\end{enumerate} 

\begin{figure}[H]
\begin{center}
\begin{tabular}{c}
\begin{minipage}{0.4\hsize}
\begin{center}
{\scalebox{0.3}{
\begin{tikzpicture}
\draw[line width=0.08cm]  (0,0) rectangle (12,12);

\draw[<-, dashed, line width=0.1cm,blue] (1,0)--(9,12); 
\draw[<-, dashed, line width=0.1cm,blue] (5,0)--(12,10.5);  
\draw[<-, dashed, line width=0.1cm,blue] (9,0)--(12,4.5); 
\draw[<-, dashed, line width=0.1cm,blue] (0,4.5)--(5,12);
\draw[<-, dashed, line width=0.1cm,blue] (0,10.5)--(1,12);

\draw[->, line width=0.1cm, blue] (3,0)--(11,12);
\draw[->, line width=0.1cm, blue] (7,0)--(12,7.5);
\draw[->, line width=0.1cm, blue] (11,0)--(12,1.5); 
\draw[->, line width=0.1cm, blue] (0,1.5)--(7,12); 
\draw[->, line width=0.1cm, blue] (0,7.5)--(3,12); 

\draw[->, line width=0.1cm,green] (12,0)--(0,12); 
\draw[->, line width=0.1cm,red] (6,0)--(0,6); \draw[->, line width=0.1cm,red] (12,6)--(6,12); 
\draw[->, dashed, line width=0.1cm,green] (0,3)--(3,0); \draw[->, dashed, line width=0.1cm,green] (3,12)--(12,3); 
\draw[->, dashed, line width=0.1cm,red] (0,9)--(9,0); \draw[->, dashed, line width=0.1cm,red] (9,12)--(12,9); 
\end{tikzpicture}
} }
\end{center}
\caption{}
\label{ex_FIA_homology}
\end{minipage}

\begin{minipage}{0.4\hsize}
\begin{center}
{\scalebox{0.3}{
\begin{tikzpicture}
\draw[line width=0.08cm]  (0,0) rectangle (12,12);

\draw[<-, dashed, line width=0.1cm,lightgray] (1,0)--(9,12); 
\draw[<-, dashed, line width=0.1cm,lightgray] (5,0)--(12,10.5);  
\draw[<-, dashed, line width=0.1cm,lightgray] (9,0)--(12,4.5); 
\draw[<-, dashed, line width=0.1cm,lightgray] (0,4.5)--(5,12);
\draw[<-, dashed, line width=0.1cm,lightgray] (0,10.5)--(1,12);

\draw[->, line width=0.1cm,lightgray] (3,0)--(11,12);
\draw[->, line width=0.1cm,lightgray] (7,0)--(12,7.5);
\draw[->, line width=0.1cm,lightgray] (11,0)--(12,1.5); 
\draw[->, line width=0.1cm,lightgray] (0,1.5)--(7,12); 
\draw[->, line width=0.1cm,lightgray] (0,7.5)--(3,12); 

\draw[->, line width=0.1cm,lightgray] (12,0)--(0,12); 
\draw[->, line width=0.1cm,lightgray] (6,0)--(0,6); \draw[->, line width=0.1cm,lightgray] (12,6)--(6,12); 
\draw[->, dashed, line width=0.1cm,lightgray] (0,3)--(3,0); \draw[->, dashed, line width=0.1cm,lightgray] (3,12)--(12,3); 
\draw[->, dashed, line width=0.1cm,lightgray] (0,9)--(9,0); \draw[->, dashed, line width=0.1cm,lightgray] (9,12)--(12,9); 

\coordinate (B1) at (0.5,4.025); \coordinate (B2) at (0.5,10.025); \coordinate (B3) at (3,1.65); \coordinate (B4) at (3,7.65);
\coordinate (B5) at (5.5,5.275); \coordinate (B6) at (5.5,11.275); \coordinate (B7) at (8,2.9); \coordinate (B8) at (8,8.9);
\coordinate (B9) at (10.5,0.525); \coordinate (B10) at (10.5,6.525); 

\coordinate (W1) at (0.5,1.025); \coordinate (W2) at (0.5,7.025); \coordinate (W3) at (3,4.65); \coordinate (W4) at (3,10.65);
\coordinate (W5) at (5.5,2.275); \coordinate (W6) at (5.5,8.275); \coordinate (W7) at (8,5.9); \coordinate (W8) at (8,11.9); \coordinate (W8b) at (8,0);
\coordinate (W9) at (10.5,3.525); \coordinate (W10) at (10.5,9.525); 

\draw[line width=0.08cm]  (W1)--(B1)--(W3)--(B3)--(W1); \draw[line width=0.08cm]  (W3)--(B3)--(W5)--(B5)--(W3); 
\draw[line width=0.08cm]  (W5)--(B5)--(W7)--(B7)--(W5); \draw[line width=0.08cm]  (W7)--(B7)--(W9)--(B10)--(W7); 
\draw[line width=0.08cm]  (B1)--(W2)--(B4)--(W3)--(B1); \draw[line width=0.08cm]  (B4)--(W6)--(B5)--(W3)--(B4); 
\draw[line width=0.08cm]  (B5)--(W6)--(B8)--(W7)--(B5); \draw[line width=0.08cm]  (B8)--(W10)--(B10)--(W7)--(B8); 
\draw[line width=0.08cm]  (W2)--(B2)--(W4)--(B4)--(W2); \draw[line width=0.08cm]  (W4)--(B4)--(W6)--(B6)--(W4);
\draw[line width=0.08cm]  (W6)--(B6)--(W8)--(B8)--(W6); \draw[line width=0.08cm]  (W8b)--(B9) ;

\draw[line width=0.08cm]  (W1)--(0.5,0); \draw[line width=0.08cm]  (B3)--(3,0); \draw[line width=0.08cm]  (W5)--(5.5,0); 
\draw[line width=0.08cm]  (B7)--(8,0); \draw[line width=0.08cm]  (W9)--(10.5,0); 
\draw[line width=0.08cm]  (B2)--(0.5,12); \draw[line width=0.08cm]  (W4)--(3,12); \draw[line width=0.08cm]  (B6)--(5.5,12); 
\draw[line width=0.08cm]  (W8)--(8,12); \draw[line width=0.08cm]  (W10)--(10.5,12);

\draw[line width=0.08cm]  (W1)--(0,0.9); \draw[line width=0.08cm]  (B1)--(0,3.9); \draw[line width=0.08cm]  (W2)--(0,6.9); 
\draw[line width=0.08cm]  (B2)--(0,9.9); \draw[line width=0.08cm]  (W9)--(12,3.9); \draw[line width=0.08cm]  (B10)--(12,6.9);
\draw[line width=0.08cm]  (W10)--(12,9.9); \draw[line width=0.08cm]  (B9)--(12,0.9);

\filldraw  [line width=0.05cm, fill=black] (B1) circle [radius=0.25] ; \filldraw  [line width=0.05cm, fill=black] (B2) circle [radius=0.25] ;
\filldraw  [line width=0.05cm, fill=black] (B3) circle [radius=0.25] ; \filldraw  [line width=0.05cm, fill=black] (B4) circle [radius=0.25] ;
\filldraw  [line width=0.05cm, fill=black] (B5) circle [radius=0.25] ; \filldraw  [line width=0.05cm, fill=black] (B6) circle [radius=0.25] ;
\filldraw  [line width=0.05cm, fill=black] (B7) circle [radius=0.25] ; \filldraw  [line width=0.05cm, fill=black] (B8) circle [radius=0.25] ;
\filldraw  [line width=0.05cm, fill=black] (B9) circle [radius=0.25] ;
\filldraw  [line width=0.05cm, fill=black] (B10) circle [radius=0.25] ;

\filldraw  [line width=0.07cm, fill=white] (W1) circle [radius=0.25] ; \filldraw  [line width=0.07cm, fill=white] (W2) circle [radius=0.25] ;
\filldraw  [line width=0.07cm, fill=white] (W3) circle [radius=0.25] ; \filldraw  [line width=0.07cm, fill=white] (W4) circle [radius=0.25] ;
\filldraw  [line width=0.07cm, fill=white] (W5) circle [radius=0.25] ; \filldraw  [line width=0.07cm, fill=white] (W6) circle [radius=0.25] ;
\filldraw  [line width=0.07cm, fill=white] (W7) circle [radius=0.25] ; \filldraw  [line width=0.07cm, fill=white] (W8) circle [radius=0.25] ;
\filldraw  [line width=0.07cm, fill=white] (W8b) circle [radius=0.25] ;
\filldraw  [line width=0.07cm, fill=white] (W9) circle [radius=0.25] ; \filldraw  [line width=0.07cm, fill=white] (W10) circle [radius=0.25] ;
\end{tikzpicture}
} }
\end{center}
\caption{}
\label{ex_FIA_dimer}
\end{minipage}

\end{tabular}
\end{center}
\end{figure}

Note that curves in an admissible position correspond to zigzag paths of the resulting consistent dimer model with the opposite direction. 
Thus, the correspondence in Proposition~\ref{zigzag_sidepolygon} asserts that the given parallelogram coincides with the perfect matching polygon 
by rotating $90$ degrees in the positive direction. Thus, we have the same lattice polygon up to unimodular transformations. 

Using the same argument, we can obtain a dimer model that is homotopy equivalent to a square dimer model for an arbitrary parallelogram.  

\medskip

Next, we show $(2)$$\Rightarrow$$(3)$. 
Let $\Gamma$ be a dimer model associated with a given toric singularity $R$, 
and suppose that $\Gamma$ is homotopy equivalent to a square dimer model. 
Thus, the universal cover of $\Gamma$ takes the form shown in Figure~\ref{covering_square}, 
and Figure~\ref{zig_zag_square} is the list of zigzag paths on the universal cover. (They continue infinitely in both directions.) 
Since these zigzag paths determine four distinct slopes, the toric diagram of $R$ is a quadrangle by Proposition~\ref{zigzag_sidepolygon}. 
In addition, by observing these zigzag paths, we see that $\Gamma$ is isoradial. 

\begin{figure}[h!]
\begin{center}
{\scalebox{0.35}{
\begin{tikzpicture}
\node (B1) at (0,0){$$}; \node (B2) at (4,0){$$}; \node (B3) at (8,0){$$}; \node (B4) at (12,0){$$}; 
\node (B5) at (2,2){$$}; \node (B6) at (6,2){$$}; \node (B7) at (10,2){$$}; \node (B8) at (14,2){$$}; 
\node (B9) at (0,4){$$}; \node (B10) at (4,4){$$}; \node (B11) at (8,4){$$}; \node (B12) at (12,4){$$}; 
\node (B13) at (2,6){$$}; \node (B14) at (6,6){$$}; \node (B15) at (10,6){$$}; \node (B16) at (14,6){$$};
\node (B17) at (0,8){$$}; \node (B18) at (4,8){$$}; \node (B19) at (8,8){$$}; \node (B20) at (12,8){$$}; 
\node (B21) at (2,10){$$}; \node (B22) at (6,10){$$}; \node (B23) at (10,10){$$}; \node (B24) at (14,10){$$};

\node (W1) at (2,0){$$}; \node (W2) at (6,0){$$}; \node (W3) at (10,0){$$}; \node (W4) at (14,0){$$}; 
\node (W5) at (0,2){$$}; \node (W6) at (4,2){$$}; \node (W7) at (8,2){$$}; \node (W8) at (12,2){$$};
\node (W9) at (2,4){$$}; \node (W10) at (6,4){$$}; \node (W11) at (10,4){$$}; \node (W12) at (14,4){$$}; 
\node (W13) at (0,6){$$}; \node (W14) at (4,6){$$}; \node (W15) at (8,6){$$}; \node (W16) at (12,6){$$};
\node (W17) at (2,8){$$}; \node (W18) at (6,8){$$}; \node (W19) at (10,8){$$}; \node (W20) at (14,8){$$}; 
\node (W21) at (0,10){$$}; \node (W22) at (4,10){$$}; \node (W23) at (8,10){$$}; \node (W24) at (12,10){$$};

\draw[line width=0.06cm]  (B1)--(W1)--(B5)--(W5)--(B1); \draw[line width=0.06cm]  (W1)--(B2)--(W6)--(B5)--(W1);
\draw[line width=0.06cm]  (B2)--(W2)--(B6)--(W6)--(B2); \draw[line width=0.06cm]  (W2)--(B3)--(W7)--(B6)--(W2);
\draw[line width=0.06cm]  (B3)--(W3)--(B7)--(W7)--(B3); \draw[line width=0.06cm]  (W3)--(B4)--(W8)--(B7)--(W3);
\draw[line width=0.06cm]  (B4)--(W4)--(B8)--(W8)--(B4); 

\draw[line width=0.06cm]  (W5)--(B5)--(W9)--(B9)--(W5); \draw[line width=0.06cm]  (B5)--(W6)--(B10)--(W9)--(B5); 
\draw[line width=0.06cm]  (W6)--(B6)--(W10)--(B10)--(W6); \draw[line width=0.06cm]  (B6)--(W7)--(B11)--(W10)--(B6); 
\draw[line width=0.06cm]  (W7)--(B7)--(W11)--(B11)--(W7); \draw[line width=0.06cm]  (B7)--(W8)--(B12)--(W11)--(B7); 
\draw[line width=0.06cm]  (W8)--(B8)--(W12)--(B12)--(W8); 

\draw[line width=0.06cm]  (B9)--(W9)--(B13)--(W13)--(B9); \draw[line width=0.06cm]  (W9)--(B10)--(W14)--(B13)--(W9);
\draw[line width=0.06cm]  (B10)--(W10)--(B14)--(W14)--(B10); \draw[line width=0.06cm]  (W10)--(B11)--(W15)--(B14)--(W10);
\draw[line width=0.06cm]  (B11)--(W11)--(B15)--(W15)--(B11); \draw[line width=0.06cm]  (W11)--(B12)--(W16)--(B15)--(W11);
\draw[line width=0.06cm]  (B12)--(W12)--(B16)--(W16)--(B12); 

\draw[line width=0.06cm]  (W13)--(B13)--(W17)--(B17)--(W13); \draw[line width=0.06cm]  (B13)--(W14)--(B18)--(W17)--(B13); 
\draw[line width=0.06cm]  (W14)--(B14)--(W18)--(B18)--(W14); \draw[line width=0.06cm]  (B14)--(W15)--(B19)--(W18)--(B14); 
\draw[line width=0.06cm]  (W15)--(B15)--(W19)--(B19)--(W15); \draw[line width=0.06cm]  (B15)--(W16)--(B20)--(W19)--(B15); 
\draw[line width=0.06cm]  (W16)--(B16)--(W20)--(B20)--(W16); 

\draw[line width=0.06cm]  (B17)--(W17)--(B21)--(W21)--(B17); \draw[line width=0.06cm]  (W17)--(B18)--(W22)--(B21)--(W17);
\draw[line width=0.06cm]  (B18)--(W18)--(B22)--(W22)--(B18); \draw[line width=0.06cm]  (W18)--(B19)--(W23)--(B22)--(W18);
\draw[line width=0.06cm]  (B19)--(W19)--(B23)--(W23)--(B19); \draw[line width=0.06cm]  (W19)--(B20)--(W24)--(B23)--(W19);
\draw[line width=0.06cm]  (B20)--(W20)--(B24)--(W24)--(B20); 

\filldraw  [ultra thick, fill=black] (0,0) circle [radius=0.25] ; \filldraw  [ultra thick, fill=black] (4,0) circle [radius=0.25] ;
\filldraw  [ultra thick, fill=black] (8,0) circle [radius=0.25] ; \filldraw  [ultra thick, fill=black] (12,0) circle [radius=0.25] ; 
\filldraw  [ultra thick, fill=black] (2,2) circle [radius=0.25] ; \filldraw  [ultra thick, fill=black] (6,2) circle [radius=0.25] ;
\filldraw  [ultra thick, fill=black] (10,2) circle [radius=0.25] ; \filldraw  [ultra thick, fill=black] (14,2) circle [radius=0.25] ;

\filldraw  [ultra thick, fill=black] (0,4) circle [radius=0.25] ; \filldraw  [ultra thick, fill=black] (4,4) circle [radius=0.25] ;
\filldraw  [ultra thick, fill=black] (8,4) circle [radius=0.25] ; \filldraw  [ultra thick, fill=black] (12,4) circle [radius=0.25] ;
\filldraw  [ultra thick, fill=black] (2,6) circle [radius=0.25] ; \filldraw  [ultra thick, fill=black] (6,6) circle [radius=0.25] ;
\filldraw  [ultra thick, fill=black] (10,6) circle [radius=0.25] ; \filldraw  [ultra thick, fill=black] (14,6) circle [radius=0.25] ;

\filldraw  [ultra thick, fill=black] (0,8) circle [radius=0.25] ; \filldraw  [ultra thick, fill=black] (4,8) circle [radius=0.25] ;
\filldraw  [ultra thick, fill=black] (8,8) circle [radius=0.25] ; \filldraw  [ultra thick, fill=black] (12,8) circle [radius=0.25] ;
\filldraw  [ultra thick, fill=black] (2,10) circle [radius=0.25] ; \filldraw  [ultra thick, fill=black] (6,10) circle [radius=0.25] ;
\filldraw  [ultra thick, fill=black] (10,10) circle [radius=0.25] ; \filldraw  [ultra thick, fill=black] (14,10) circle [radius=0.25] ;

\filldraw  [ultra thick, fill=white] (2,0) circle [radius=0.25] ; \filldraw  [ultra thick, fill=white] (6,0) circle [radius=0.25] ;
\filldraw  [ultra thick, fill=white] (10,0) circle [radius=0.25] ; \filldraw  [ultra thick, fill=white] (14,0) circle [radius=0.25] ;
\filldraw  [ultra thick, fill=white] (0,2) circle [radius=0.25] ; \filldraw  [ultra thick, fill=white] (4,2) circle [radius=0.25] ;
\filldraw  [ultra thick, fill=white] (8,2) circle [radius=0.25] ; \filldraw  [ultra thick, fill=white] (12,2) circle [radius=0.25] ;

\filldraw  [ultra thick, fill=white] (2,4) circle [radius=0.25] ; \filldraw  [ultra thick, fill=white] (6,4) circle [radius=0.25] ;
\filldraw  [ultra thick, fill=white] (10,4) circle [radius=0.25] ; \filldraw  [ultra thick, fill=white] (14,4) circle [radius=0.25] ;
\filldraw  [ultra thick, fill=white] (0,6) circle [radius=0.25] ; \filldraw  [ultra thick, fill=white] (4,6) circle [radius=0.25] ;
\filldraw  [ultra thick, fill=white] (8,6) circle [radius=0.25] ; \filldraw  [ultra thick, fill=white] (12,6) circle [radius=0.25] ;

\filldraw  [ultra thick, fill=white] (2,8) circle [radius=0.25] ; \filldraw  [ultra thick, fill=white] (6,8) circle [radius=0.25] ;
\filldraw  [ultra thick, fill=white] (10,8) circle [radius=0.25] ; \filldraw  [ultra thick, fill=white] (14,8) circle [radius=0.25] ;
\filldraw  [ultra thick, fill=white] (0,10) circle [radius=0.25] ; \filldraw  [ultra thick, fill=white] (4,10) circle [radius=0.25] ;
\filldraw  [ultra thick, fill=white] (8,10) circle [radius=0.25] ; \filldraw  [ultra thick, fill=white] (12,10) circle [radius=0.25] ;

\draw[loosely dotted, line width=0.1cm]  (7,-0.8)--(7,-1.8) ; \draw[loosely dotted, line width=0.1cm]  (7,10.8)--(7,11.8) ;
\draw[loosely dotted, line width=0.1cm]  (-1,5)--(-2,5) ; \draw[loosely dotted, line width=0.1cm]  (15,5)--(16,5) ;
\end{tikzpicture}
} }
\caption{The universal cover of a square dimer model}
\label{covering_square}
\end{center}
\end{figure}
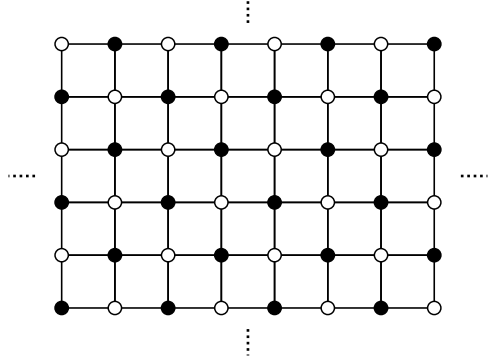

\begin{figure}[h!]
\begin{center}
{\scalebox{0.95}{
\begin{tikzpicture}
\node (ZZ1) at (0,0) 
{\scalebox{0.3}{
\begin{tikzpicture}
\coordinate (B1) at (0,0); \coordinate (B2) at (4,0); \coordinate (B3) at (8,0); \coordinate (B4) at (12,0); 
\coordinate (B5) at (2,2); \coordinate (B6) at (6,2); \coordinate (B7) at (10,2); \coordinate (B8) at (14,2); 
\coordinate (B9) at (0,4); \coordinate (B10) at (4,4); \coordinate (B11) at (8,4); \coordinate (B12) at (12,4); 
\coordinate (B13) at (2,6); \coordinate (B14) at (6,6); \coordinate (B15) at (10,6); \coordinate (B16) at (14,6);
\coordinate (B17) at (0,8); \coordinate (B18) at (4,8); \coordinate (B19) at (8,8); \coordinate (B20) at (12,8); 
\coordinate (B21) at (2,10); \coordinate (B22) at (6,10); \coordinate (B23) at (10,10); \coordinate (B24) at (14,10);

\coordinate (W1) at (2,0); \coordinate (W2) at (6,0); \coordinate (W3) at (10,0); \coordinate (W4) at (14,0); 
\coordinate (W5) at (0,2); \coordinate (W6) at (4,2); \coordinate (W7) at (8,2); \coordinate (W8) at (12,2);
\coordinate (W9) at (2,4); \coordinate (W10) at (6,4); \coordinate (W11) at (10,4); \coordinate (W12) at (14,4); 
\coordinate (W13) at (0,6); \coordinate (W14) at (4,6); \coordinate (W15) at (8,6); \coordinate (W16) at (12,6);
\coordinate (W17) at (2,8); \coordinate (W18) at (6,8); \coordinate (W19) at (10,8); \coordinate (W20) at (14,8); 
\coordinate (W21) at (0,10); \coordinate (W22) at (4,10); \coordinate (W23) at (8,10); \coordinate (W24) at (12,10);

\draw[line width=0.06cm]  (B1)--(W1)--(B5)--(W5)--(B1); \draw[line width=0.06cm]  (W1)--(B2)--(W6)--(B5)--(W1);
\draw[line width=0.06cm]  (B2)--(W2)--(B6)--(W6)--(B2); \draw[line width=0.06cm]  (W2)--(B3)--(W7)--(B6)--(W2);
\draw[line width=0.06cm]  (B3)--(W3)--(B7)--(W7)--(B3); \draw[line width=0.06cm]  (W3)--(B4)--(W8)--(B7)--(W3);
\draw[line width=0.06cm]  (B4)--(W4)--(B8)--(W8)--(B4); 

\draw[line width=0.06cm]  (W5)--(B5)--(W9)--(B9)--(W5); \draw[line width=0.06cm]  (B5)--(W6)--(B10)--(W9)--(B5); 
\draw[line width=0.06cm]  (W6)--(B6)--(W10)--(B10)--(W6); \draw[line width=0.06cm]  (B6)--(W7)--(B11)--(W10)--(B6); 
\draw[line width=0.06cm]  (W7)--(B7)--(W11)--(B11)--(W7); \draw[line width=0.06cm]  (B7)--(W8)--(B12)--(W11)--(B7); 
\draw[line width=0.06cm]  (W8)--(B8)--(W12)--(B12)--(W8); 

\draw[line width=0.06cm]  (B9)--(W9)--(B13)--(W13)--(B9); \draw[line width=0.06cm]  (W9)--(B10)--(W14)--(B13)--(W9);
\draw[line width=0.06cm]  (B10)--(W10)--(B14)--(W14)--(B10); \draw[line width=0.06cm]  (W10)--(B11)--(W15)--(B14)--(W10);
\draw[line width=0.06cm]  (B11)--(W11)--(B15)--(W15)--(B11); \draw[line width=0.06cm]  (W11)--(B12)--(W16)--(B15)--(W11);
\draw[line width=0.06cm]  (B12)--(W12)--(B16)--(W16)--(B12); 

\draw[line width=0.06cm]  (W13)--(B13)--(W17)--(B17)--(W13); \draw[line width=0.06cm]  (B13)--(W14)--(B18)--(W17)--(B13); 
\draw[line width=0.06cm]  (W14)--(B14)--(W18)--(B18)--(W14); \draw[line width=0.06cm]  (B14)--(W15)--(B19)--(W18)--(B14); 
\draw[line width=0.06cm]  (W15)--(B15)--(W19)--(B19)--(W15); \draw[line width=0.06cm]  (B15)--(W16)--(B20)--(W19)--(B15); 
\draw[line width=0.06cm]  (W16)--(B16)--(W20)--(B20)--(W16); 

\draw[line width=0.06cm]  (B17)--(W17)--(B21)--(W21)--(B17); \draw[line width=0.06cm]  (W17)--(B18)--(W22)--(B21)--(W17);
\draw[line width=0.06cm]  (B18)--(W18)--(B22)--(W22)--(B18); \draw[line width=0.06cm]  (W18)--(B19)--(W23)--(B22)--(W18);
\draw[line width=0.06cm]  (B19)--(W19)--(B23)--(W23)--(B19); \draw[line width=0.06cm]  (W19)--(B20)--(W24)--(B23)--(W19);
\draw[line width=0.06cm]  (B20)--(W20)--(B24)--(W24)--(B20); 
 
\filldraw  [ultra thick, fill=black] (0,0) circle [radius=0.25] ; \filldraw  [ultra thick, fill=black] (4,0) circle [radius=0.25] ;
\filldraw  [ultra thick, fill=black] (8,0) circle [radius=0.25] ; \filldraw  [ultra thick, fill=black] (12,0) circle [radius=0.25] ; 
\filldraw  [ultra thick, fill=black] (2,2) circle [radius=0.25] ; \filldraw  [ultra thick, fill=black] (6,2) circle [radius=0.25] ;
\filldraw  [ultra thick, fill=black] (10,2) circle [radius=0.25] ; \filldraw  [ultra thick, fill=black] (14,2) circle [radius=0.25] ;

\filldraw  [ultra thick, fill=black] (0,4) circle [radius=0.25] ; \filldraw  [ultra thick, fill=black] (4,4) circle [radius=0.25] ;
\filldraw  [ultra thick, fill=black] (8,4) circle [radius=0.25] ; \filldraw  [ultra thick, fill=black] (12,4) circle [radius=0.25] ;
\filldraw  [ultra thick, fill=black] (2,6) circle [radius=0.25] ; \filldraw  [ultra thick, fill=black] (6,6) circle [radius=0.25] ;
\filldraw  [ultra thick, fill=black] (10,6) circle [radius=0.25] ; \filldraw  [ultra thick, fill=black] (14,6) circle [radius=0.25] ;

\filldraw  [ultra thick, fill=black] (0,8) circle [radius=0.25] ; \filldraw  [ultra thick, fill=black] (4,8) circle [radius=0.25] ;
\filldraw  [ultra thick, fill=black] (8,8) circle [radius=0.25] ; \filldraw  [ultra thick, fill=black] (12,8) circle [radius=0.25] ;
\filldraw  [ultra thick, fill=black] (2,10) circle [radius=0.25] ; \filldraw  [ultra thick, fill=black] (6,10) circle [radius=0.25] ;
\filldraw  [ultra thick, fill=black] (10,10) circle [radius=0.25] ; \filldraw  [ultra thick, fill=black] (14,10) circle [radius=0.25] ;

\filldraw  [ultra thick, fill=white] (2,0) circle [radius=0.25] ; \filldraw  [ultra thick, fill=white] (6,0) circle [radius=0.25] ;
\filldraw  [ultra thick, fill=white] (10,0) circle [radius=0.25] ; \filldraw  [ultra thick, fill=white] (14,0) circle [radius=0.25] ;
\filldraw  [ultra thick, fill=white] (0,2) circle [radius=0.25] ; \filldraw  [ultra thick, fill=white] (4,2) circle [radius=0.25] ;
\filldraw  [ultra thick, fill=white] (8,2) circle [radius=0.25] ; \filldraw  [ultra thick, fill=white] (12,2) circle [radius=0.25] ;

\filldraw  [ultra thick, fill=white] (2,4) circle [radius=0.25] ; \filldraw  [ultra thick, fill=white] (6,4) circle [radius=0.25] ;
\filldraw  [ultra thick, fill=white] (10,4) circle [radius=0.25] ; \filldraw  [ultra thick, fill=white] (14,4) circle [radius=0.25] ;
\filldraw  [ultra thick, fill=white] (0,6) circle [radius=0.25] ; \filldraw  [ultra thick, fill=white] (4,6) circle [radius=0.25] ;
\filldraw  [ultra thick, fill=white] (8,6) circle [radius=0.25] ; \filldraw  [ultra thick, fill=white] (12,6) circle [radius=0.25] ;

\filldraw  [ultra thick, fill=white] (2,8) circle [radius=0.25] ; \filldraw  [ultra thick, fill=white] (6,8) circle [radius=0.25] ;
\filldraw  [ultra thick, fill=white] (10,8) circle [radius=0.25] ; \filldraw  [ultra thick, fill=white] (14,8) circle [radius=0.25] ;
\filldraw  [ultra thick, fill=white] (0,10) circle [radius=0.25] ; \filldraw  [ultra thick, fill=white] (4,10) circle [radius=0.25] ;
\filldraw  [ultra thick, fill=white] (8,10) circle [radius=0.25] ; \filldraw  [ultra thick, fill=white] (12,10) circle [radius=0.25] ;

\draw[->, line width=0.25cm, rounded corners, color=red] (W1)--(B2)--(W6)--(B6)--(W10)--(B11)--(W15)--(B15)--(W19)--(B20)--(W24)--(B24) ;
\draw[->, line width=0.25cm, rounded corners, color=red] (B1)--(W5)--(B5)--(W9)--(B10)--(W14)--(B14)--(W18)--(B19)--(W23)--(B23); 
\draw[->, line width=0.25cm, rounded corners, color=red] (B9)--(W13)--(B13)--(W17)--(B18)--(W22)--(B22) ;
\draw[->, line width=0.25cm, rounded corners, color=red] (B17)--(W21)--(B21) ; 
\draw[->, line width=0.25cm, rounded corners, color=red] (W2)--(B3)--(W7)--(B7)--(W11)--(B12)--(W16)--(B16)--(W20) ; 
\draw[->, line width=0.25cm, rounded corners, color=red] (W3)--(B4)--(W8)--(B8)--(W12) ; 
\end{tikzpicture}
} }; 

\node (ZZ2) at (6,0) 
{\scalebox{0.3}{
\begin{tikzpicture}
\coordinate (B1) at (0,0); \coordinate (B2) at (4,0); \coordinate (B3) at (8,0); \coordinate (B4) at (12,0); 
\coordinate (B5) at (2,2); \coordinate (B6) at (6,2); \coordinate (B7) at (10,2); \coordinate (B8) at (14,2); 
\coordinate (B9) at (0,4); \coordinate (B10) at (4,4); \coordinate (B11) at (8,4); \coordinate (B12) at (12,4); 
\coordinate (B13) at (2,6); \coordinate (B14) at (6,6); \coordinate (B15) at (10,6); \coordinate (B16) at (14,6);
\coordinate (B17) at (0,8); \coordinate (B18) at (4,8); \coordinate (B19) at (8,8); \coordinate (B20) at (12,8); 
\coordinate (B21) at (2,10); \coordinate (B22) at (6,10); \coordinate (B23) at (10,10); \coordinate (B24) at (14,10);

\coordinate (W1) at (2,0); \coordinate (W2) at (6,0); \coordinate (W3) at (10,0); \coordinate (W4) at (14,0); 
\coordinate (W5) at (0,2); \coordinate (W6) at (4,2); \coordinate (W7) at (8,2); \coordinate (W8) at (12,2);
\coordinate (W9) at (2,4); \coordinate (W10) at (6,4); \coordinate (W11) at (10,4); \coordinate (W12) at (14,4); 
\coordinate (W13) at (0,6); \coordinate (W14) at (4,6); \coordinate (W15) at (8,6); \coordinate (W16) at (12,6);
\coordinate (W17) at (2,8); \coordinate (W18) at (6,8); \coordinate (W19) at (10,8); \coordinate (W20) at (14,8); 
\coordinate (W21) at (0,10); \coordinate (W22) at (4,10); \coordinate (W23) at (8,10); \coordinate (W24) at (12,10);

\draw[line width=0.06cm]  (B1)--(W1)--(B5)--(W5)--(B1); \draw[line width=0.06cm]  (W1)--(B2)--(W6)--(B5)--(W1);
\draw[line width=0.06cm]  (B2)--(W2)--(B6)--(W6)--(B2); \draw[line width=0.06cm]  (W2)--(B3)--(W7)--(B6)--(W2);
\draw[line width=0.06cm]  (B3)--(W3)--(B7)--(W7)--(B3); \draw[line width=0.06cm]  (W3)--(B4)--(W8)--(B7)--(W3);
\draw[line width=0.06cm]  (B4)--(W4)--(B8)--(W8)--(B4); 

\draw[line width=0.06cm]  (W5)--(B5)--(W9)--(B9)--(W5); \draw[line width=0.06cm]  (B5)--(W6)--(B10)--(W9)--(B5); 
\draw[line width=0.06cm]  (W6)--(B6)--(W10)--(B10)--(W6); \draw[line width=0.06cm]  (B6)--(W7)--(B11)--(W10)--(B6); 
\draw[line width=0.06cm]  (W7)--(B7)--(W11)--(B11)--(W7); \draw[line width=0.06cm]  (B7)--(W8)--(B12)--(W11)--(B7); 
\draw[line width=0.06cm]  (W8)--(B8)--(W12)--(B12)--(W8); 

\draw[line width=0.06cm]  (B9)--(W9)--(B13)--(W13)--(B9); \draw[line width=0.06cm]  (W9)--(B10)--(W14)--(B13)--(W9);
\draw[line width=0.06cm]  (B10)--(W10)--(B14)--(W14)--(B10); \draw[line width=0.06cm]  (W10)--(B11)--(W15)--(B14)--(W10);
\draw[line width=0.06cm]  (B11)--(W11)--(B15)--(W15)--(B11); \draw[line width=0.06cm]  (W11)--(B12)--(W16)--(B15)--(W11);
\draw[line width=0.06cm]  (B12)--(W12)--(B16)--(W16)--(B12); 

\draw[line width=0.06cm]  (W13)--(B13)--(W17)--(B17)--(W13); \draw[line width=0.06cm]  (B13)--(W14)--(B18)--(W17)--(B13); 
\draw[line width=0.06cm]  (W14)--(B14)--(W18)--(B18)--(W14); \draw[line width=0.06cm]  (B14)--(W15)--(B19)--(W18)--(B14); 
\draw[line width=0.06cm]  (W15)--(B15)--(W19)--(B19)--(W15); \draw[line width=0.06cm]  (B15)--(W16)--(B20)--(W19)--(B15); 
\draw[line width=0.06cm]  (W16)--(B16)--(W20)--(B20)--(W16); 

\draw[line width=0.06cm]  (B17)--(W17)--(B21)--(W21)--(B17); \draw[line width=0.06cm]  (W17)--(B18)--(W22)--(B21)--(W17);
\draw[line width=0.06cm]  (B18)--(W18)--(B22)--(W22)--(B18); \draw[line width=0.06cm]  (W18)--(B19)--(W23)--(B22)--(W18);
\draw[line width=0.06cm]  (B19)--(W19)--(B23)--(W23)--(B19); \draw[line width=0.06cm]  (W19)--(B20)--(W24)--(B23)--(W19);
\draw[line width=0.06cm]  (B20)--(W20)--(B24)--(W24)--(B20); 

\filldraw  [ultra thick, fill=black] (0,0) circle [radius=0.25] ; \filldraw  [ultra thick, fill=black] (4,0) circle [radius=0.25] ;
\filldraw  [ultra thick, fill=black] (8,0) circle [radius=0.25] ; \filldraw  [ultra thick, fill=black] (12,0) circle [radius=0.25] ; 
\filldraw  [ultra thick, fill=black] (2,2) circle [radius=0.25] ; \filldraw  [ultra thick, fill=black] (6,2) circle [radius=0.25] ;
\filldraw  [ultra thick, fill=black] (10,2) circle [radius=0.25] ; \filldraw  [ultra thick, fill=black] (14,2) circle [radius=0.25] ;

\filldraw  [ultra thick, fill=black] (0,4) circle [radius=0.25] ; \filldraw  [ultra thick, fill=black] (4,4) circle [radius=0.25] ;
\filldraw  [ultra thick, fill=black] (8,4) circle [radius=0.25] ; \filldraw  [ultra thick, fill=black] (12,4) circle [radius=0.25] ;
\filldraw  [ultra thick, fill=black] (2,6) circle [radius=0.25] ; \filldraw  [ultra thick, fill=black] (6,6) circle [radius=0.25] ;
\filldraw  [ultra thick, fill=black] (10,6) circle [radius=0.25] ; \filldraw  [ultra thick, fill=black] (14,6) circle [radius=0.25] ;

\filldraw  [ultra thick, fill=black] (0,8) circle [radius=0.25] ; \filldraw  [ultra thick, fill=black] (4,8) circle [radius=0.25] ;
\filldraw  [ultra thick, fill=black] (8,8) circle [radius=0.25] ; \filldraw  [ultra thick, fill=black] (12,8) circle [radius=0.25] ;
\filldraw  [ultra thick, fill=black] (2,10) circle [radius=0.25] ; \filldraw  [ultra thick, fill=black] (6,10) circle [radius=0.25] ;
\filldraw  [ultra thick, fill=black] (10,10) circle [radius=0.25] ; \filldraw  [ultra thick, fill=black] (14,10) circle [radius=0.25] ;

\filldraw  [ultra thick, fill=white] (2,0) circle [radius=0.25] ; \filldraw  [ultra thick, fill=white] (6,0) circle [radius=0.25] ;
\filldraw  [ultra thick, fill=white] (10,0) circle [radius=0.25] ; \filldraw  [ultra thick, fill=white] (14,0) circle [radius=0.25] ;
\filldraw  [ultra thick, fill=white] (0,2) circle [radius=0.25] ; \filldraw  [ultra thick, fill=white] (4,2) circle [radius=0.25] ;
\filldraw  [ultra thick, fill=white] (8,2) circle [radius=0.25] ; \filldraw  [ultra thick, fill=white] (12,2) circle [radius=0.25] ;

\filldraw  [ultra thick, fill=white] (2,4) circle [radius=0.25] ; \filldraw  [ultra thick, fill=white] (6,4) circle [radius=0.25] ;
\filldraw  [ultra thick, fill=white] (10,4) circle [radius=0.25] ; \filldraw  [ultra thick, fill=white] (14,4) circle [radius=0.25] ;
\filldraw  [ultra thick, fill=white] (0,6) circle [radius=0.25] ; \filldraw  [ultra thick, fill=white] (4,6) circle [radius=0.25] ;
\filldraw  [ultra thick, fill=white] (8,6) circle [radius=0.25] ; \filldraw  [ultra thick, fill=white] (12,6) circle [radius=0.25] ;

\filldraw  [ultra thick, fill=white] (2,8) circle [radius=0.25] ; \filldraw  [ultra thick, fill=white] (6,8) circle [radius=0.25] ;
\filldraw  [ultra thick, fill=white] (10,8) circle [radius=0.25] ; \filldraw  [ultra thick, fill=white] (14,8) circle [radius=0.25] ;
\filldraw  [ultra thick, fill=white] (0,10) circle [radius=0.25] ; \filldraw  [ultra thick, fill=white] (4,10) circle [radius=0.25] ;
\filldraw  [ultra thick, fill=white] (8,10) circle [radius=0.25] ; \filldraw  [ultra thick, fill=white] (12,10) circle [radius=0.25] ;

\draw[->, line width=0.25cm, rounded corners, color=red] (W4)--(B8)--(W8)--(B12)--(W11)--(B15)--(W15)--(B19)--(W18)--(B22)--(W22) ;
\draw[->, line width=0.25cm, rounded corners, color=red] (B4)--(W3)--(B7)--(W7)--(B11)--(W10)--(B14)--(W14)--(B18)--(W17)--(B21)--(W21); 
\draw[->, line width=0.25cm, rounded corners, color=red] (B3)--(W2)--(B6)--(W6)--(B10)--(W9)--(B13)--(W13)--(B17); 
\draw[->, line width=0.25cm, rounded corners, color=red] (B2)--(W1)--(B5)--(W5)--(B9); 
\draw[->, line width=0.25cm, rounded corners, color=red] (W12)--(B16)--(W16)--(B20)--(W19)--(B23)--(W23); 
\draw[->, line width=0.25cm, rounded corners, color=red] (W20)--(B24)--(W24);
\end{tikzpicture}
} }; 

\node (ZZ3) at (0,-4) 
{\scalebox{0.3}{
\begin{tikzpicture}
\coordinate (B1) at (0,0); \coordinate (B2) at (4,0); \coordinate (B3) at (8,0); \coordinate (B4) at (12,0); 
\coordinate (B5) at (2,2); \coordinate (B6) at (6,2); \coordinate (B7) at (10,2); \coordinate (B8) at (14,2); 
\coordinate (B9) at (0,4); \coordinate (B10) at (4,4); \coordinate (B11) at (8,4); \coordinate (B12) at (12,4); 
\coordinate (B13) at (2,6); \coordinate (B14) at (6,6); \coordinate (B15) at (10,6); \coordinate (B16) at (14,6);
\coordinate (B17) at (0,8); \coordinate (B18) at (4,8); \coordinate (B19) at (8,8); \coordinate (B20) at (12,8); 
\coordinate (B21) at (2,10); \coordinate (B22) at (6,10); \coordinate (B23) at (10,10); \coordinate (B24) at (14,10);

\coordinate (W1) at (2,0); \coordinate (W2) at (6,0); \coordinate (W3) at (10,0); \coordinate (W4) at (14,0); 
\coordinate (W5) at (0,2); \coordinate (W6) at (4,2); \coordinate (W7) at (8,2); \coordinate (W8) at (12,2);
\coordinate (W9) at (2,4); \coordinate (W10) at (6,4); \coordinate (W11) at (10,4); \coordinate (W12) at (14,4); 
\coordinate (W13) at (0,6); \coordinate (W14) at (4,6); \coordinate (W15) at (8,6); \coordinate (W16) at (12,6);
\coordinate (W17) at (2,8); \coordinate (W18) at (6,8); \coordinate (W19) at (10,8); \coordinate (W20) at (14,8); 
\coordinate (W21) at (0,10); \coordinate (W22) at (4,10); \coordinate (W23) at (8,10); \coordinate (W24) at (12,10);

\draw[line width=0.06cm]  (B1)--(W1)--(B5)--(W5)--(B1); \draw[line width=0.06cm]  (W1)--(B2)--(W6)--(B5)--(W1);
\draw[line width=0.06cm]  (B2)--(W2)--(B6)--(W6)--(B2); \draw[line width=0.06cm]  (W2)--(B3)--(W7)--(B6)--(W2);
\draw[line width=0.06cm]  (B3)--(W3)--(B7)--(W7)--(B3); \draw[line width=0.06cm]  (W3)--(B4)--(W8)--(B7)--(W3);
\draw[line width=0.06cm]  (B4)--(W4)--(B8)--(W8)--(B4); 

\draw[line width=0.06cm]  (W5)--(B5)--(W9)--(B9)--(W5); \draw[line width=0.06cm]  (B5)--(W6)--(B10)--(W9)--(B5); 
\draw[line width=0.06cm]  (W6)--(B6)--(W10)--(B10)--(W6); \draw[line width=0.06cm]  (B6)--(W7)--(B11)--(W10)--(B6); 
\draw[line width=0.06cm]  (W7)--(B7)--(W11)--(B11)--(W7); \draw[line width=0.06cm]  (B7)--(W8)--(B12)--(W11)--(B7); 
\draw[line width=0.06cm]  (W8)--(B8)--(W12)--(B12)--(W8); 

\draw[line width=0.06cm]  (B9)--(W9)--(B13)--(W13)--(B9); \draw[line width=0.06cm]  (W9)--(B10)--(W14)--(B13)--(W9);
\draw[line width=0.06cm]  (B10)--(W10)--(B14)--(W14)--(B10); \draw[line width=0.06cm]  (W10)--(B11)--(W15)--(B14)--(W10);
\draw[line width=0.06cm]  (B11)--(W11)--(B15)--(W15)--(B11); \draw[line width=0.06cm]  (W11)--(B12)--(W16)--(B15)--(W11);
\draw[line width=0.06cm]  (B12)--(W12)--(B16)--(W16)--(B12); 

\draw[line width=0.06cm]  (W13)--(B13)--(W17)--(B17)--(W13); \draw[line width=0.06cm]  (B13)--(W14)--(B18)--(W17)--(B13); 
\draw[line width=0.06cm]  (W14)--(B14)--(W18)--(B18)--(W14); \draw[line width=0.06cm]  (B14)--(W15)--(B19)--(W18)--(B14); 
\draw[line width=0.06cm]  (W15)--(B15)--(W19)--(B19)--(W15); \draw[line width=0.06cm]  (B15)--(W16)--(B20)--(W19)--(B15); 
\draw[line width=0.06cm]  (W16)--(B16)--(W20)--(B20)--(W16); 

\draw[line width=0.06cm]  (B17)--(W17)--(B21)--(W21)--(B17); \draw[line width=0.06cm]  (W17)--(B18)--(W22)--(B21)--(W17);
\draw[line width=0.06cm]  (B18)--(W18)--(B22)--(W22)--(B18); \draw[line width=0.06cm]  (W18)--(B19)--(W23)--(B22)--(W18);
\draw[line width=0.06cm]  (B19)--(W19)--(B23)--(W23)--(B19); \draw[line width=0.06cm]  (W19)--(B20)--(W24)--(B23)--(W19);
\draw[line width=0.06cm]  (B20)--(W20)--(B24)--(W24)--(B20); 

\filldraw  [ultra thick, fill=black] (0,0) circle [radius=0.25] ; \filldraw  [ultra thick, fill=black] (4,0) circle [radius=0.25] ;
\filldraw  [ultra thick, fill=black] (8,0) circle [radius=0.25] ; \filldraw  [ultra thick, fill=black] (12,0) circle [radius=0.25] ; 
\filldraw  [ultra thick, fill=black] (2,2) circle [radius=0.25] ; \filldraw  [ultra thick, fill=black] (6,2) circle [radius=0.25] ;
\filldraw  [ultra thick, fill=black] (10,2) circle [radius=0.25] ; \filldraw  [ultra thick, fill=black] (14,2) circle [radius=0.25] ;

\filldraw  [ultra thick, fill=black] (0,4) circle [radius=0.25] ; \filldraw  [ultra thick, fill=black] (4,4) circle [radius=0.25] ;
\filldraw  [ultra thick, fill=black] (8,4) circle [radius=0.25] ; \filldraw  [ultra thick, fill=black] (12,4) circle [radius=0.25] ;
\filldraw  [ultra thick, fill=black] (2,6) circle [radius=0.25] ; \filldraw  [ultra thick, fill=black] (6,6) circle [radius=0.25] ;
\filldraw  [ultra thick, fill=black] (10,6) circle [radius=0.25] ; \filldraw  [ultra thick, fill=black] (14,6) circle [radius=0.25] ;

\filldraw  [ultra thick, fill=black] (0,8) circle [radius=0.25] ; \filldraw  [ultra thick, fill=black] (4,8) circle [radius=0.25] ;
\filldraw  [ultra thick, fill=black] (8,8) circle [radius=0.25] ; \filldraw  [ultra thick, fill=black] (12,8) circle [radius=0.25] ;
\filldraw  [ultra thick, fill=black] (2,10) circle [radius=0.25] ; \filldraw  [ultra thick, fill=black] (6,10) circle [radius=0.25] ;
\filldraw  [ultra thick, fill=black] (10,10) circle [radius=0.25] ; \filldraw  [ultra thick, fill=black] (14,10) circle [radius=0.25] ;

\filldraw  [ultra thick, fill=white] (2,0) circle [radius=0.25] ; \filldraw  [ultra thick, fill=white] (6,0) circle [radius=0.25] ;
\filldraw  [ultra thick, fill=white] (10,0) circle [radius=0.25] ; \filldraw  [ultra thick, fill=white] (14,0) circle [radius=0.25] ;
\filldraw  [ultra thick, fill=white] (0,2) circle [radius=0.25] ; \filldraw  [ultra thick, fill=white] (4,2) circle [radius=0.25] ;
\filldraw  [ultra thick, fill=white] (8,2) circle [radius=0.25] ; \filldraw  [ultra thick, fill=white] (12,2) circle [radius=0.25] ;

\filldraw  [ultra thick, fill=white] (2,4) circle [radius=0.25] ; \filldraw  [ultra thick, fill=white] (6,4) circle [radius=0.25] ;
\filldraw  [ultra thick, fill=white] (10,4) circle [radius=0.25] ; \filldraw  [ultra thick, fill=white] (14,4) circle [radius=0.25] ;
\filldraw  [ultra thick, fill=white] (0,6) circle [radius=0.25] ; \filldraw  [ultra thick, fill=white] (4,6) circle [radius=0.25] ;
\filldraw  [ultra thick, fill=white] (8,6) circle [radius=0.25] ; \filldraw  [ultra thick, fill=white] (12,6) circle [radius=0.25] ;

\filldraw  [ultra thick, fill=white] (2,8) circle [radius=0.25] ; \filldraw  [ultra thick, fill=white] (6,8) circle [radius=0.25] ;
\filldraw  [ultra thick, fill=white] (10,8) circle [radius=0.25] ; \filldraw  [ultra thick, fill=white] (14,8) circle [radius=0.25] ;
\filldraw  [ultra thick, fill=white] (0,10) circle [radius=0.25] ; \filldraw  [ultra thick, fill=white] (4,10) circle [radius=0.25] ;
\filldraw  [ultra thick, fill=white] (8,10) circle [radius=0.25] ; \filldraw  [ultra thick, fill=white] (12,10) circle [radius=0.25] ;

\draw[->, line width=0.25cm, rounded corners, color=red] (W24)--(B23)--(W19)--(B19)--(W15)--(B14)--(W10)--(B10)--(W6)--(B5)--(W1)--(B1) ;
\draw[->, line width=0.25cm, rounded corners, color=red] (W23)--(B22)--(W18)--(B18)--(W14)--(B13)--(W9)--(B9)--(W5) ;
\draw[->, line width=0.25cm, rounded corners, color=red] (W22)--(B21)--(W17)--(B17)--(W13) ;
\draw[->, line width=0.25cm, rounded corners, color=red] (B24)--(W20)--(B20)--(W16)--(B15)--(W11)--(B11)--(W7)--(B6)--(W2)--(B2) ;
\draw[->, line width=0.25cm, rounded corners, color=red] (B16)--(W12)--(B12)--(W8)--(B7)--(W3)--(B3) ; 
\draw[->, line width=0.25cm, rounded corners, color=red] (B8)--(W4)--(B4) ;
\end{tikzpicture}
} }; 

\node (ZZ4) at (6,-4) 
{\scalebox{0.3}{
\begin{tikzpicture}
\coordinate (B1) at (0,0); \coordinate (B2) at (4,0); \coordinate (B3) at (8,0); \coordinate (B4) at (12,0); 
\coordinate (B5) at (2,2); \coordinate (B6) at (6,2); \coordinate (B7) at (10,2); \coordinate (B8) at (14,2); 
\coordinate (B9) at (0,4); \coordinate (B10) at (4,4); \coordinate (B11) at (8,4); \coordinate (B12) at (12,4); 
\coordinate (B13) at (2,6); \coordinate (B14) at (6,6); \coordinate (B15) at (10,6); \coordinate (B16) at (14,6);
\coordinate (B17) at (0,8); \coordinate (B18) at (4,8); \coordinate (B19) at (8,8); \coordinate (B20) at (12,8); 
\coordinate (B21) at (2,10); \coordinate (B22) at (6,10); \coordinate (B23) at (10,10); \coordinate (B24) at (14,10);

\coordinate (W1) at (2,0); \coordinate (W2) at (6,0); \coordinate (W3) at (10,0); \coordinate (W4) at (14,0); 
\coordinate (W5) at (0,2); \coordinate (W6) at (4,2); \coordinate (W7) at (8,2); \coordinate (W8) at (12,2);
\coordinate (W9) at (2,4); \coordinate (W10) at (6,4); \coordinate (W11) at (10,4); \coordinate (W12) at (14,4); 
\coordinate (W13) at (0,6); \coordinate (W14) at (4,6); \coordinate (W15) at (8,6); \coordinate (W16) at (12,6);
\coordinate (W17) at (2,8); \coordinate (W18) at (6,8); \coordinate (W19) at (10,8); \coordinate (W20) at (14,8); 
\coordinate (W21) at (0,10); \coordinate (W22) at (4,10); \coordinate (W23) at (8,10); \coordinate (W24) at (12,10);

\draw[line width=0.06cm]  (B1)--(W1)--(B5)--(W5)--(B1); \draw[line width=0.06cm]  (W1)--(B2)--(W6)--(B5)--(W1);
\draw[line width=0.06cm]  (B2)--(W2)--(B6)--(W6)--(B2); \draw[line width=0.06cm]  (W2)--(B3)--(W7)--(B6)--(W2);
\draw[line width=0.06cm]  (B3)--(W3)--(B7)--(W7)--(B3); \draw[line width=0.06cm]  (W3)--(B4)--(W8)--(B7)--(W3);
\draw[line width=0.06cm]  (B4)--(W4)--(B8)--(W8)--(B4); 

\draw[line width=0.06cm]  (W5)--(B5)--(W9)--(B9)--(W5); \draw[line width=0.06cm]  (B5)--(W6)--(B10)--(W9)--(B5); 
\draw[line width=0.06cm]  (W6)--(B6)--(W10)--(B10)--(W6); \draw[line width=0.06cm]  (B6)--(W7)--(B11)--(W10)--(B6); 
\draw[line width=0.06cm]  (W7)--(B7)--(W11)--(B11)--(W7); \draw[line width=0.06cm]  (B7)--(W8)--(B12)--(W11)--(B7); 
\draw[line width=0.06cm]  (W8)--(B8)--(W12)--(B12)--(W8); 

\draw[line width=0.06cm]  (B9)--(W9)--(B13)--(W13)--(B9); \draw[line width=0.06cm]  (W9)--(B10)--(W14)--(B13)--(W9);
\draw[line width=0.06cm]  (B10)--(W10)--(B14)--(W14)--(B10); \draw[line width=0.06cm]  (W10)--(B11)--(W15)--(B14)--(W10);
\draw[line width=0.06cm]  (B11)--(W11)--(B15)--(W15)--(B11); \draw[line width=0.06cm]  (W11)--(B12)--(W16)--(B15)--(W11);
\draw[line width=0.06cm]  (B12)--(W12)--(B16)--(W16)--(B12); 

\draw[line width=0.06cm]  (W13)--(B13)--(W17)--(B17)--(W13); \draw[line width=0.06cm]  (B13)--(W14)--(B18)--(W17)--(B13); 
\draw[line width=0.06cm]  (W14)--(B14)--(W18)--(B18)--(W14); \draw[line width=0.06cm]  (B14)--(W15)--(B19)--(W18)--(B14); 
\draw[line width=0.06cm]  (W15)--(B15)--(W19)--(B19)--(W15); \draw[line width=0.06cm]  (B15)--(W16)--(B20)--(W19)--(B15); 
\draw[line width=0.06cm]  (W16)--(B16)--(W20)--(B20)--(W16); 

\draw[line width=0.06cm]  (B17)--(W17)--(B21)--(W21)--(B17); \draw[line width=0.06cm]  (W17)--(B18)--(W22)--(B21)--(W17);
\draw[line width=0.06cm]  (B18)--(W18)--(B22)--(W22)--(B18); \draw[line width=0.06cm]  (W18)--(B19)--(W23)--(B22)--(W18);
\draw[line width=0.06cm]  (B19)--(W19)--(B23)--(W23)--(B19); \draw[line width=0.06cm]  (W19)--(B20)--(W24)--(B23)--(W19);
\draw[line width=0.06cm]  (B20)--(W20)--(B24)--(W24)--(B20); 

\filldraw  [ultra thick, fill=black] (0,0) circle [radius=0.25] ; \filldraw  [ultra thick, fill=black] (4,0) circle [radius=0.25] ;
\filldraw  [ultra thick, fill=black] (8,0) circle [radius=0.25] ; \filldraw  [ultra thick, fill=black] (12,0) circle [radius=0.25] ; 
\filldraw  [ultra thick, fill=black] (2,2) circle [radius=0.25] ; \filldraw  [ultra thick, fill=black] (6,2) circle [radius=0.25] ;
\filldraw  [ultra thick, fill=black] (10,2) circle [radius=0.25] ; \filldraw  [ultra thick, fill=black] (14,2) circle [radius=0.25] ;

\filldraw  [ultra thick, fill=black] (0,4) circle [radius=0.25] ; \filldraw  [ultra thick, fill=black] (4,4) circle [radius=0.25] ;
\filldraw  [ultra thick, fill=black] (8,4) circle [radius=0.25] ; \filldraw  [ultra thick, fill=black] (12,4) circle [radius=0.25] ;
\filldraw  [ultra thick, fill=black] (2,6) circle [radius=0.25] ; \filldraw  [ultra thick, fill=black] (6,6) circle [radius=0.25] ;
\filldraw  [ultra thick, fill=black] (10,6) circle [radius=0.25] ; \filldraw  [ultra thick, fill=black] (14,6) circle [radius=0.25] ;

\filldraw  [ultra thick, fill=black] (0,8) circle [radius=0.25] ; \filldraw  [ultra thick, fill=black] (4,8) circle [radius=0.25] ;
\filldraw  [ultra thick, fill=black] (8,8) circle [radius=0.25] ; \filldraw  [ultra thick, fill=black] (12,8) circle [radius=0.25] ;
\filldraw  [ultra thick, fill=black] (2,10) circle [radius=0.25] ; \filldraw  [ultra thick, fill=black] (6,10) circle [radius=0.25] ;
\filldraw  [ultra thick, fill=black] (10,10) circle [radius=0.25] ; \filldraw  [ultra thick, fill=black] (14,10) circle [radius=0.25] ;

\filldraw  [ultra thick, fill=white] (2,0) circle [radius=0.25] ; \filldraw  [ultra thick, fill=white] (6,0) circle [radius=0.25] ;
\filldraw  [ultra thick, fill=white] (10,0) circle [radius=0.25] ; \filldraw  [ultra thick, fill=white] (14,0) circle [radius=0.25] ;
\filldraw  [ultra thick, fill=white] (0,2) circle [radius=0.25] ; \filldraw  [ultra thick, fill=white] (4,2) circle [radius=0.25] ;
\filldraw  [ultra thick, fill=white] (8,2) circle [radius=0.25] ; \filldraw  [ultra thick, fill=white] (12,2) circle [radius=0.25] ;

\filldraw  [ultra thick, fill=white] (2,4) circle [radius=0.25] ; \filldraw  [ultra thick, fill=white] (6,4) circle [radius=0.25] ;
\filldraw  [ultra thick, fill=white] (10,4) circle [radius=0.25] ; \filldraw  [ultra thick, fill=white] (14,4) circle [radius=0.25] ;
\filldraw  [ultra thick, fill=white] (0,6) circle [radius=0.25] ; \filldraw  [ultra thick, fill=white] (4,6) circle [radius=0.25] ;
\filldraw  [ultra thick, fill=white] (8,6) circle [radius=0.25] ; \filldraw  [ultra thick, fill=white] (12,6) circle [radius=0.25] ;

\filldraw  [ultra thick, fill=white] (2,8) circle [radius=0.25] ; \filldraw  [ultra thick, fill=white] (6,8) circle [radius=0.25] ;
\filldraw  [ultra thick, fill=white] (10,8) circle [radius=0.25] ; \filldraw  [ultra thick, fill=white] (14,8) circle [radius=0.25] ;
\filldraw  [ultra thick, fill=white] (0,10) circle [radius=0.25] ; \filldraw  [ultra thick, fill=white] (4,10) circle [radius=0.25] ;
\filldraw  [ultra thick, fill=white] (8,10) circle [radius=0.25] ; \filldraw  [ultra thick, fill=white] (12,10) circle [radius=0.25] ;

\draw[->, line width=0.25cm, rounded corners, color=red] (B21)--(W22)--(B18)--(W18)--(B14)--(W15)--(B11)--(W11)--(B7)--(W8)--(B4)--(W4) ;
\draw[->, line width=0.25cm, rounded corners, color=red] (B22)--(W23)--(B19)--(W19)--(B15)--(W16)--(B12)--(W12)--(B8) ;
\draw[->, line width=0.25cm, rounded corners, color=red] (B23)--(W24)--(B20)--(W20)--(B16) ;
\draw[->, line width=0.25cm, rounded corners, color=red] (W21)--(B17)--(W17)--(B13)--(W14)--(B10)--(W10)--(B6)--(W7)--(B3)--(W3) ;
\draw[->, line width=0.25cm, rounded corners, color=red] (W13)--(B9)--(W9)--(B5)--(W6)--(B2)--(W2) ;
\draw[->, line width=0.25cm, rounded corners, color=red] (W5)--(B1)--(W1) ;
\end{tikzpicture}
} }; 
\end{tikzpicture}
}}
\caption{Zigzag paths on a square dimer model}
\label{zig_zag_square}
\end{center}
\end{figure}
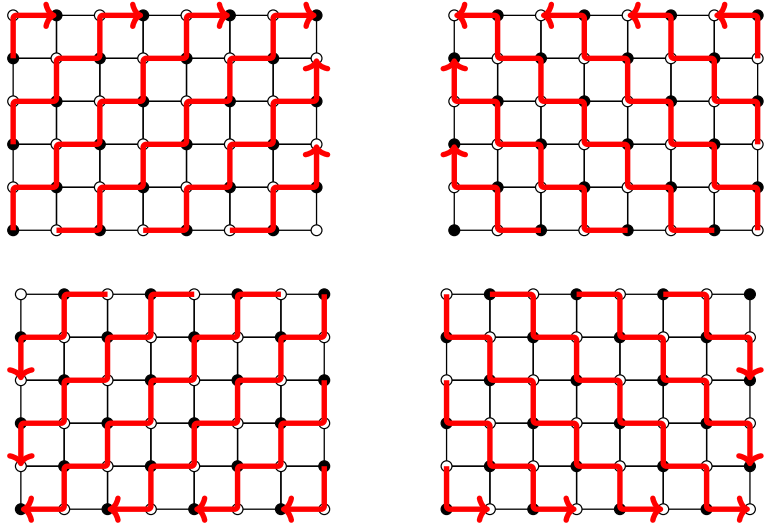

Let $(Q,W_Q)$ be the QP $(Q,W_Q)$ associated with $\Gamma$. By Theorem~\ref{NCCR1}, an MCM $R$-module 
\[ e_i\calP(Q, W_Q)\cong\Hom_R(T_{ii}, \bigoplus_{j\in Q_0}T_{ij})\cong\bigoplus_{j\in Q_0}T_{ij}\] 
gives an NCCR of $R$ for all $i\in Q_0$. 
Since we know that the toric diagram $\Delta$ of $R$ is a quadrangle, 
let $u_1,\cdots,u_4\in\ZZ^2$ be vertices of $\Delta$, and we assume that these are ordered cyclically along $\Delta$. 
Since $\Gamma$ is consistent, there exists a unique perfect matching, which is called extremal, corresponding to each vertex. 
We denote extremal perfect matchings corresponding to $u_1,\cdots,u_4$ by $\sfP_1,\cdots, \sfP_4$ respectively. 
Here, we recall that each module $T_{ij}$ can be constructed from the perfect matching functions of extremal ones (see subsection~\ref{sec_consist}). 
In our situation, extremal perfect matchings $\sfP_1,\cdots,\sfP_4$ are of the form shown in Figure~\ref{expm_square}, 
because differences of adjacent extremal perfect matchings induce zigzag paths. 

\begin{figure}[h!]
\begin{center}
{\scalebox{0.95}{
\begin{tikzpicture}
\node at (0,-2) {$\sfP_1$};\node at (6,-2) {$\sfP_2$}; 
\node at (0,-6.2) {$\sfP_3$}; \node at (6,-6.2) {$\sfP_4$}; 

\node (PM1) at (0,0) 
{\scalebox{0.3}{
\begin{tikzpicture}
\node (B1) at (0,0){$$}; \node (B2) at (4,0){$$}; \node (B3) at (8,0){$$}; \node (B4) at (12,0){$$}; 
\node (B5) at (2,2){$$}; \node (B6) at (6,2){$$}; \node (B7) at (10,2){$$}; \node (B8) at (14,2){$$}; 
\node (B9) at (0,4){$$}; \node (B10) at (4,4){$$}; \node (B11) at (8,4){$$}; \node (B12) at (12,4){$$}; 
\node (B13) at (2,6){$$}; \node (B14) at (6,6){$$}; \node (B15) at (10,6){$$}; \node (B16) at (14,6){$$};
\node (B17) at (0,8){$$}; \node (B18) at (4,8){$$}; \node (B19) at (8,8){$$}; \node (B20) at (12,8){$$}; 
\node (B21) at (2,10){$$}; \node (B22) at (6,10){$$}; \node (B23) at (10,10){$$}; \node (B24) at (14,10){$$};

\node (W1) at (2,0){$$}; \node (W2) at (6,0){$$}; \node (W3) at (10,0){$$}; \node (W4) at (14,0){$$}; 
\node (W5) at (0,2){$$}; \node (W6) at (4,2){$$}; \node (W7) at (8,2){$$}; \node (W8) at (12,2){$$};
\node (W9) at (2,4){$$}; \node (W10) at (6,4){$$}; \node (W11) at (10,4){$$}; \node (W12) at (14,4){$$}; 
\node (W13) at (0,6){$$}; \node (W14) at (4,6){$$}; \node (W15) at (8,6){$$}; \node (W16) at (12,6){$$};
\node (W17) at (2,8){$$}; \node (W18) at (6,8){$$}; \node (W19) at (10,8){$$}; \node (W20) at (14,8){$$}; 
\node (W21) at (0,10){$$}; \node (W22) at (4,10){$$}; \node (W23) at (8,10){$$}; \node (W24) at (12,10){$$};

\draw[line width=0.06cm]  (B1)--(W1)--(B5)--(W5)--(B1); \draw[line width=0.06cm]  (W1)--(B2)--(W6)--(B5)--(W1);
\draw[line width=0.06cm]  (B2)--(W2)--(B6)--(W6)--(B2); \draw[line width=0.06cm]  (W2)--(B3)--(W7)--(B6)--(W2);
\draw[line width=0.06cm]  (B3)--(W3)--(B7)--(W7)--(B3); \draw[line width=0.06cm]  (W3)--(B4)--(W8)--(B7)--(W3);
\draw[line width=0.06cm]  (B4)--(W4)--(B8)--(W8)--(B4); 

\draw[line width=0.06cm]  (W5)--(B5)--(W9)--(B9)--(W5); \draw[line width=0.06cm]  (B5)--(W6)--(B10)--(W9)--(B5); 
\draw[line width=0.06cm]  (W6)--(B6)--(W10)--(B10)--(W6); \draw[line width=0.06cm]  (B6)--(W7)--(B11)--(W10)--(B6); 
\draw[line width=0.06cm]  (W7)--(B7)--(W11)--(B11)--(W7); \draw[line width=0.06cm]  (B7)--(W8)--(B12)--(W11)--(B7); 
\draw[line width=0.06cm]  (W8)--(B8)--(W12)--(B12)--(W8); 

\draw[line width=0.06cm]  (B9)--(W9)--(B13)--(W13)--(B9); \draw[line width=0.06cm]  (W9)--(B10)--(W14)--(B13)--(W9);
\draw[line width=0.06cm]  (B10)--(W10)--(B14)--(W14)--(B10); \draw[line width=0.06cm]  (W10)--(B11)--(W15)--(B14)--(W10);
\draw[line width=0.06cm]  (B11)--(W11)--(B15)--(W15)--(B11); \draw[line width=0.06cm]  (W11)--(B12)--(W16)--(B15)--(W11);
\draw[line width=0.06cm]  (B12)--(W12)--(B16)--(W16)--(B12); 

\draw[line width=0.06cm]  (W13)--(B13)--(W17)--(B17)--(W13); \draw[line width=0.06cm]  (B13)--(W14)--(B18)--(W17)--(B13); 
\draw[line width=0.06cm]  (W14)--(B14)--(W18)--(B18)--(W14); \draw[line width=0.06cm]  (B14)--(W15)--(B19)--(W18)--(B14); 
\draw[line width=0.06cm]  (W15)--(B15)--(W19)--(B19)--(W15); \draw[line width=0.06cm]  (B15)--(W16)--(B20)--(W19)--(B15); 
\draw[line width=0.06cm]  (W16)--(B16)--(W20)--(B20)--(W16); 

\draw[line width=0.06cm]  (B17)--(W17)--(B21)--(W21)--(B17); \draw[line width=0.06cm]  (W17)--(B18)--(W22)--(B21)--(W17);
\draw[line width=0.06cm]  (B18)--(W18)--(B22)--(W22)--(B18); \draw[line width=0.06cm]  (W18)--(B19)--(W23)--(B22)--(W18);
\draw[line width=0.06cm]  (B19)--(W19)--(B23)--(W23)--(B19); \draw[line width=0.06cm]  (W19)--(B20)--(W24)--(B23)--(W19);
\draw[line width=0.06cm]  (B20)--(W20)--(B24)--(W24)--(B20); 

\draw[line width=0.4cm,color=blue] (B2)--(W1); \draw[line width=0.4cm,color=blue] (B3)--(W2); 
\draw[line width=0.4cm,color=blue] (B4)--(W3); \draw[line width=0.4cm,color=blue] (B5)--(W5); 
\draw[line width=0.4cm,color=blue] (B6)--(W6); \draw[line width=0.4cm,color=blue] (B7)--(W7); \draw[line width=0.4cm,color=blue] (B8)--(W8); 
\draw[line width=0.4cm,color=blue] (B10)--(W9); \draw[line width=0.4cm,color=blue] (B11)--(W10); 
\draw[line width=0.4cm,color=blue] (B12)--(W11); \draw[line width=0.4cm,color=blue] (B13)--(W13); 
\draw[line width=0.4cm,color=blue] (B14)--(W14); \draw[line width=0.4cm,color=blue] (B15)--(W15);  
\draw[line width=0.4cm,color=blue] (B16)--(W16); \draw[line width=0.4cm,color=blue] (B18)--(W17); 
\draw[line width=0.4cm,color=blue] (B19)--(W18); \draw[line width=0.4cm,color=blue] (B20)--(W19); 
\draw[line width=0.4cm,color=blue] (B21)--(W21); \draw[line width=0.4cm,color=blue] (B22)--(W22); 
\draw[line width=0.4cm,color=blue] (B23)--(W23); \draw[line width=0.4cm,color=blue] (B24)--(W24); 

\filldraw  [ultra thick, fill=black] (0,0) circle [radius=0.25] ; \filldraw  [ultra thick, fill=black] (4,0) circle [radius=0.25] ;
\filldraw  [ultra thick, fill=black] (8,0) circle [radius=0.25] ; \filldraw  [ultra thick, fill=black] (12,0) circle [radius=0.25] ; 
\filldraw  [ultra thick, fill=black] (2,2) circle [radius=0.25] ; \filldraw  [ultra thick, fill=black] (6,2) circle [radius=0.25] ;
\filldraw  [ultra thick, fill=black] (10,2) circle [radius=0.25] ; \filldraw  [ultra thick, fill=black] (14,2) circle [radius=0.25] ;

\filldraw  [ultra thick, fill=black] (0,4) circle [radius=0.25] ; \filldraw  [ultra thick, fill=black] (4,4) circle [radius=0.25] ;
\filldraw  [ultra thick, fill=black] (8,4) circle [radius=0.25] ; \filldraw  [ultra thick, fill=black] (12,4) circle [radius=0.25] ;
\filldraw  [ultra thick, fill=black] (2,6) circle [radius=0.25] ; \filldraw  [ultra thick, fill=black] (6,6) circle [radius=0.25] ;
\filldraw  [ultra thick, fill=black] (10,6) circle [radius=0.25] ; \filldraw  [ultra thick, fill=black] (14,6) circle [radius=0.25] ;

\filldraw  [ultra thick, fill=black] (0,8) circle [radius=0.25] ; \filldraw  [ultra thick, fill=black] (4,8) circle [radius=0.25] ;
\filldraw  [ultra thick, fill=black] (8,8) circle [radius=0.25] ; \filldraw  [ultra thick, fill=black] (12,8) circle [radius=0.25] ;
\filldraw  [ultra thick, fill=black] (2,10) circle [radius=0.25] ; \filldraw  [ultra thick, fill=black] (6,10) circle [radius=0.25] ;
\filldraw  [ultra thick, fill=black] (10,10) circle [radius=0.25] ; \filldraw  [ultra thick, fill=black] (14,10) circle [radius=0.25] ;

\filldraw  [ultra thick, fill=white] (2,0) circle [radius=0.25] ; \filldraw  [ultra thick, fill=white] (6,0) circle [radius=0.25] ;
\filldraw  [ultra thick, fill=white] (10,0) circle [radius=0.25] ; \filldraw  [ultra thick, fill=white] (14,0) circle [radius=0.25] ;
\filldraw  [ultra thick, fill=white] (0,2) circle [radius=0.25] ; \filldraw  [ultra thick, fill=white] (4,2) circle [radius=0.25] ;
\filldraw  [ultra thick, fill=white] (8,2) circle [radius=0.25] ; \filldraw  [ultra thick, fill=white] (12,2) circle [radius=0.25] ;

\filldraw  [ultra thick, fill=white] (2,4) circle [radius=0.25] ; \filldraw  [ultra thick, fill=white] (6,4) circle [radius=0.25] ;
\filldraw  [ultra thick, fill=white] (10,4) circle [radius=0.25] ; \filldraw  [ultra thick, fill=white] (14,4) circle [radius=0.25] ;
\filldraw  [ultra thick, fill=white] (0,6) circle [radius=0.25] ; \filldraw  [ultra thick, fill=white] (4,6) circle [radius=0.25] ;
\filldraw  [ultra thick, fill=white] (8,6) circle [radius=0.25] ; \filldraw  [ultra thick, fill=white] (12,6) circle [radius=0.25] ;

\filldraw  [ultra thick, fill=white] (2,8) circle [radius=0.25] ; \filldraw  [ultra thick, fill=white] (6,8) circle [radius=0.25] ;
\filldraw  [ultra thick, fill=white] (10,8) circle [radius=0.25] ; \filldraw  [ultra thick, fill=white] (14,8) circle [radius=0.25] ;
\filldraw  [ultra thick, fill=white] (0,10) circle [radius=0.25] ; \filldraw  [ultra thick, fill=white] (4,10) circle [radius=0.25] ;
\filldraw  [ultra thick, fill=white] (8,10) circle [radius=0.25] ; \filldraw  [ultra thick, fill=white] (12,10) circle [radius=0.25] ;
\end{tikzpicture}
} }; 

\node (PM2) at (6,0) 
{\scalebox{0.3}{
\begin{tikzpicture}
\node (B1) at (0,0){$$}; \node (B2) at (4,0){$$}; \node (B3) at (8,0){$$}; \node (B4) at (12,0){$$}; 
\node (B5) at (2,2){$$}; \node (B6) at (6,2){$$}; \node (B7) at (10,2){$$}; \node (B8) at (14,2){$$}; 
\node (B9) at (0,4){$$}; \node (B10) at (4,4){$$}; \node (B11) at (8,4){$$}; \node (B12) at (12,4){$$}; 
\node (B13) at (2,6){$$}; \node (B14) at (6,6){$$}; \node (B15) at (10,6){$$}; \node (B16) at (14,6){$$};
\node (B17) at (0,8){$$}; \node (B18) at (4,8){$$}; \node (B19) at (8,8){$$}; \node (B20) at (12,8){$$}; 
\node (B21) at (2,10){$$}; \node (B22) at (6,10){$$}; \node (B23) at (10,10){$$}; \node (B24) at (14,10){$$};

\node (W1) at (2,0){$$}; \node (W2) at (6,0){$$}; \node (W3) at (10,0){$$}; \node (W4) at (14,0){$$}; 
\node (W5) at (0,2){$$}; \node (W6) at (4,2){$$}; \node (W7) at (8,2){$$}; \node (W8) at (12,2){$$};
\node (W9) at (2,4){$$}; \node (W10) at (6,4){$$}; \node (W11) at (10,4){$$}; \node (W12) at (14,4){$$}; 
\node (W13) at (0,6){$$}; \node (W14) at (4,6){$$}; \node (W15) at (8,6){$$}; \node (W16) at (12,6){$$};
\node (W17) at (2,8){$$}; \node (W18) at (6,8){$$}; \node (W19) at (10,8){$$}; \node (W20) at (14,8){$$}; 
\node (W21) at (0,10){$$}; \node (W22) at (4,10){$$}; \node (W23) at (8,10){$$}; \node (W24) at (12,10){$$};

\draw[line width=0.06cm]  (B1)--(W1)--(B5)--(W5)--(B1); \draw[line width=0.06cm]  (W1)--(B2)--(W6)--(B5)--(W1);
\draw[line width=0.06cm]  (B2)--(W2)--(B6)--(W6)--(B2); \draw[line width=0.06cm]  (W2)--(B3)--(W7)--(B6)--(W2);
\draw[line width=0.06cm]  (B3)--(W3)--(B7)--(W7)--(B3); \draw[line width=0.06cm]  (W3)--(B4)--(W8)--(B7)--(W3);
\draw[line width=0.06cm]  (B4)--(W4)--(B8)--(W8)--(B4); 

\draw[line width=0.06cm]  (W5)--(B5)--(W9)--(B9)--(W5); \draw[line width=0.06cm]  (B5)--(W6)--(B10)--(W9)--(B5); 
\draw[line width=0.06cm]  (W6)--(B6)--(W10)--(B10)--(W6); \draw[line width=0.06cm]  (B6)--(W7)--(B11)--(W10)--(B6); 
\draw[line width=0.06cm]  (W7)--(B7)--(W11)--(B11)--(W7); \draw[line width=0.06cm]  (B7)--(W8)--(B12)--(W11)--(B7); 
\draw[line width=0.06cm]  (W8)--(B8)--(W12)--(B12)--(W8); 

\draw[line width=0.06cm]  (B9)--(W9)--(B13)--(W13)--(B9); \draw[line width=0.06cm]  (W9)--(B10)--(W14)--(B13)--(W9);
\draw[line width=0.06cm]  (B10)--(W10)--(B14)--(W14)--(B10); \draw[line width=0.06cm]  (W10)--(B11)--(W15)--(B14)--(W10);
\draw[line width=0.06cm]  (B11)--(W11)--(B15)--(W15)--(B11); \draw[line width=0.06cm]  (W11)--(B12)--(W16)--(B15)--(W11);
\draw[line width=0.06cm]  (B12)--(W12)--(B16)--(W16)--(B12); 

\draw[line width=0.06cm]  (W13)--(B13)--(W17)--(B17)--(W13); \draw[line width=0.06cm]  (B13)--(W14)--(B18)--(W17)--(B13); 
\draw[line width=0.06cm]  (W14)--(B14)--(W18)--(B18)--(W14); \draw[line width=0.06cm]  (B14)--(W15)--(B19)--(W18)--(B14); 
\draw[line width=0.06cm]  (W15)--(B15)--(W19)--(B19)--(W15); \draw[line width=0.06cm]  (B15)--(W16)--(B20)--(W19)--(B15); 
\draw[line width=0.06cm]  (W16)--(B16)--(W20)--(B20)--(W16); 

\draw[line width=0.06cm]  (B17)--(W17)--(B21)--(W21)--(B17); \draw[line width=0.06cm]  (W17)--(B18)--(W22)--(B21)--(W17);
\draw[line width=0.06cm]  (B18)--(W18)--(B22)--(W22)--(B18); \draw[line width=0.06cm]  (W18)--(B19)--(W23)--(B22)--(W18);
\draw[line width=0.06cm]  (B19)--(W19)--(B23)--(W23)--(B19); \draw[line width=0.06cm]  (W19)--(B20)--(W24)--(B23)--(W19);
\draw[line width=0.06cm]  (B20)--(W20)--(B24)--(W24)--(B20); 

\draw[line width=0.4cm,color=blue] (W1)--(B5); \draw[line width=0.4cm,color=blue] (W2)--(B6); \draw[line width=0.4cm,color=blue] (W3)--(B7);  
\draw[line width=0.4cm,color=blue] (W4)--(B8); \draw[line width=0.4cm,color=blue] (W5)--(B9); \draw[line width=0.4cm,color=blue] (W6)--(B10); 
\draw[line width=0.4cm,color=blue] (W7)--(B11); \draw[line width=0.4cm,color=blue] (W8)--(B12); \draw[line width=0.4cm,color=blue] (W9)--(B13); 
\draw[line width=0.4cm,color=blue] (W10)--(B14); \draw[line width=0.4cm,color=blue] (W11)--(B15); \draw[line width=0.4cm,color=blue] (W12)--(B16); 
\draw[line width=0.4cm,color=blue] (W13)--(B17); \draw[line width=0.4cm,color=blue] (W14)--(B18); \draw[line width=0.4cm,color=blue] (W15)--(B19);
\draw[line width=0.4cm,color=blue] (W16)--(B20); \draw[line width=0.4cm,color=blue] (W17)--(B21);  \draw[line width=0.4cm,color=blue] (W18)--(B22); 
\draw[line width=0.4cm,color=blue] (W19)--(B23);  \draw[line width=0.4cm,color=blue] (W20)--(B24); 

\filldraw  [ultra thick, fill=black] (0,0) circle [radius=0.25] ; \filldraw  [ultra thick, fill=black] (4,0) circle [radius=0.25] ;
\filldraw  [ultra thick, fill=black] (8,0) circle [radius=0.25] ; \filldraw  [ultra thick, fill=black] (12,0) circle [radius=0.25] ; 
\filldraw  [ultra thick, fill=black] (2,2) circle [radius=0.25] ; \filldraw  [ultra thick, fill=black] (6,2) circle [radius=0.25] ;
\filldraw  [ultra thick, fill=black] (10,2) circle [radius=0.25] ; \filldraw  [ultra thick, fill=black] (14,2) circle [radius=0.25] ;

\filldraw  [ultra thick, fill=black] (0,4) circle [radius=0.25] ; \filldraw  [ultra thick, fill=black] (4,4) circle [radius=0.25] ;
\filldraw  [ultra thick, fill=black] (8,4) circle [radius=0.25] ; \filldraw  [ultra thick, fill=black] (12,4) circle [radius=0.25] ;
\filldraw  [ultra thick, fill=black] (2,6) circle [radius=0.25] ; \filldraw  [ultra thick, fill=black] (6,6) circle [radius=0.25] ;
\filldraw  [ultra thick, fill=black] (10,6) circle [radius=0.25] ; \filldraw  [ultra thick, fill=black] (14,6) circle [radius=0.25] ;

\filldraw  [ultra thick, fill=black] (0,8) circle [radius=0.25] ; \filldraw  [ultra thick, fill=black] (4,8) circle [radius=0.25] ;
\filldraw  [ultra thick, fill=black] (8,8) circle [radius=0.25] ; \filldraw  [ultra thick, fill=black] (12,8) circle [radius=0.25] ;
\filldraw  [ultra thick, fill=black] (2,10) circle [radius=0.25] ; \filldraw  [ultra thick, fill=black] (6,10) circle [radius=0.25] ;
\filldraw  [ultra thick, fill=black] (10,10) circle [radius=0.25] ; \filldraw  [ultra thick, fill=black] (14,10) circle [radius=0.25] ;

\filldraw  [ultra thick, fill=white] (2,0) circle [radius=0.25] ; \filldraw  [ultra thick, fill=white] (6,0) circle [radius=0.25] ;
\filldraw  [ultra thick, fill=white] (10,0) circle [radius=0.25] ; \filldraw  [ultra thick, fill=white] (14,0) circle [radius=0.25] ;
\filldraw  [ultra thick, fill=white] (0,2) circle [radius=0.25] ; \filldraw  [ultra thick, fill=white] (4,2) circle [radius=0.25] ;
\filldraw  [ultra thick, fill=white] (8,2) circle [radius=0.25] ; \filldraw  [ultra thick, fill=white] (12,2) circle [radius=0.25] ;

\filldraw  [ultra thick, fill=white] (2,4) circle [radius=0.25] ; \filldraw  [ultra thick, fill=white] (6,4) circle [radius=0.25] ;
\filldraw  [ultra thick, fill=white] (10,4) circle [radius=0.25] ; \filldraw  [ultra thick, fill=white] (14,4) circle [radius=0.25] ;
\filldraw  [ultra thick, fill=white] (0,6) circle [radius=0.25] ; \filldraw  [ultra thick, fill=white] (4,6) circle [radius=0.25] ;
\filldraw  [ultra thick, fill=white] (8,6) circle [radius=0.25] ; \filldraw  [ultra thick, fill=white] (12,6) circle [radius=0.25] ;

\filldraw  [ultra thick, fill=white] (2,8) circle [radius=0.25] ; \filldraw  [ultra thick, fill=white] (6,8) circle [radius=0.25] ;
\filldraw  [ultra thick, fill=white] (10,8) circle [radius=0.25] ; \filldraw  [ultra thick, fill=white] (14,8) circle [radius=0.25] ;
\filldraw  [ultra thick, fill=white] (0,10) circle [radius=0.25] ; \filldraw  [ultra thick, fill=white] (4,10) circle [radius=0.25] ;
\filldraw  [ultra thick, fill=white] (8,10) circle [radius=0.25] ; \filldraw  [ultra thick, fill=white] (12,10) circle [radius=0.25] ;
\end{tikzpicture}
} }; 

\node (PM3) at (0,-4.2) 
{\scalebox{0.3}{
\begin{tikzpicture}
\node (B1) at (0,0){$$}; \node (B2) at (4,0){$$}; \node (B3) at (8,0){$$}; \node (B4) at (12,0){$$}; 
\node (B5) at (2,2){$$}; \node (B6) at (6,2){$$}; \node (B7) at (10,2){$$}; \node (B8) at (14,2){$$}; 
\node (B9) at (0,4){$$}; \node (B10) at (4,4){$$}; \node (B11) at (8,4){$$}; \node (B12) at (12,4){$$}; 
\node (B13) at (2,6){$$}; \node (B14) at (6,6){$$}; \node (B15) at (10,6){$$}; \node (B16) at (14,6){$$};
\node (B17) at (0,8){$$}; \node (B18) at (4,8){$$}; \node (B19) at (8,8){$$}; \node (B20) at (12,8){$$}; 
\node (B21) at (2,10){$$}; \node (B22) at (6,10){$$}; \node (B23) at (10,10){$$}; \node (B24) at (14,10){$$};

\node (W1) at (2,0){$$}; \node (W2) at (6,0){$$}; \node (W3) at (10,0){$$}; \node (W4) at (14,0){$$}; 
\node (W5) at (0,2){$$}; \node (W6) at (4,2){$$}; \node (W7) at (8,2){$$}; \node (W8) at (12,2){$$};
\node (W9) at (2,4){$$}; \node (W10) at (6,4){$$}; \node (W11) at (10,4){$$}; \node (W12) at (14,4){$$}; 
\node (W13) at (0,6){$$}; \node (W14) at (4,6){$$}; \node (W15) at (8,6){$$}; \node (W16) at (12,6){$$};
\node (W17) at (2,8){$$}; \node (W18) at (6,8){$$}; \node (W19) at (10,8){$$}; \node (W20) at (14,8){$$}; 
\node (W21) at (0,10){$$}; \node (W22) at (4,10){$$}; \node (W23) at (8,10){$$}; \node (W24) at (12,10){$$};

\draw[line width=0.06cm]  (B1)--(W1)--(B5)--(W5)--(B1); \draw[line width=0.06cm]  (W1)--(B2)--(W6)--(B5)--(W1);
\draw[line width=0.06cm]  (B2)--(W2)--(B6)--(W6)--(B2); \draw[line width=0.06cm]  (W2)--(B3)--(W7)--(B6)--(W2);
\draw[line width=0.06cm]  (B3)--(W3)--(B7)--(W7)--(B3); \draw[line width=0.06cm]  (W3)--(B4)--(W8)--(B7)--(W3);
\draw[line width=0.06cm]  (B4)--(W4)--(B8)--(W8)--(B4); 

\draw[line width=0.06cm]  (W5)--(B5)--(W9)--(B9)--(W5); \draw[line width=0.06cm]  (B5)--(W6)--(B10)--(W9)--(B5); 
\draw[line width=0.06cm]  (W6)--(B6)--(W10)--(B10)--(W6); \draw[line width=0.06cm]  (B6)--(W7)--(B11)--(W10)--(B6); 
\draw[line width=0.06cm]  (W7)--(B7)--(W11)--(B11)--(W7); \draw[line width=0.06cm]  (B7)--(W8)--(B12)--(W11)--(B7); 
\draw[line width=0.06cm]  (W8)--(B8)--(W12)--(B12)--(W8); 

\draw[line width=0.06cm]  (B9)--(W9)--(B13)--(W13)--(B9); \draw[line width=0.06cm]  (W9)--(B10)--(W14)--(B13)--(W9);
\draw[line width=0.06cm]  (B10)--(W10)--(B14)--(W14)--(B10); \draw[line width=0.06cm]  (W10)--(B11)--(W15)--(B14)--(W10);
\draw[line width=0.06cm]  (B11)--(W11)--(B15)--(W15)--(B11); \draw[line width=0.06cm]  (W11)--(B12)--(W16)--(B15)--(W11);
\draw[line width=0.06cm]  (B12)--(W12)--(B16)--(W16)--(B12); 

\draw[line width=0.06cm]  (W13)--(B13)--(W17)--(B17)--(W13); \draw[line width=0.06cm]  (B13)--(W14)--(B18)--(W17)--(B13); 
\draw[line width=0.06cm]  (W14)--(B14)--(W18)--(B18)--(W14); \draw[line width=0.06cm]  (B14)--(W15)--(B19)--(W18)--(B14); 
\draw[line width=0.06cm]  (W15)--(B15)--(W19)--(B19)--(W15); \draw[line width=0.06cm]  (B15)--(W16)--(B20)--(W19)--(B15); 
\draw[line width=0.06cm]  (W16)--(B16)--(W20)--(B20)--(W16); 

\draw[line width=0.06cm]  (B17)--(W17)--(B21)--(W21)--(B17); \draw[line width=0.06cm]  (W17)--(B18)--(W22)--(B21)--(W17);
\draw[line width=0.06cm]  (B18)--(W18)--(B22)--(W22)--(B18); \draw[line width=0.06cm]  (W18)--(B19)--(W23)--(B22)--(W18);
\draw[line width=0.06cm]  (B19)--(W19)--(B23)--(W23)--(B19); \draw[line width=0.06cm]  (W19)--(B20)--(W24)--(B23)--(W19);
\draw[line width=0.06cm]  (B20)--(W20)--(B24)--(W24)--(B20); 

\draw[line width=0.4cm,color=blue] (B1)--(W1); \draw[line width=0.4cm,color=blue] (B2)--(W2); 
\draw[line width=0.4cm,color=blue] (B3)--(W3); \draw[line width=0.4cm,color=blue] (B4)--(W4); 
\draw[line width=0.4cm,color=blue] (B5)--(W6); \draw[line width=0.4cm,color=blue] (B6)--(W7); 
\draw[line width=0.4cm,color=blue] (B7)--(W8); 
\draw[line width=0.4cm,color=blue] (B9)--(W9); \draw[line width=0.4cm,color=blue] (B10)--(W10); 
\draw[line width=0.4cm,color=blue] (B11)--(W11); \draw[line width=0.4cm,color=blue] (B12)--(W12); 
\draw[line width=0.4cm,color=blue] (B13)--(W14); \draw[line width=0.4cm,color=blue] (B14)--(W15); 
\draw[line width=0.4cm,color=blue] (B15)--(W16); \draw[line width=0.4cm,color=blue] (B17)--(W17); 
\draw[line width=0.4cm,color=blue] (B18)--(W18); \draw[line width=0.4cm,color=blue] (B19)--(W19); 
\draw[line width=0.4cm,color=blue] (B20)--(W20); 
\draw[line width=0.4cm,color=blue] (B21)--(W22); \draw[line width=0.4cm,color=blue] (B22)--(W23); 
\draw[line width=0.4cm,color=blue] (B23)--(W24);

\filldraw  [ultra thick, fill=black] (0,0) circle [radius=0.25] ; \filldraw  [ultra thick, fill=black] (4,0) circle [radius=0.25] ;
\filldraw  [ultra thick, fill=black] (8,0) circle [radius=0.25] ; \filldraw  [ultra thick, fill=black] (12,0) circle [radius=0.25] ; 
\filldraw  [ultra thick, fill=black] (2,2) circle [radius=0.25] ; \filldraw  [ultra thick, fill=black] (6,2) circle [radius=0.25] ;
\filldraw  [ultra thick, fill=black] (10,2) circle [radius=0.25] ; \filldraw  [ultra thick, fill=black] (14,2) circle [radius=0.25] ;

\filldraw  [ultra thick, fill=black] (0,4) circle [radius=0.25] ; \filldraw  [ultra thick, fill=black] (4,4) circle [radius=0.25] ;
\filldraw  [ultra thick, fill=black] (8,4) circle [radius=0.25] ; \filldraw  [ultra thick, fill=black] (12,4) circle [radius=0.25] ;
\filldraw  [ultra thick, fill=black] (2,6) circle [radius=0.25] ; \filldraw  [ultra thick, fill=black] (6,6) circle [radius=0.25] ;
\filldraw  [ultra thick, fill=black] (10,6) circle [radius=0.25] ; \filldraw  [ultra thick, fill=black] (14,6) circle [radius=0.25] ;

\filldraw  [ultra thick, fill=black] (0,8) circle [radius=0.25] ; \filldraw  [ultra thick, fill=black] (4,8) circle [radius=0.25] ;
\filldraw  [ultra thick, fill=black] (8,8) circle [radius=0.25] ; \filldraw  [ultra thick, fill=black] (12,8) circle [radius=0.25] ;
\filldraw  [ultra thick, fill=black] (2,10) circle [radius=0.25] ; \filldraw  [ultra thick, fill=black] (6,10) circle [radius=0.25] ;
\filldraw  [ultra thick, fill=black] (10,10) circle [radius=0.25] ; \filldraw  [ultra thick, fill=black] (14,10) circle [radius=0.25] ;

\filldraw  [ultra thick, fill=white] (2,0) circle [radius=0.25] ; \filldraw  [ultra thick, fill=white] (6,0) circle [radius=0.25] ;
\filldraw  [ultra thick, fill=white] (10,0) circle [radius=0.25] ; \filldraw  [ultra thick, fill=white] (14,0) circle [radius=0.25] ;
\filldraw  [ultra thick, fill=white] (0,2) circle [radius=0.25] ; \filldraw  [ultra thick, fill=white] (4,2) circle [radius=0.25] ;
\filldraw  [ultra thick, fill=white] (8,2) circle [radius=0.25] ; \filldraw  [ultra thick, fill=white] (12,2) circle [radius=0.25] ;

\filldraw  [ultra thick, fill=white] (2,4) circle [radius=0.25] ; \filldraw  [ultra thick, fill=white] (6,4) circle [radius=0.25] ;
\filldraw  [ultra thick, fill=white] (10,4) circle [radius=0.25] ; \filldraw  [ultra thick, fill=white] (14,4) circle [radius=0.25] ;
\filldraw  [ultra thick, fill=white] (0,6) circle [radius=0.25] ; \filldraw  [ultra thick, fill=white] (4,6) circle [radius=0.25] ;
\filldraw  [ultra thick, fill=white] (8,6) circle [radius=0.25] ; \filldraw  [ultra thick, fill=white] (12,6) circle [radius=0.25] ;

\filldraw  [ultra thick, fill=white] (2,8) circle [radius=0.25] ; \filldraw  [ultra thick, fill=white] (6,8) circle [radius=0.25] ;
\filldraw  [ultra thick, fill=white] (10,8) circle [radius=0.25] ; \filldraw  [ultra thick, fill=white] (14,8) circle [radius=0.25] ;
\filldraw  [ultra thick, fill=white] (0,10) circle [radius=0.25] ; \filldraw  [ultra thick, fill=white] (4,10) circle [radius=0.25] ;
\filldraw  [ultra thick, fill=white] (8,10) circle [radius=0.25] ; \filldraw  [ultra thick, fill=white] (12,10) circle [radius=0.25] ;
\end{tikzpicture}
} }; 

\node (PM4) at (6,-4.2) 
{\scalebox{0.3}{
\begin{tikzpicture}
\node (B1) at (0,0){$$}; \node (B2) at (4,0){$$}; \node (B3) at (8,0){$$}; \node (B4) at (12,0){$$}; 
\node (B5) at (2,2){$$}; \node (B6) at (6,2){$$}; \node (B7) at (10,2){$$}; \node (B8) at (14,2){$$}; 
\node (B9) at (0,4){$$}; \node (B10) at (4,4){$$}; \node (B11) at (8,4){$$}; \node (B12) at (12,4){$$}; 
\node (B13) at (2,6){$$}; \node (B14) at (6,6){$$}; \node (B15) at (10,6){$$}; \node (B16) at (14,6){$$};
\node (B17) at (0,8){$$}; \node (B18) at (4,8){$$}; \node (B19) at (8,8){$$}; \node (B20) at (12,8){$$}; 
\node (B21) at (2,10){$$}; \node (B22) at (6,10){$$}; \node (B23) at (10,10){$$}; \node (B24) at (14,10){$$};

\node (W1) at (2,0){$$}; \node (W2) at (6,0){$$}; \node (W3) at (10,0){$$}; \node (W4) at (14,0){$$}; 
\node (W5) at (0,2){$$}; \node (W6) at (4,2){$$}; \node (W7) at (8,2){$$}; \node (W8) at (12,2){$$};
\node (W9) at (2,4){$$}; \node (W10) at (6,4){$$}; \node (W11) at (10,4){$$}; \node (W12) at (14,4){$$}; 
\node (W13) at (0,6){$$}; \node (W14) at (4,6){$$}; \node (W15) at (8,6){$$}; \node (W16) at (12,6){$$};
\node (W17) at (2,8){$$}; \node (W18) at (6,8){$$}; \node (W19) at (10,8){$$}; \node (W20) at (14,8){$$}; 
\node (W21) at (0,10){$$}; \node (W22) at (4,10){$$}; \node (W23) at (8,10){$$}; \node (W24) at (12,10){$$};

\draw[line width=0.06cm]  (B1)--(W1)--(B5)--(W5)--(B1); \draw[line width=0.06cm]  (W1)--(B2)--(W6)--(B5)--(W1);
\draw[line width=0.06cm]  (B2)--(W2)--(B6)--(W6)--(B2); \draw[line width=0.06cm]  (W2)--(B3)--(W7)--(B6)--(W2);
\draw[line width=0.06cm]  (B3)--(W3)--(B7)--(W7)--(B3); \draw[line width=0.06cm]  (W3)--(B4)--(W8)--(B7)--(W3);
\draw[line width=0.06cm]  (B4)--(W4)--(B8)--(W8)--(B4); 

\draw[line width=0.06cm]  (W5)--(B5)--(W9)--(B9)--(W5); \draw[line width=0.06cm]  (B5)--(W6)--(B10)--(W9)--(B5); 
\draw[line width=0.06cm]  (W6)--(B6)--(W10)--(B10)--(W6); \draw[line width=0.06cm]  (B6)--(W7)--(B11)--(W10)--(B6); 
\draw[line width=0.06cm]  (W7)--(B7)--(W11)--(B11)--(W7); \draw[line width=0.06cm]  (B7)--(W8)--(B12)--(W11)--(B7); 
\draw[line width=0.06cm]  (W8)--(B8)--(W12)--(B12)--(W8); 

\draw[line width=0.06cm]  (B9)--(W9)--(B13)--(W13)--(B9); \draw[line width=0.06cm]  (W9)--(B10)--(W14)--(B13)--(W9);
\draw[line width=0.06cm]  (B10)--(W10)--(B14)--(W14)--(B10); \draw[line width=0.06cm]  (W10)--(B11)--(W15)--(B14)--(W10);
\draw[line width=0.06cm]  (B11)--(W11)--(B15)--(W15)--(B11); \draw[line width=0.06cm]  (W11)--(B12)--(W16)--(B15)--(W11);
\draw[line width=0.06cm]  (B12)--(W12)--(B16)--(W16)--(B12); 

\draw[line width=0.06cm]  (W13)--(B13)--(W17)--(B17)--(W13); \draw[line width=0.06cm]  (B13)--(W14)--(B18)--(W17)--(B13); 
\draw[line width=0.06cm]  (W14)--(B14)--(W18)--(B18)--(W14); \draw[line width=0.06cm]  (B14)--(W15)--(B19)--(W18)--(B14); 
\draw[line width=0.06cm]  (W15)--(B15)--(W19)--(B19)--(W15); \draw[line width=0.06cm]  (B15)--(W16)--(B20)--(W19)--(B15); 
\draw[line width=0.06cm]  (W16)--(B16)--(W20)--(B20)--(W16); 

\draw[line width=0.06cm]  (B17)--(W17)--(B21)--(W21)--(B17); \draw[line width=0.06cm]  (W17)--(B18)--(W22)--(B21)--(W17);
\draw[line width=0.06cm]  (B18)--(W18)--(B22)--(W22)--(B18); \draw[line width=0.06cm]  (W18)--(B19)--(W23)--(B22)--(W18);
\draw[line width=0.06cm]  (B19)--(W19)--(B23)--(W23)--(B19); \draw[line width=0.06cm]  (W19)--(B20)--(W24)--(B23)--(W19);
\draw[line width=0.06cm]  (B20)--(W20)--(B24)--(W24)--(B20); 

\draw[line width=0.4cm,color=blue] (W5)--(B1); \draw[line width=0.4cm,color=blue] (W6)--(B2); \draw[line width=0.4cm,color=blue] (W7)--(B3); 
\draw[line width=0.4cm,color=blue] (W8)--(B4); \draw[line width=0.4cm,color=blue] (W9)--(B5); \draw[line width=0.4cm,color=blue] (W10)--(B6); 
\draw[line width=0.4cm,color=blue] (W11)--(B7); \draw[line width=0.4cm,color=blue] (W12)--(B8); \draw[line width=0.4cm,color=blue] (W13)--(B9); 
\draw[line width=0.4cm,color=blue] (W14)--(B10); \draw[line width=0.4cm,color=blue] (W15)--(B11); \draw[line width=0.4cm,color=blue] (W16)--(B12); 
\draw[line width=0.4cm,color=blue] (W17)--(B13); \draw[line width=0.4cm,color=blue] (W18)--(B14); \draw[line width=0.4cm,color=blue] (W19)--(B15); 
\draw[line width=0.4cm,color=blue] (W20)--(B16); \draw[line width=0.4cm,color=blue] (W21)--(B17); \draw[line width=0.4cm,color=blue] (W22)--(B18); 
\draw[line width=0.4cm,color=blue] (W23)--(B19); \draw[line width=0.4cm,color=blue] (W24)--(B20); 

\filldraw  [ultra thick, fill=black] (0,0) circle [radius=0.25] ; \filldraw  [ultra thick, fill=black] (4,0) circle [radius=0.25] ;
\filldraw  [ultra thick, fill=black] (8,0) circle [radius=0.25] ; \filldraw  [ultra thick, fill=black] (12,0) circle [radius=0.25] ; 
\filldraw  [ultra thick, fill=black] (2,2) circle [radius=0.25] ; \filldraw  [ultra thick, fill=black] (6,2) circle [radius=0.25] ;
\filldraw  [ultra thick, fill=black] (10,2) circle [radius=0.25] ; \filldraw  [ultra thick, fill=black] (14,2) circle [radius=0.25] ;

\filldraw  [ultra thick, fill=black] (0,4) circle [radius=0.25] ; \filldraw  [ultra thick, fill=black] (4,4) circle [radius=0.25] ;
\filldraw  [ultra thick, fill=black] (8,4) circle [radius=0.25] ; \filldraw  [ultra thick, fill=black] (12,4) circle [radius=0.25] ;
\filldraw  [ultra thick, fill=black] (2,6) circle [radius=0.25] ; \filldraw  [ultra thick, fill=black] (6,6) circle [radius=0.25] ;
\filldraw  [ultra thick, fill=black] (10,6) circle [radius=0.25] ; \filldraw  [ultra thick, fill=black] (14,6) circle [radius=0.25] ;

\filldraw  [ultra thick, fill=black] (0,8) circle [radius=0.25] ; \filldraw  [ultra thick, fill=black] (4,8) circle [radius=0.25] ;
\filldraw  [ultra thick, fill=black] (8,8) circle [radius=0.25] ; \filldraw  [ultra thick, fill=black] (12,8) circle [radius=0.25] ;
\filldraw  [ultra thick, fill=black] (2,10) circle [radius=0.25] ; \filldraw  [ultra thick, fill=black] (6,10) circle [radius=0.25] ;
\filldraw  [ultra thick, fill=black] (10,10) circle [radius=0.25] ; \filldraw  [ultra thick, fill=black] (14,10) circle [radius=0.25] ;

\filldraw  [ultra thick, fill=white] (2,0) circle [radius=0.25] ; \filldraw  [ultra thick, fill=white] (6,0) circle [radius=0.25] ;
\filldraw  [ultra thick, fill=white] (10,0) circle [radius=0.25] ; \filldraw  [ultra thick, fill=white] (14,0) circle [radius=0.25] ;
\filldraw  [ultra thick, fill=white] (0,2) circle [radius=0.25] ; \filldraw  [ultra thick, fill=white] (4,2) circle [radius=0.25] ;
\filldraw  [ultra thick, fill=white] (8,2) circle [radius=0.25] ; \filldraw  [ultra thick, fill=white] (12,2) circle [radius=0.25] ;

\filldraw  [ultra thick, fill=white] (2,4) circle [radius=0.25] ; \filldraw  [ultra thick, fill=white] (6,4) circle [radius=0.25] ;
\filldraw  [ultra thick, fill=white] (10,4) circle [radius=0.25] ; \filldraw  [ultra thick, fill=white] (14,4) circle [radius=0.25] ;
\filldraw  [ultra thick, fill=white] (0,6) circle [radius=0.25] ; \filldraw  [ultra thick, fill=white] (4,6) circle [radius=0.25] ;
\filldraw  [ultra thick, fill=white] (8,6) circle [radius=0.25] ; \filldraw  [ultra thick, fill=white] (12,6) circle [radius=0.25] ;

\filldraw  [ultra thick, fill=white] (2,8) circle [radius=0.25] ; \filldraw  [ultra thick, fill=white] (6,8) circle [radius=0.25] ;
\filldraw  [ultra thick, fill=white] (10,8) circle [radius=0.25] ; \filldraw  [ultra thick, fill=white] (14,8) circle [radius=0.25] ;
\filldraw  [ultra thick, fill=white] (0,10) circle [radius=0.25] ; \filldraw  [ultra thick, fill=white] (4,10) circle [radius=0.25] ;
\filldraw  [ultra thick, fill=white] (8,10) circle [radius=0.25] ; \filldraw  [ultra thick, fill=white] (12,10) circle [radius=0.25] ;
\end{tikzpicture}
} }; 

\end{tikzpicture}
}}
\caption{Extremal perfect matchings of a square dimer model}
\label{expm_square}
\end{center}
\end{figure}
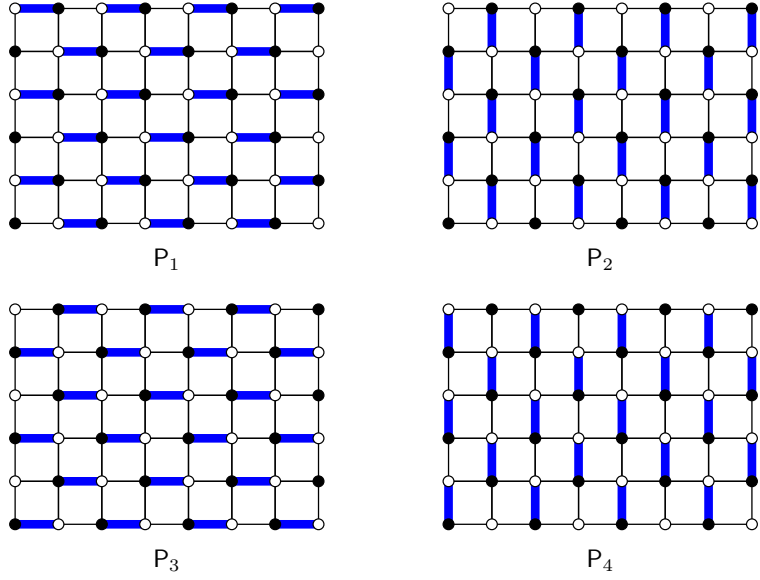

\begin{figure}[h!]
\begin{center}
\begin{tabular}{c}
\begin{minipage}{0.45\hsize}
\begin{center}
{\scalebox{0.35}{
\begin{tikzpicture}
\node (B1) at (0,0){$$}; \node (B2) at (4,0){$$}; \node (B3) at (8,0){$$}; \node (B4) at (12,0){$$}; 
\node (B5) at (2,2){$$}; \node (B6) at (6,2){$$}; \node (B7) at (10,2){$$}; \node (B8) at (14,2){$$}; 
\node (B9) at (0,4){$$}; \node (B10) at (4,4){$$}; \node (B11) at (8,4){$$}; \node (B12) at (12,4){$$}; 
\node (B13) at (2,6){$$}; \node (B14) at (6,6){$$}; \node (B15) at (10,6){$$}; \node (B16) at (14,6){$$};
\node (B17) at (0,8){$$}; \node (B18) at (4,8){$$}; \node (B19) at (8,8){$$}; \node (B20) at (12,8){$$}; 
\node (B21) at (2,10){$$}; \node (B22) at (6,10){$$}; \node (B23) at (10,10){$$}; \node (B24) at (14,10){$$};

\node (W1) at (2,0){$$}; \node (W2) at (6,0){$$}; \node (W3) at (10,0){$$}; \node (W4) at (14,0){$$}; 
\node (W5) at (0,2){$$}; \node (W6) at (4,2){$$}; \node (W7) at (8,2){$$}; \node (W8) at (12,2){$$};
\node (W9) at (2,4){$$}; \node (W10) at (6,4){$$}; \node (W11) at (10,4){$$}; \node (W12) at (14,4){$$}; 
\node (W13) at (0,6){$$}; \node (W14) at (4,6){$$}; \node (W15) at (8,6){$$}; \node (W16) at (12,6){$$};
\node (W17) at (2,8){$$}; \node (W18) at (6,8){$$}; \node (W19) at (10,8){$$}; \node (W20) at (14,8){$$}; 
\node (W21) at (0,10){$$}; \node (W22) at (4,10){$$}; \node (W23) at (8,10){$$}; \node (W24) at (12,10){$$};

\draw [fill=lightgray] (B1) rectangle (B5); \draw [fill=lightgray] (B2) rectangle (B6); \draw [fill=lightgray] (B3) rectangle (B7); \draw [fill=lightgray] (B4) rectangle (B8); 
\draw [fill=lightgray] (B5) rectangle (B10); \draw [fill=lightgray] (B6) rectangle (B11); \draw [fill=lightgray] (B7) rectangle (B12); 
\draw [fill=lightgray] (B9) rectangle (B13); \draw [fill=lightgray] (B10) rectangle (B14); \draw [fill=lightgray] (B11) rectangle (B15); \draw [fill=lightgray] (B12) rectangle (B16); 
\draw [fill=lightgray] (B13) rectangle (B18); \draw [fill=lightgray] (B14) rectangle (B19); \draw [fill=lightgray] (B15) rectangle (B20); 
\draw [fill=lightgray] (B17) rectangle (B21); \draw [fill=lightgray] (B18) rectangle (B22); \draw [fill=lightgray] (B19) rectangle (B23); 
\draw [fill=lightgray] (B20) rectangle (B24);  

\draw[line width=0.06cm]  (B1)--(W1)--(B5)--(W5)--(B1); \draw[line width=0.06cm]  (W1)--(B2)--(W6)--(B5)--(W1);
\draw[line width=0.06cm]  (B2)--(W2)--(B6)--(W6)--(B2); \draw[line width=0.06cm]  (W2)--(B3)--(W7)--(B6)--(W2);
\draw[line width=0.06cm]  (B3)--(W3)--(B7)--(W7)--(B3); \draw[line width=0.06cm]  (W3)--(B4)--(W8)--(B7)--(W3);
\draw[line width=0.06cm]  (B4)--(W4)--(B8)--(W8)--(B4); 

\draw[line width=0.06cm]  (W5)--(B5)--(W9)--(B9)--(W5); \draw[line width=0.06cm]  (B5)--(W6)--(B10)--(W9)--(B5); 
\draw[line width=0.06cm]  (W6)--(B6)--(W10)--(B10)--(W6); \draw[line width=0.06cm]  (B6)--(W7)--(B11)--(W10)--(B6); 
\draw[line width=0.06cm]  (W7)--(B7)--(W11)--(B11)--(W7); \draw[line width=0.06cm]  (B7)--(W8)--(B12)--(W11)--(B7); 
\draw[line width=0.06cm]  (W8)--(B8)--(W12)--(B12)--(W8); 

\draw[line width=0.06cm]  (B9)--(W9)--(B13)--(W13)--(B9); \draw[line width=0.06cm]  (W9)--(B10)--(W14)--(B13)--(W9);
\draw[line width=0.06cm]  (B10)--(W10)--(B14)--(W14)--(B10); \draw[line width=0.06cm]  (W10)--(B11)--(W15)--(B14)--(W10);
\draw[line width=0.06cm]  (B11)--(W11)--(B15)--(W15)--(B11); \draw[line width=0.06cm]  (W11)--(B12)--(W16)--(B15)--(W11);
\draw[line width=0.06cm]  (B12)--(W12)--(B16)--(W16)--(B12); 

\draw[line width=0.06cm]  (W13)--(B13)--(W17)--(B17)--(W13); \draw[line width=0.06cm]  (B13)--(W14)--(B18)--(W17)--(B13); 
\draw[line width=0.06cm]  (W14)--(B14)--(W18)--(B18)--(W14); \draw[line width=0.06cm]  (B14)--(W15)--(B19)--(W18)--(B14); 
\draw[line width=0.06cm]  (W15)--(B15)--(W19)--(B19)--(W15); \draw[line width=0.06cm]  (B15)--(W16)--(B20)--(W19)--(B15); 
\draw[line width=0.06cm]  (W16)--(B16)--(W20)--(B20)--(W16); 

\draw[line width=0.06cm]  (B17)--(W17)--(B21)--(W21)--(B17); \draw[line width=0.06cm]  (W17)--(B18)--(W22)--(B21)--(W17);
\draw[line width=0.06cm]  (B18)--(W18)--(B22)--(W22)--(B18); \draw[line width=0.06cm]  (W18)--(B19)--(W23)--(B22)--(W18);
\draw[line width=0.06cm]  (B19)--(W19)--(B23)--(W23)--(B19); \draw[line width=0.06cm]  (W19)--(B20)--(W24)--(B23)--(W19);
\draw[line width=0.06cm]  (B20)--(W20)--(B24)--(W24)--(B20); 

\filldraw  [ultra thick, fill=black] (0,0) circle [radius=0.25] ; \filldraw  [ultra thick, fill=black] (4,0) circle [radius=0.25] ;
\filldraw  [ultra thick, fill=black] (8,0) circle [radius=0.25] ; \filldraw  [ultra thick, fill=black] (12,0) circle [radius=0.25] ; 
\filldraw  [ultra thick, fill=black] (2,2) circle [radius=0.25] ; \filldraw  [ultra thick, fill=black] (6,2) circle [radius=0.25] ;
\filldraw  [ultra thick, fill=black] (10,2) circle [radius=0.25] ; \filldraw  [ultra thick, fill=black] (14,2) circle [radius=0.25] ;

\filldraw  [ultra thick, fill=black] (0,4) circle [radius=0.25] ; \filldraw  [ultra thick, fill=black] (4,4) circle [radius=0.25] ;
\filldraw  [ultra thick, fill=black] (8,4) circle [radius=0.25] ; \filldraw  [ultra thick, fill=black] (12,4) circle [radius=0.25] ;
\filldraw  [ultra thick, fill=black] (2,6) circle [radius=0.25] ; \filldraw  [ultra thick, fill=black] (6,6) circle [radius=0.25] ;
\filldraw  [ultra thick, fill=black] (10,6) circle [radius=0.25] ; \filldraw  [ultra thick, fill=black] (14,6) circle [radius=0.25] ;

\filldraw  [ultra thick, fill=black] (0,8) circle [radius=0.25] ; \filldraw  [ultra thick, fill=black] (4,8) circle [radius=0.25] ;
\filldraw  [ultra thick, fill=black] (8,8) circle [radius=0.25] ; \filldraw  [ultra thick, fill=black] (12,8) circle [radius=0.25] ;
\filldraw  [ultra thick, fill=black] (2,10) circle [radius=0.25] ; \filldraw  [ultra thick, fill=black] (6,10) circle [radius=0.25] ;
\filldraw  [ultra thick, fill=black] (10,10) circle [radius=0.25] ; \filldraw  [ultra thick, fill=black] (14,10) circle [radius=0.25] ;

\filldraw  [ultra thick, fill=white] (2,0) circle [radius=0.25] ; \filldraw  [ultra thick, fill=white] (6,0) circle [radius=0.25] ;
\filldraw  [ultra thick, fill=white] (10,0) circle [radius=0.25] ; \filldraw  [ultra thick, fill=white] (14,0) circle [radius=0.25] ;
\filldraw  [ultra thick, fill=white] (0,2) circle [radius=0.25] ; \filldraw  [ultra thick, fill=white] (4,2) circle [radius=0.25] ;
\filldraw  [ultra thick, fill=white] (8,2) circle [radius=0.25] ; \filldraw  [ultra thick, fill=white] (12,2) circle [radius=0.25] ;

\filldraw  [ultra thick, fill=white] (2,4) circle [radius=0.25] ; \filldraw  [ultra thick, fill=white] (6,4) circle [radius=0.25] ;
\filldraw  [ultra thick, fill=white] (10,4) circle [radius=0.25] ; \filldraw  [ultra thick, fill=white] (14,4) circle [radius=0.25] ;
\filldraw  [ultra thick, fill=white] (0,6) circle [radius=0.25] ; \filldraw  [ultra thick, fill=white] (4,6) circle [radius=0.25] ;
\filldraw  [ultra thick, fill=white] (8,6) circle [radius=0.25] ; \filldraw  [ultra thick, fill=white] (12,6) circle [radius=0.25] ;

\filldraw  [ultra thick, fill=white] (2,8) circle [radius=0.25] ; \filldraw  [ultra thick, fill=white] (6,8) circle [radius=0.25] ;
\filldraw  [ultra thick, fill=white] (10,8) circle [radius=0.25] ; \filldraw  [ultra thick, fill=white] (14,8) circle [radius=0.25] ;
\filldraw  [ultra thick, fill=white] (0,10) circle [radius=0.25] ; \filldraw  [ultra thick, fill=white] (4,10) circle [radius=0.25] ;
\filldraw  [ultra thick, fill=white] (8,10) circle [radius=0.25] ; \filldraw  [ultra thick, fill=white] (12,10) circle [radius=0.25] ;

\end{tikzpicture} }}
\end{center}
\caption{}
\label{faces_semisteady}
\end{minipage}

\begin{minipage}{0.45\hsize}
\begin{center}
{\scalebox{0.35}{
\begin{tikzpicture}
\node (B1) at (0,0){$$}; \node (B2) at (4,0){$$}; \node (B3) at (8,0){$$}; 
\node (B5) at (2,2){$$}; \node (B6) at (6,2){$$}; \node (B7) at (10,2){$$}; 
\node (B9) at (0,4){$$}; \node (B10) at (4,4){$$}; \node (B11) at (8,4){$$}; 
\node (B13) at (2,6){$$}; \node (B14) at (6,6){$$}; \node (B15) at (10,6){$$}; 
\node (B17) at (0,8){$$}; \node (B18) at (4,8){$$}; \node (B19) at (8,8){$$}; 
\node (B21) at (2,10){$$}; \node (B22) at (6,10){$$}; \node (B23) at (10,10){$$}; 

\node (W1) at (2,0){$$}; \node (W2) at (6,0){$$}; \node (W3) at (10,0){$$};
\node (W5) at (0,2){$$}; \node (W6) at (4,2){$$}; \node (W7) at (8,2){$$}; 
\node (W9) at (2,4){$$}; \node (W10) at (6,4){$$}; \node (W11) at (10,4){$$}; 
\node (W13) at (0,6){$$}; \node (W14) at (4,6){$$}; \node (W15) at (8,6){$$}; 
\node (W17) at (2,8){$$}; \node (W18) at (6,8){$$}; \node (W19) at (10,8){$$}; 
\node (W21) at (0,10){$$}; \node (W22) at (4,10){$$}; \node (W23) at (8,10){$$}; 

\draw [fill=lightgray] (B1) rectangle (B5); \draw [fill=lightgray] (B2) rectangle (B6); \draw [fill=lightgray] (B3) rectangle (B7); 
\draw [fill=lightgray] (B5) rectangle (B10); \draw [fill=lightgray] (B6) rectangle (B11); 
\draw [fill=lightgray] (B9) rectangle (B13); \draw [fill=lightgray] (B10) rectangle (B14); \draw [fill=lightgray] (B11) rectangle (B15); 
\draw [fill=lightgray] (B13) rectangle (B18); \draw [fill=lightgray] (B14) rectangle (B19); 
\draw [fill=lightgray] (B17) rectangle (B21); \draw [fill=lightgray] (B18) rectangle (B22); \draw [fill=lightgray] (B19) rectangle (B23); 

\draw[line width=0.06cm]  (B1)--(W1)--(B5)--(W5)--(B1); \draw[line width=0.06cm]  (W1)--(B2)--(W6)--(B5)--(W1);
\draw[line width=0.06cm]  (B2)--(W2)--(B6)--(W6)--(B2); \draw[line width=0.06cm]  (W2)--(B3)--(W7)--(B6)--(W2);
\draw[line width=0.06cm]  (B3)--(W3)--(B7)--(W7)--(B3); 

\draw[line width=0.06cm]  (W5)--(B5)--(W9)--(B9)--(W5); \draw[line width=0.06cm]  (B5)--(W6)--(B10)--(W9)--(B5); 
\draw[line width=0.06cm]  (W6)--(B6)--(W10)--(B10)--(W6); \draw[line width=0.06cm]  (B6)--(W7)--(B11)--(W10)--(B6); 
\draw[line width=0.06cm]  (W7)--(B7)--(W11)--(B11)--(W7); 

\draw[line width=0.06cm]  (B9)--(W9)--(B13)--(W13)--(B9); \draw[line width=0.06cm]  (W9)--(B10)--(W14)--(B13)--(W9);
\draw[line width=0.06cm]  (B10)--(W10)--(B14)--(W14)--(B10); \draw[line width=0.06cm]  (W10)--(B11)--(W15)--(B14)--(W10);
\draw[line width=0.06cm]  (B11)--(W11)--(B15)--(W15)--(B11); 

\draw[line width=0.06cm]  (W13)--(B13)--(W17)--(B17)--(W13); \draw[line width=0.06cm]  (B13)--(W14)--(B18)--(W17)--(B13); 
\draw[line width=0.06cm]  (W14)--(B14)--(W18)--(B18)--(W14); \draw[line width=0.06cm]  (B14)--(W15)--(B19)--(W18)--(B14); 
\draw[line width=0.06cm]  (W15)--(B15)--(W19)--(B19)--(W15); 

\draw[line width=0.06cm]  (B17)--(W17)--(B21)--(W21)--(B17); \draw[line width=0.06cm]  (W17)--(B18)--(W22)--(B21)--(W17);
\draw[line width=0.06cm]  (B18)--(W18)--(B22)--(W22)--(B18); \draw[line width=0.06cm]  (W18)--(B19)--(W23)--(B22)--(W18);
\draw[line width=0.06cm]  (B19)--(W19)--(B23)--(W23)--(B19); 

\filldraw  [ultra thick, fill=black] (0,0) circle [radius=0.25] ; \filldraw  [ultra thick, fill=black] (4,0) circle [radius=0.25] ;
\filldraw  [ultra thick, fill=black] (8,0) circle [radius=0.25] ; 
\filldraw  [ultra thick, fill=black] (2,2) circle [radius=0.25] ; \filldraw  [ultra thick, fill=black] (6,2) circle [radius=0.25] ;
\filldraw  [ultra thick, fill=black] (10,2) circle [radius=0.25] ; 

\filldraw  [ultra thick, fill=black] (0,4) circle [radius=0.25] ; \filldraw  [ultra thick, fill=black] (4,4) circle [radius=0.25] ;
\filldraw  [ultra thick, fill=black] (8,4) circle [radius=0.25] ; 
\filldraw  [ultra thick, fill=black] (2,6) circle [radius=0.25] ; \filldraw  [ultra thick, fill=black] (6,6) circle [radius=0.25] ;
\filldraw  [ultra thick, fill=black] (10,6) circle [radius=0.25] ; 

\filldraw  [ultra thick, fill=black] (0,8) circle [radius=0.25] ; \filldraw  [ultra thick, fill=black] (4,8) circle [radius=0.25] ;
\filldraw  [ultra thick, fill=black] (8,8) circle [radius=0.25] ; 
\filldraw  [ultra thick, fill=black] (2,10) circle [radius=0.25] ; \filldraw  [ultra thick, fill=black] (6,10) circle [radius=0.25] ;
\filldraw  [ultra thick, fill=black] (10,10) circle [radius=0.25] ; 

\filldraw  [ultra thick, fill=white] (2,0) circle [radius=0.25] ; \filldraw  [ultra thick, fill=white] (6,0) circle [radius=0.25] ;
\filldraw  [ultra thick, fill=white] (10,0) circle [radius=0.25] ; 
\filldraw  [ultra thick, fill=white] (0,2) circle [radius=0.25] ; \filldraw  [ultra thick, fill=white] (4,2) circle [radius=0.25] ;
\filldraw  [ultra thick, fill=white] (8,2) circle [radius=0.25] ; 

\filldraw  [ultra thick, fill=white] (2,4) circle [radius=0.25] ; \filldraw  [ultra thick, fill=white] (6,4) circle [radius=0.25] ;
\filldraw  [ultra thick, fill=white] (10,4) circle [radius=0.25] ; 
\filldraw  [ultra thick, fill=white] (0,6) circle [radius=0.25] ; \filldraw  [ultra thick, fill=white] (4,6) circle [radius=0.25] ;
\filldraw  [ultra thick, fill=white] (8,6) circle [radius=0.25] ; 

\filldraw  [ultra thick, fill=white] (2,8) circle [radius=0.25] ; \filldraw  [ultra thick, fill=white] (6,8) circle [radius=0.25] ;
\filldraw  [ultra thick, fill=white] (10,8) circle [radius=0.25] ; 
\filldraw  [ultra thick, fill=white] (0,10) circle [radius=0.25] ; \filldraw  [ultra thick, fill=white] (4,10) circle [radius=0.25] ;
\filldraw  [ultra thick, fill=white] (8,10) circle [radius=0.25] ; 

\node at (1,9.5) {{\huge$k$}}; 
\coordinate (pa1) at (1,9) ;
\coordinate (pa2) at (3,9); \coordinate (pa3) at (3,7); \coordinate (pa4) at (5,7); 
\coordinate (pa5) at (5,5); \coordinate (pa6) at (3,5); 
\node (pa7) at (3,3) {{\huge$j$}}; 

\node at (3,9.5) {{\huge$k^\prime$}}; 
\coordinate (pa1+) at (3,9) ; 
\coordinate (pa2+) at (5,9); \coordinate (pa3+) at (5,7); \coordinate (pa4+) at (7,7); 
\coordinate (pa5+) at (7,5); \coordinate (pa6+) at (5,5); 
\node (pa7+) at (5,3) {{\huge$j^\prime$}}; 

\draw[->, line width=0.15cm, rounded corners, color=red] (pa1)--(pa2)--(pa3)--(pa4)--(pa5)--(pa6)--(pa7) ;
\draw[->, line width=0.15cm, rounded corners, color=blue] (pa1+)--(pa2+)--(pa3+)--(pa4+)--(pa5+)--(pa6+)--(pa7+) ;

\end{tikzpicture} }}
\end{center}
\caption{}
\label{faces_semisteady_path}
\end{minipage}
\end{tabular}
\end{center}
\end{figure}

Now, we fix a vertex $k\in Q_0$, and let $M\coloneqq e_k\calP(Q, W_Q)\cong\bigoplus_{j\in Q_0}T_{kj}$. 
In the following, we show $\Hom_R(T_{ki},M)\cong M$ or $M^*$ for any $i\in Q_0$, and this means $M$ is semi-steady. 
We divide faces of a dimer model into gray faces and white faces as shown in Figure~\ref{faces_semisteady}, 
then vertices of $Q$ are also divided into two parts. 
We denote by $Q_0^{\rmg}$ (resp. $Q_0^{\rmw}$) the subset of $Q_0$ consisting of vertices corresponding to gray (resp. white) faces. 
We assume that the fixed vertex $k\in Q_0$ is in $Q_0^{\rmg}$, and fix a path $a_{kj}$ starting from $k\in Q_0^{\rmg}$ to $j\in Q_0$ for all $j$. 
(Note that module $T_{kj}$ does not depend on a choice of $a_{kj}$.) 
By the form of extremal perfect matchings, it is easy to see that for any $\ell\in Q_0^{\rmg}$ we can find a path starting from 
$\ell$ that evaluates to the same perfect matching function as $a_{kj}$. 
Therefore, we have that $M=\bigoplus_{j\in Q_0}T_{kj}\cong\bigoplus_{j\in Q_0}T_{\ell j}$ for any $\ell\in Q_0^{\rmg}$. 
In order to investigate a module $\bigoplus_{j\in Q_0}T_{\ell j}$ with $\ell\in Q_0^{\rmw}$, 
we again consider the fixed vertex $k\in Q_0^{\rmg}$ and paths $a_{kj}$. 
Then, we shift the vertex $k\in Q_0^{\rmg}$ to the right adjacent vertex. 
The shifted vertex is in $Q_0^{\rmw}$, and denote it by $k^\prime\in Q_0^{\rmw}$. 
In addition, we denote by $b_{k^\prime j^\prime}$ the path shifted from $a_{kj}$ for any $j$. 
(Figure~\ref{faces_semisteady_path} is an example of paths $a_{kj}$ and $b_{k^\prime j^\prime}$.) 
In particular, $b_{k^\prime j^\prime}$'s are paths on $Q^{\rm op}$, hence we have that 
\[
e_{k^\prime}\calP(Q,W_Q)\cong e_{k^\prime}\calP(Q^{\rm op},W_{Q^{\rm op}})
\cong\bigoplus_{j^\prime\in Q_0^{\rm op}} T(\sfP^{\rm op}_1(b_{k^\prime j^\prime}),\cdots,\sfP^{\rm op}_4(b_{k^\prime j^\prime})), 
\]
where $\sfP^{\rm op}_1,\cdots,\sfP^{\rm op}_4$ are extremal perfect matchings on $Q^{\rm op}$ 
corresponding to vertices $u_1,\cdots,u_4$ of $\Delta$. 
Thus, we have that 
\[
\bigoplus_{j\in Q_0}T_{k^\prime j}\cong e_{k^\prime}\calP(Q,W_Q)\cong\bigoplus_{j\in Q_0}T(\sfP_1(a_{kj}),\cdots,\sfP_4(a_{kj}))^*\cong M^* 
\]
by Lemma~\ref{key_lem} below. 
By the same argument used in the case of $Q_0^{\rmg}$, we have that $M^*\cong\bigoplus_{j\in Q_0}T_{k^\prime j}\cong\bigoplus_{j\in Q_0}T_{\ell j}$ 
for any $\ell\in Q_0^{\rmw}$. 
Consequently, we see that $M$ is semi-steady. 
Furthermore, since $\Gamma$ is not a regular hexagonal dimer model, this is not steady by Theorem~\ref{dimer}, and hence we have the desired conclusion. 

\medskip

Finally, we will show $(3)$$\Rightarrow$$(1)$. 
Let $\Gamma$ be an isoradial dimer model giving a semi-steady NCCR of $R$ that is not steady. 
By combining Theorem~\ref{class_semi_steady} and Lemma~\ref{cl_toric}, we see that the toric diagram of $R$ is a quadrangle. 
By the correspondence in Proposition~\ref{zigzag_sidepolygon}, there are four slopes $[z_1], [z_2] ,[z_3], [z_4]$ of zigzag paths corresponding to side segments of the toric diagram of $R$, and we suppose that these are ordered cyclically with this order. 
Let $\calZ_i$ be the set of zigzag paths having the same slope $[z_i]$ for $i=1,2,3,4$. 
Here, we recall that zigzag paths having the same slope arise as the difference of two extremal perfect matchings that are adjacent (see Proposition~\ref{zigzag_sidepolygon}). 
Thus, let $\sfP_1,\sfP_2, \sfP_3, \sfP_4$ be extremal perfect matchings, and suppose that zigzag paths in $\calZ_i$ can be obtained as the difference $\sfP_i-\sfP_{i-1}$ where $\sfP_0\coloneqq\sfP_4$. 
Since $\Gamma$ is isoradial, it is properly ordered, and hence slopes of zigzag paths factoring through the same node of $\Gamma$ differ from each other. 
Therefore, the number of edges incident to the same node is $3$ or $4$. (Note that $\Gamma$ does not have bivalent nodes.) 

We now assume that the toric diagram of $R$ is not a parallelogram. 
Then, at least one of pairs of slopes $([z_1], [z_3])$, $([z_2], [z_4])$ are linearly independent. We may assume that $[z_1]$ and $[z_3]$ are linearly independent. 
Thus, by Proposition~\ref{slope_isoradial} there is an intersection $E\in\Gamma_1$ of $z\in\calZ_1$ and $z^\prime\in\calZ_3$. 
We denote the white (resp. black) node that is an endpoint of $E$ by $w_E$ (resp. $b_E$). 
Then, the number of edges incident to $w_E$, which will be called the \emph{valency} of $w_E$, must be $3$ and zigzag paths factoring through $w_E$ are $z$, $z^\prime$ and the one contained in $\calZ_4$ because $\Gamma$ is properly ordered (see Figure~\ref{behavior_zigzag}). Also, the same properties hold for the node $b_E$. 
By Proposition~\ref{zigzag_sidepolygon}, the edges that are intersections of $z$ (resp. $z^\prime$) and a zigzag path in $\calZ_4$ are contained in $\sfP_4$ (resp. $\sfP_3$). 
Since each small cycle $\omega$ satisfies $\sfP_i(\omega)=1$ for all $i=1,\cdots,4$, the edge $E$ is contained in both $\sfP_1$ and $\sfP_2$. 

\begin{figure}[H]
\begin{tabular}{c}
\begin{minipage}{0.5\hsize}
\begin{center}
{\scalebox{0.7}{
\begin{tikzpicture}
\newcommand{\edgewidth}{0.05cm} 
\newcommand{\nodewidth}{0.05cm} 
\newcommand{\noderad}{0.16} 

\coordinate (W1) at (0,0); \coordinate (B1) at (3,0); 
\path (W1) ++(135:1.5cm) coordinate (W1a); \path (W1) ++(225:1.5cm) coordinate (W1b); 
\path (B1) ++(45:1.5cm) coordinate (B1a); \path (B1) ++(315:1.5cm) coordinate (B1b); 

\draw [line width=\edgewidth] (W1)--(B1); \draw [line width=\edgewidth] (W1)--(W1a); \draw [line width=\edgewidth] (W1)--(W1b); 
\draw [line width=\edgewidth] (B1)--(B1a); \draw [line width=\edgewidth] (B1)--(B1b); 
\draw  [line width=\nodewidth, fill=white] (W1) circle [radius=\noderad] ; 
\draw [line width=\nodewidth, fill=black] (B1) circle [radius=\noderad] ; 

\path (W1) ++(270:0.3cm) coordinate (W1-); \path (W1) ++(90:0.3cm) coordinate (W1+); 
\path (B1) ++(270:0.3cm) coordinate (B1-); \path (B1) ++(90:0.3cm) coordinate (B1+); 
\path (W1-) ++(225:1.7cm) coordinate (W1--); \path (W1+) ++(135:1.7cm) coordinate (W1++); 
\path (B1-) ++(315:1.7cm) coordinate (B1--); \path (B1+) ++(45:1.7cm) coordinate (B1++); 
\draw [->, rounded corners, line width=0.08cm, red] (W1--)--(W1-)--(B1+)--(B1++) ; 
\draw [->, rounded corners, line width=0.08cm, blue] (B1--)--(B1-)--(W1+)--(W1++) ; 

\path (W1) ++(180:0.3cm) coordinate (W1w); \path (W1w) ++(135:1.7cm) coordinate (W1w+); \path (W1w) ++(225:1.7cm) coordinate (W1w-); 
\path (B1) ++(0:0.3cm) coordinate (B1e); \path (B1e) ++(45:1.7cm) coordinate (B1e+); \path (B1e) ++(315:1.7cm) coordinate (B1e-); 
\draw [->, rounded corners, line width=0.08cm, green] (W1w+)--(W1w)--(W1w-) ; 
\draw [->, rounded corners, line width=0.08cm, green] (B1e+)--(B1e)--(B1e-) ; 
\node[red] at (4.5,1.9) {\Large$z\in\calZ_1$}; \node[blue] at (-1.5,1.9) {\Large$z^\prime\in\calZ_3$}; 
\node[green] at (-1.5,0) {\Large$w^{\prime}\in\calZ_4$} ; 
\node[green] at (4.5,0) {\Large$w\in\calZ_4$} ; 
\end{tikzpicture}
} }
\end{center}
\caption{Zigzag paths around the intersection $E$}
\label{behavior_zigzag}
\end{minipage}

\begin{minipage}{0.5\hsize}
\begin{center}
{\scalebox{0.7}{
\begin{tikzpicture}
\newcommand{\edgewidth}{0.05cm} 
\newcommand{\nodewidth}{0.05cm} 
\newcommand{\noderad}{0.16} 

\coordinate (W1) at (0,0); \coordinate (B1) at (3,0); 
\path (W1) ++(135:1.5cm) coordinate (W1a); \path (W1) ++(225:1.5cm) coordinate (W1b); 
\path (B1) ++(45:1.5cm) coordinate (B1a); \path (B1) ++(315:1.5cm) coordinate (B1b); 

\draw [lightgray, line width=\edgewidth] (W1)--(B1); \draw [lightgray, line width=\edgewidth] (W1)--(W1a); \draw [lightgray, line width=\edgewidth] (W1)--(W1b); 
\draw [lightgray, line width=\edgewidth] (B1)--(B1a); \draw [lightgray, line width=\edgewidth] (B1)--(B1b); 
\draw  [lightgray, line width=\nodewidth, fill=white] (W1) circle [radius=\noderad] ; 
\draw [lightgray, line width=\nodewidth, fill=lightgray, ] (B1) circle [radius=\noderad] ; 

\node (Va) at (1.5,1.3) {\Large$\alpha$}; \node (Vb) at (1.5,-1.3) {\Large$\beta$}; 
\node (Vc) at (-1.2,0) {\Large$\gamma$}; \node (Vd) at (4.2,0) {\Large$\delta$}; 
\draw [->, line width=\edgewidth] (Va)--(Vb); \draw [->, line width=\edgewidth] (Vb)--(Vc); \draw [->, line width=\edgewidth] (Vc)--(Va); 
\draw [->, line width=\edgewidth] (Vb)--(Vd); \draw [->, line width=\edgewidth] (Vd)--(Va); 
\node at (1.2,0.3) {\Large$a_1$}; \node at (0.2,-1) {\Large$a_2$}; \node at (0.2,1) {\Large$a_3$}; 
\node at (2.8,-1) {\Large$a_4$}; \node at (2.8,1) {\Large$a_5$}; 
\end{tikzpicture}
} }
\end{center}
\caption{The arrows around $w_E$ and $b_E$}
\label{behavior_PM}
\end{minipage}

\end{tabular}
\end{figure}

Let $(Q,W_Q)$ be the QP associated with $\Gamma$. Thus, $\calP\coloneqq\calP(Q,W_Q)$ is a semi-steady NCCR of $R$ that is not steady.  
Let $\alpha,\beta,\gamma,\delta$ be vertices of $Q$ appearing around $w_E,b_E$ and $a_1,\cdots,a_5$ be arrows between these vertices as shown in Figure~\ref{behavior_PM}. 
Thus, these arrows give divisorial ideals $T_{a_1}=T(1,1,0,0)$, $T_{a_3}=T_{a_4}=T(0,0,1,0)$, and $T_{a_2}=T_{a_5}=T(0,0,0,1)$. 
Thus, we see that 
\begin{center}
\begin{tabular}{ll}
$e_\alpha\calP\cong (e_\beta\calP\otimes_RT_{a_1})^{**}$,& $e_\gamma\calP\cong (e_\alpha\calP\otimes_RT_{a_3})^{**}$\\
$e_\delta\calP\cong (e_\alpha\calP\otimes_RT_{a_5})^{**}$,&
$e_\beta\calP\cong (e_\gamma\calP\otimes_RT_{a_2})^{**}\cong(e_\delta\calP\otimes_RT_{a_4})^{**}$. 
\end{tabular}
\end{center}

Let $v_1^\prime,\cdots,v_4^\prime\in\ZZ^2$ be the vertices of the toric diagram of $R$ corresponding to $\sfP_1,\cdots,\sfP_4$ respectively. 
Let $v_i\coloneqq(v_i^\prime,1)$ for $i=1,\cdots,4$. Then, the cone $\sigma$ generated by $v_1,\cdots,v_4$ defines $R$. 
Let $D_i\coloneqq[T(\delta_{i1},\cdots,\delta_{i4})]\in\Cl(R)$ for $i=1,\cdots,4$, where $\delta_{ij}$ is the Kronecker delta. 
Thus, we have $[T_{a_1}]=D_1+D_2,\, [T_{a_3}]=[T_{a_4}]=D_3,\, [T_{a_2}]=[T_{a_5}]=D_4$. 
By Lemma~\ref{cl_toric}, these satisfy 
\begin{equation}
\label{relations_divisor}
v_1D_1+v_2D_2+v_3D_3+v_4D_4=0. 
\end{equation}

By Lemma~\ref{lem_semi}, there is a splitting generator $M=\bigoplus_{i\in Q_0}M_i$ such that $\calP\cong\End_R(M)$ and $e_i\calP\cong M$ or $M^*$ for any $i\in Q_0$. Let $\calM\coloneqq\{[M_i]\in\Cl(R)\mid i\in Q_0\}$. 
If $(M\otimes_RI)^{**}\cong M$ holds for a divisorial ideal $I$, then we have $[I]\in\calM$ because $M$ is a generator. 
Using the above isomorphism again, we have $2[I]\in\calM$. 
Repeating this argument, we see that $[I]$ is a torsion element in $\Cl(R)$ because of the finiteness of $\calM$. 
Here, we assume that $D_4$ is torsion in $\Cl(R)$, and hence there is a positive integer $a$ such that $aD_4=0$ in $\Cl(R)$. 
By Lemma~\ref{cl_toric}, the equation $aD_4=0$ can be obtained by the relation (\ref{relations_divisor}). 
Thus, for the $3\times 3$ matrix $V\coloneqq(v_1\,v_2\,v_3)$, there exists a unimodular matrix $U$ such that one of rows of $UV$ is the zero vector. 
This means that $sa_1+tb_1=sa_2+tb_2=sa_3+tb_3$ for some $(0,0){\neq}(s,t)\in\ZZ^2$ where $v_i={}^t(a_i,b_i,1)$. 
We easily see that it is impossible to take such $(s,t)$ because the lattice points $v_1^\prime,v_2^\prime,v_3^\prime$ do not lie on the same line in $\RR^2$. Thus, $D_4$ is not torsion in $\Cl(R)$. By a similar argument, we also see that $D_3$ is not torsion. 
By these observations, we have 
\[
e_\alpha\calP\not\cong e_\gamma\calP,\quad e_\alpha\calP\not\cong e_\delta\calP,\quad e_\beta\calP\not\cong e_\gamma\calP,\quad e_\beta\calP\not\cong e_\delta\calP. 
\]
If $e_\alpha\calP\cong M$, then we have that $e_\beta\calP\cong M, e_\gamma\calP\cong M^*$ and $e_\delta\calP\cong M^*$, and hence 
\begin{equation}
\label{eq_proof1}
M^*\cong e_\gamma\calP\cong(e_\alpha\calP\otimes_RT_{a_3})^{**}\cong(M\otimes_RT_{a_3}), 
\end{equation}
\begin{equation}
\label{eq_proof2}
M\cong e_\beta\calP\cong(e_\delta\calP\otimes_RT_{a_4})^{**}\cong(M^*\otimes_RT_{a_4}). 
\end{equation}
Let $\calM^*\coloneqq\{[M_i^*]=-[M_i]\in\Cl(R)\mid i\in Q_0\}$. 
Then, by (\ref{eq_proof1}) we have $D_3\in\calM^*$. Using (\ref{eq_proof2}), we then have $2D_3\in\calM$. 
Since $D_3$ is not torsion, we can repeat these arguments infinitely, but this contradicts the finiteness of $\calM$. 
Even if $e_\alpha\calP\cong M^*$, we have the same conclusion by a similar argument. Therefore, the toric diagram of $R$ is a parallelogram. 
\end{proof}

In order to complete the proof of Theorem~\ref{main_thm}, we requre the following lemmas. 

\begin{lemma}
\label{key_lem}
With the notation as in the proof of Theorem~\ref{main_thm} $(2)$$\Rightarrow$$(3)$, we have that 
\begin{align*}
T(\sfP^{\rm op}_1(b_{k^\prime j^\prime}),\cdots,\sfP^{\rm op}_4(b_{k^\prime j^\prime}))
&\cong T(\sfP_3(b_{k^\prime j^\prime}),\sfP_4(b_{k^\prime j^\prime}),\sfP_1(b_{k^\prime j^\prime}),\sfP_2(b_{k^\prime j^\prime}))\\
&\cong T(\sfP_1(a_{kj}),\cdots,\sfP_4(a_{kj}))^*
\end{align*}
for each $j\in Q_0$. 
\end{lemma}

\begin{proof}
Let $z$ be a zigzag path on $\Gamma$. By replacing white nodes with black ones and vice versa, we have $\Gamma^{\rm op}$ and 
the associated quiver $Q^{\rm op}$. Then, $-z$ is a zigzag path on $\Gamma^{\rm op}$. 
Considering slopes of zigzag paths, we see that extremal perfect matchings on $\Gamma^{\rm op}$ corresponding to vertices $u_1,\cdots,u_4$ 
are $\sfP_3, \sfP_4,\sfP_1,\sfP_2$ respectively by Proposition~\ref{zigzag_sidepolygon}. 
Therefore, we obtain the first isomorphism. 

Next, we consider the operation of shifting a path $a_{kj}$ to $b_{k^\prime j^\prime}$. 
By this operation, an arrow evaluating on $\sfP_1$ will shift to that on $-\sfP_3$. 
Similarly, an arrow evaluating on $\sfP_2,\sfP_3,\sfP_4$ will shift to that on $-\sfP_4,-\sfP_1,-\sfP_2$ respectively. 
Therefore, we obtain the second isomorphism. 
\end{proof}

By combining this theorem with Theorem~\ref{dimer}, we obtain a characterization of dimer models that are homotopy equivalent to regular dimer models 
in terms of NCCRs. 

\begin{corollary}
\label{main_cor}
With the notation as Theorem~\ref{main_thm}, 
the following conditions are equivalent.
\begin{enumerate}[\rm (1)]
\item $\Gamma$ is isoradial and gives a semi-steady NCCR of $R$.
\item $\Gamma$ is homotopy equivalent to a regular dimer model. 
\end{enumerate}
When this is the case, the toric diagram of $R$ is a triangle or parallelogram. 
\end{corollary}

\section{\bf Examples}
\label{subsec_ex}

We end this paper by giving several examples. 

\begin{example}[{See also \cite[Corollary~1.7 and Example~1.8]{IN}, \cite{UY}}]
\label{ex_steady}
The following figures are a consistent dimer model that is homotopy equivalent to a regular hexagonal dimer model, and the associated quiver. 
Here, the red area denotes the fundamental domain of the two-torus. 
This quiver coincides with the McKay quiver of $G=\langle\mathrm{diag}(\omega,\omega^2,\omega^4)\rangle$ 
where $\omega$ is a primitive $7$-th root of unity, 
and the complete Jacobian algebra is isomorphic to the skew group ring $S*G$ where $S\coloneqq k[[x_1,x_2,x_3]]$. 
Furthermore, the center of the complete Jacobian algebra is the quotient singularity $R=S^G$. 
By Theorem~\ref{dimer}, this dimer model gives a steady NCCR of $R$, which is $\End_R(S)\cong S*G$. 

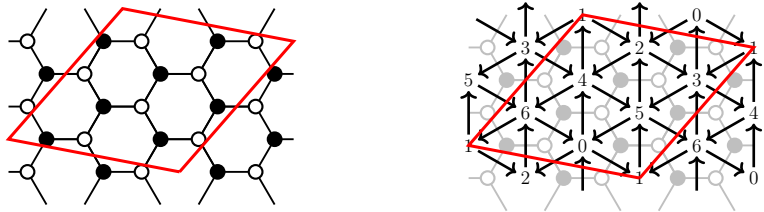
\medskip

\begin{center}
\begin{tikzpicture}
\node (DM) at (0,0) 
{\scalebox{0.5}{
\begin{tikzpicture}
\coordinate (e05) at (0:1cm); \coordinate (e15) at (60:1cm); \coordinate (e25) at (120:1cm); 
\coordinate (e35) at (180:1cm); \coordinate (e45) at (240:1cm); \coordinate (e55) at (300:1cm); 

\path (e15) ++(1, 0) coordinate (V3); 
\path (V3) ++(0:1cm) coordinate (e03); \path (V3) ++(60:1cm) coordinate (e13); \path (V3) ++(120:1cm) coordinate (e23);
\path (V3) ++(180:1cm) coordinate (e33); \path (V3) ++(240:1cm) coordinate (e43); \path (V3) ++(300:1cm) coordinate (e53);

\path (e55) ++(1, 0) coordinate (V6); 
\path (V6) ++(0:1cm) coordinate (e06); \path (V6) ++(60:1cm) coordinate (e16); \path (V6) ++(120:1cm) coordinate (e26);
\path (V6) ++(180:1cm) coordinate (e36); \path (V6) ++(240:1cm) coordinate (e46); \path (V6) ++(300:1cm) coordinate (e56);

\path (e23) ++(-1, 0) coordinate (V2); 
\path (V2) ++(0:1cm) coordinate (e02); \path (V2) ++(60:1cm) coordinate (e12); \path (V2) ++(120:1cm) coordinate (e22);
\path (V2) ++(180:1cm) coordinate (e32); \path (V2) ++(240:1cm) coordinate (e42); \path (V2) ++(300:1cm) coordinate (e52);

\path (e46) ++(-1, 0) coordinate (V1); 
\path (V1) ++(0:1cm) coordinate (e01); \path (V1) ++(60:1cm) coordinate (e11); \path (V1) ++(120:1cm) coordinate (e21);
\path (V1) ++(180:1cm) coordinate (e31); \path (V1) ++(240:1cm) coordinate (e41); \path (V1) ++(300:1cm) coordinate (e51);

\path (e25) ++(-1, 0) coordinate (V4); 
\path (V4) ++(0:1cm) coordinate (e04); \path (V4) ++(60:1cm) coordinate (e14); \path (V4) ++(120:1cm) coordinate (e24);
\path (V4) ++(180:1cm) coordinate (e34); \path (V4) ++(240:1cm) coordinate (e44); \path (V4) ++(300:1cm) coordinate (e54);

\path (e45) ++(-1, 0) coordinate (V0); 
\path (V0) ++(0:1cm) coordinate (e00); \path (V0) ++(60:1cm) coordinate (e10); \path (V0) ++(120:1cm) coordinate (e20);
\path (V0) ++(180:1cm) coordinate (e30); \path (V0) ++(240:1cm) coordinate (e40); \path (V0) ++(300:1cm) coordinate (e50);

\path (e13) ++(1, 0) coordinate (V1a); 
\path (V1a) ++(0:1cm) coordinate (e01a); \path (V1a) ++(60:1cm) coordinate (e11a); \path (V1a) ++(120:1cm) coordinate (e21a);
\path (V1a) ++(180:1cm) coordinate (e31a); \path (V1a) ++(240:1cm) coordinate (e41a); \path (V1a) ++(300:1cm) coordinate (e51a);

\path (e53) ++(1, 0) coordinate (V4a); 
\path (V4a) ++(0:1cm) coordinate (e04a); \path (V4a) ++(60:1cm) coordinate (e14a); \path (V4a) ++(120:1cm) coordinate (e24a);
\path (V4a) ++(180:1cm) coordinate (e34a); \path (V4a) ++(240:1cm) coordinate (e44a); \path (V4a) ++(300:1cm) coordinate (e54a);

\path (e56) ++(1, 0) coordinate (V0a); 
\path (V0a) ++(0:1cm) coordinate (e00a); \path (V0a) ++(60:1cm) coordinate (e10a); \path (V0a) ++(120:1cm) coordinate (e20a);
\path (V0a) ++(180:1cm) coordinate (e30a); \path (V0a) ++(240:1cm) coordinate (e40a); \path (V0a) ++(300:1cm) coordinate (e50a);

\path (e24) ++(-1, 0) coordinate (V3a); 
\path (V3a) ++(0:1cm) coordinate (e03a); \path (V3a) ++(60:1cm) coordinate (e13a); \path (V3a) ++(120:1cm) coordinate (e23a);
\path (V3a) ++(180:1cm) coordinate (e33a); \path (V3a) ++(240:1cm) coordinate (e43a); \path (V3a) ++(300:1cm) coordinate (e53a);

\path (e44) ++(-1, 0) coordinate (V6a); 
\path (V6a) ++(0:1cm) coordinate (e06a); \path (V6a) ++(60:1cm) coordinate (e16a); \path (V6a) ++(120:1cm) coordinate (e26a);
\path (V6a) ++(180:1cm) coordinate (e36a); \path (V6a) ++(240:1cm) coordinate (e46a); \path (V6a) ++(300:1cm) coordinate (e56a);

\path (e40) ++(-1, 0) coordinate (V2a); 
\path (V2a) ++(0:1cm) coordinate (e02a); \path (V2a) ++(60:1cm) coordinate (e12a); \path (V2a) ++(120:1cm) coordinate (e22a);
\path (V2a) ++(180:1cm) coordinate (e32a); \path (V2a) ++(240:1cm) coordinate (e42a); \path (V2a) ++(300:1cm) coordinate (e52a);

\draw [line width=0.05cm]  (e00)--(e10)--(e20)--(e30)--(e40)--(e50)--(e00) ;
\draw [line width=0.05cm]  (e51)--(e01)--(e11)--(e21)--(e31)--(e41) ;
\draw [line width=0.05cm]  (e02)--(e12);
\draw [line width=0.05cm]  (e22)--(e32)--(e42)--(e52)--(e02) ;
\draw [line width=0.05cm]  (e03)--(e13)--(e23)--(e33)--(e43)--(e53)--(e03) ; 
\draw [line width=0.05cm]  (e04)--(e14)--(e24)--(e34)--(e44)--(e54)--(e04) ;
\draw [line width=0.05cm]  (e05)--(e15)--(e25)--(e35)--(e45)--(e55)--(e05) ;
\draw [line width=0.05cm]  (e06)--(e16)--(e26)--(e36)--(e46)--(e56)--(e06) ;

\draw [line width=0.05cm]  (e21a)--(e31a) ;
\draw [line width=0.05cm]  (e30a)--(e40a) ;
\draw [line width=0.05cm]  (e03a)--(e13a) ;
\draw [line width=0.05cm]  (e23a)--(e33a)--(e43a)--(e53a)--(e03a) ;
\draw [line width=0.05cm]  (e06a)--(e16a)--(e26a)--(e36a)--(e46a)--(e56a)--(e06a) ;
\draw [line width=0.05cm]  (e52a)--(e02a)--(e12a)--(e22a)--(e32a)--(e42a) ; 

\draw [line width=0.05cm] (e33a)-- +(-0.5,0) ; \draw [line width=0.05cm] (e36a)-- +(-0.5,0) ; \draw [line width=0.05cm] (e32a)-- +(-0.5,0) ;
\draw [line width=0.05cm] (e03)-- +(0.5,0) ; \draw [line width=0.05cm] (e06)-- +(0.5,0) ; 

\path (e22) ++(-1, 0) coordinate (V1b); 
\path (e22a) ++(-1, 0) coordinate (V1c); 
\path (e26a) ++(-1, 0) coordinate (V5a); 

\filldraw [ultra thick, fill=black] (e05) circle [radius=0.18]; \filldraw [ultra thick, fill=black] (e25) circle [radius=0.18]; \filldraw [ultra thick, fill=black] (e45) circle [radius=0.18];
\filldraw [ultra thick, fill=white] (e15) circle [radius=0.18]; \filldraw [ultra thick, fill=white] (e35) circle [radius=0.18]; \filldraw [ultra thick, fill=white] (e55) circle [radius=0.18]; 
\filldraw [ultra thick, fill=black] (e03) circle [radius=0.18]; \filldraw [ultra thick, fill=black] (e23) circle [radius=0.18]; 
\filldraw [ultra thick, fill=white] (e13) circle [radius=0.18]; \filldraw [ultra thick, fill=white] (e53) circle [radius=0.18]; 
\filldraw [ultra thick, fill=black] (e24) circle [radius=0.18]; \filldraw [ultra thick, fill=black] (e44) circle [radius=0.18];
\filldraw [ultra thick, fill=white] (e14) circle [radius=0.18]; \filldraw [ultra thick, fill=white] (e34) circle [radius=0.18]; 
\filldraw [ultra thick, fill=black] (e40) circle [radius=0.18];
\filldraw [ultra thick, fill=white] (e30) circle [radius=0.18]; \filldraw [ultra thick, fill=white] (e50) circle [radius=0.18]; 
\filldraw [ultra thick, fill=black] (e06) circle [radius=0.18]; \filldraw [ultra thick, fill=black] (e46) circle [radius=0.18]; 
\filldraw [ultra thick, fill=white] (e56) circle [radius=0.18]; 
\filldraw [ultra thick, fill=black] (e26a) circle [radius=0.18]; \filldraw [ultra thick, fill=black] (e46a) circle [radius=0.18]; 
\filldraw [ultra thick, fill=white] (e36a) circle [radius=0.18]; 
\filldraw [ultra thick, fill=white] (e33a) circle [radius=0.18]; \filldraw [ultra thick, fill=white] (e32a) circle [radius=0.18]; 

\draw[line width=0.08cm, red]  (V1)--(V1a)--(V1b)--(V1c)--(V1) ; 
\end{tikzpicture}
} } ;  

\node (QV) at (6,0) 
{\scalebox{0.5}{
\begin{tikzpicture}
\node (V5) at (0,0) {}; 
\node (n5) at (V5) {{\LARGE$5$}}; 
\coordinate (e05) at (0:1cm); \coordinate (e15) at (60:1cm); \coordinate (e25) at (120:1cm); 
\coordinate (e35) at (180:1cm); \coordinate (e45) at (240:1cm); \coordinate (e55) at (300:1cm); 

\path (e15) ++(1, 0) coordinate (V3); \node (n3) at (V3) {{\LARGE$3$}}; 
\path (V3) ++(0:1cm) coordinate (e03); \path (V3) ++(60:1cm) coordinate (e13); \path (V3) ++(120:1cm) coordinate (e23);
\path (V3) ++(180:1cm) coordinate (e33); \path (V3) ++(240:1cm) coordinate (e43); \path (V3) ++(300:1cm) coordinate (e53);

\path (e55) ++(1, 0) coordinate (V6); \node (n6) at (V6) {{\LARGE$6$}}; 
\path (V6) ++(0:1cm) coordinate (e06); \path (V6) ++(60:1cm) coordinate (e16); \path (V6) ++(120:1cm) coordinate (e26);
\path (V6) ++(180:1cm) coordinate (e36); \path (V6) ++(240:1cm) coordinate (e46); \path (V6) ++(300:1cm) coordinate (e56);

\path (e23) ++(-1, 0) coordinate (V2); \node (n2) at (V2) {{\LARGE$2$}}; 
\path (V2) ++(0:1cm) coordinate (e02); \path (V2) ++(60:1cm) coordinate (e12); \path (V2) ++(120:1cm) coordinate (e22);
\path (V2) ++(180:1cm) coordinate (e32); \path (V2) ++(240:1cm) coordinate (e42); \path (V2) ++(300:1cm) coordinate (e52);

\path (e46) ++(-1, 0) coordinate (V1); \node (n1) at (V1) {{\LARGE$1$}};
\path (V1) ++(0:1cm) coordinate (e01); \path (V1) ++(60:1cm) coordinate (e11); \path (V1) ++(120:1cm) coordinate (e21);
\path (V1) ++(180:1cm) coordinate (e31); \path (V1) ++(240:1cm) coordinate (e41); \path (V1) ++(300:1cm) coordinate (e51);

\path (e25) ++(-1, 0) coordinate (V4); \node (n4) at (V4) {{\LARGE$4$}};
\path (V4) ++(0:1cm) coordinate (e04); \path (V4) ++(60:1cm) coordinate (e14); \path (V4) ++(120:1cm) coordinate (e24);
\path (V4) ++(180:1cm) coordinate (e34); \path (V4) ++(240:1cm) coordinate (e44); \path (V4) ++(300:1cm) coordinate (e54);

\path (e45) ++(-1, 0) coordinate (V0); \node (n0) at (V0) {{\LARGE$0$}};
\path (V0) ++(0:1cm) coordinate (e00); \path (V0) ++(60:1cm) coordinate (e10); \path (V0) ++(120:1cm) coordinate (e20);
\path (V0) ++(180:1cm) coordinate (e30); \path (V0) ++(240:1cm) coordinate (e40); \path (V0) ++(300:1cm) coordinate (e50);

\path (e13) ++(1, 0) coordinate (V1a); \node (n1a) at (V1a) {{\LARGE$1$}};
\path (V1a) ++(0:1cm) coordinate (e01a); \path (V1a) ++(60:1cm) coordinate (e11a); \path (V1a) ++(120:1cm) coordinate (e21a);
\path (V1a) ++(180:1cm) coordinate (e31a); \path (V1a) ++(240:1cm) coordinate (e41a); \path (V1a) ++(300:1cm) coordinate (e51a);

\path (e53) ++(1, 0) coordinate (V4a); \node (n4a) at (V4a) {{\LARGE$4$}};
\path (V4a) ++(0:1cm) coordinate (e04a); \path (V4a) ++(60:1cm) coordinate (e14a); \path (V4a) ++(120:1cm) coordinate (e24a);
\path (V4a) ++(180:1cm) coordinate (e34a); \path (V4a) ++(240:1cm) coordinate (e44a); \path (V4a) ++(300:1cm) coordinate (e54a);

\path (e56) ++(1, 0) coordinate (V0a); \node (n0a) at (V0a) {{\LARGE$0$}};
\path (V0a) ++(0:1cm) coordinate (e00a); \path (V0a) ++(60:1cm) coordinate (e10a); \path (V0a) ++(120:1cm) coordinate (e20a);
\path (V0a) ++(180:1cm) coordinate (e30a); \path (V0a) ++(240:1cm) coordinate (e40a); \path (V0a) ++(300:1cm) coordinate (e50a);

\path (e24) ++(-1, 0) coordinate (V3a); \node (n3a) at (V3a) {{\LARGE$3$}};
\path (V3a) ++(0:1cm) coordinate (e03a); \path (V3a) ++(60:1cm) coordinate (e13a); \path (V3a) ++(120:1cm) coordinate (e23a);
\path (V3a) ++(180:1cm) coordinate (e33a); \path (V3a) ++(240:1cm) coordinate (e43a); \path (V3a) ++(300:1cm) coordinate (e53a);

\path (e44) ++(-1, 0) coordinate (V6a); \node (n6a) at (V6a) {{\LARGE$6$}};
\path (V6a) ++(0:1cm) coordinate (e06a); \path (V6a) ++(60:1cm) coordinate (e16a); \path (V6a) ++(120:1cm) coordinate (e26a);
\path (V6a) ++(180:1cm) coordinate (e36a); \path (V6a) ++(240:1cm) coordinate (e46a); \path (V6a) ++(300:1cm) coordinate (e56a);

\path (e40) ++(-1, 0) coordinate (V2a); \node (n2a) at (V2a) {{\LARGE$2$}};
\path (V2a) ++(0:1cm) coordinate (e02a); \path (V2a) ++(60:1cm) coordinate (e12a); \path (V2a) ++(120:1cm) coordinate (e22a);
\path (V2a) ++(180:1cm) coordinate (e32a); \path (V2a) ++(240:1cm) coordinate (e42a); \path (V2a) ++(300:1cm) coordinate (e52a);

\draw [line width=0.05cm, lightgray]  (e00)--(e10)--(e20)--(e30)--(e40)--(e50)--(e00) ;
\draw [line width=0.05cm, lightgray]  (e51)--(e01)--(e11)--(e21)--(e31)--(e41) ;
\draw [line width=0.05cm, lightgray]  (e02)--(e12);
\draw [line width=0.05cm, lightgray]  (e22)--(e32)--(e42)--(e52)--(e02) ;
\draw [line width=0.05cm, lightgray]  (e03)--(e13)--(e23)--(e33)--(e43)--(e53)--(e03) ; 
\draw [line width=0.05cm, lightgray]  (e04)--(e14)--(e24)--(e34)--(e44)--(e54)--(e04) ;
\draw [line width=0.05cm, lightgray]  (e05)--(e15)--(e25)--(e35)--(e45)--(e55)--(e05) ;
\draw [line width=0.05cm, lightgray]  (e06)--(e16)--(e26)--(e36)--(e46)--(e56)--(e06) ;

\draw [line width=0.05cm, lightgray]  (e21a)--(e31a) ;
\draw [line width=0.05cm, lightgray]  (e30a)--(e40a) ;
\draw [line width=0.05cm, lightgray]  (e03a)--(e13a) ;
\draw [line width=0.05cm, lightgray]  (e23a)--(e33a)--(e43a)--(e53a)--(e03a) ;
\draw [line width=0.05cm, lightgray]  (e06a)--(e16a)--(e26a)--(e36a)--(e46a)--(e56a)--(e06a) ;
\draw [line width=0.05cm, lightgray]  (e52a)--(e02a)--(e12a)--(e22a)--(e32a)--(e42a) ; 

\draw [line width=0.05cm, lightgray] (e33a)-- +(-0.5,0) ; \draw [line width=0.05cm, lightgray] (e36a)-- +(-0.5,0) ; \draw [line width=0.05cm, lightgray] (e32a)-- +(-0.5,0) ;
\draw [line width=0.05cm, lightgray] (e03)-- +(0.5,0) ; \draw [line width=0.05cm, lightgray] (e06)-- +(0.5,0) ; 

\path (e22) ++(-1, 0) coordinate (V1b); \node (n1b) at (V1b) {{\LARGE$1$}};
\path (e22a) ++(-1, 0) coordinate (V1c); \node (n1c) at (V1c) {{\LARGE$1$}};
\path (e26a) ++(-1, 0) coordinate (V5a); \node (n5a) at (V5a) {{\LARGE$5$}};

\filldraw [ultra thick, draw=lightgray, fill=lightgray] (e05) circle [radius=0.18]; \filldraw [ultra thick, draw=lightgray, fill=lightgray] (e25) circle [radius=0.18]; 
\filldraw [ultra thick, draw=lightgray, fill=lightgray] (e45) circle [radius=0.18];
\filldraw [ultra thick, draw=lightgray, fill=white] (e15) circle [radius=0.18]; \filldraw [ultra thick, draw=lightgray, fill=white] (e35) circle [radius=0.18]; 
\filldraw [ultra thick, draw=lightgray, fill=white] (e55) circle [radius=0.18]; 
\filldraw [ultra thick, draw=lightgray, fill=lightgray] (e03) circle [radius=0.18]; \filldraw [ultra thick, draw=lightgray, fill=lightgray] (e23) circle [radius=0.18]; 
\filldraw [ultra thick, draw=lightgray, fill=white] (e13) circle [radius=0.18]; \filldraw [ultra thick, draw=lightgray, fill=white] (e53) circle [radius=0.18]; 
\filldraw [ultra thick, draw=lightgray, fill=lightgray] (e24) circle [radius=0.18]; \filldraw [ultra thick, draw=lightgray, fill=lightgray] (e44) circle [radius=0.18];
\filldraw [ultra thick, draw=lightgray, fill=white] (e14) circle [radius=0.18]; \filldraw [ultra thick, draw=lightgray, fill=white] (e34) circle [radius=0.18]; 
\filldraw [ultra thick, draw=lightgray, fill=lightgray] (e40) circle [radius=0.18];
\filldraw [ultra thick, draw=lightgray, fill=white] (e30) circle [radius=0.18]; \filldraw [ultra thick, draw=lightgray, fill=white] (e50) circle [radius=0.18]; 
\filldraw [ultra thick, draw=lightgray, fill=lightgray] (e06) circle [radius=0.18]; \filldraw [ultra thick, draw=lightgray, fill=lightgray] (e46) circle [radius=0.18]; 
\filldraw [ultra thick, draw=lightgray, fill=white] (e56) circle [radius=0.18]; 
\filldraw [ultra thick, draw=lightgray, fill=lightgray] (e26a) circle [radius=0.18]; \filldraw [ultra thick, draw=lightgray, fill=lightgray] (e46a) circle [radius=0.18]; 
\filldraw [ultra thick, draw=lightgray, fill=white] (e36a) circle [radius=0.18]; 
\filldraw [ultra thick, draw=lightgray, fill=white] (e33a) circle [radius=0.18]; \filldraw [ultra thick, draw=lightgray, fill=white] (e32a) circle [radius=0.18]; 

\draw[->, line width=0.07cm] (n0)--(n1); \draw[->, line width=0.07cm] (n0)--(n4); \draw[->, line width=0.07cm] (n0)--(n2a);  
\draw[->, line width=0.07cm] (n5)--(n0); \draw[->, line width=0.07cm] (n5)--(n2); \draw[->, line width=0.07cm] (n5)--(n6);
\draw[->, line width=0.07cm] (n6a)--(n0); \draw[->, line width=0.07cm] (n6a)--(n3a); \draw[->, line width=0.07cm] (n6a)--(n1c);
\draw[->, line width=0.07cm] (n4)--(n5); \draw[->, line width=0.07cm] (n4)--(n6a); \draw[->, line width=0.07cm] (n4)--(n1b); 
\draw[->, line width=0.07cm] (n2)--(n3); \draw[->, line width=0.07cm] (n2)--(n4); \draw[->, line width=0.07cm] (n1)--(n5); 
\draw[->, line width=0.07cm] (n3)--(n5); \draw[->, line width=0.07cm] (n3)--(n4a); 
\draw[->, line width=0.07cm] (n6)--(n1); \draw[->, line width=0.07cm] (n6)--(n3); \draw[->, line width=0.07cm] (n6)--(n0a);
\draw[->, line width=0.07cm] (n1b)--(n2); \draw[->, line width=0.07cm] (n1b)--(n3a); \draw[->, line width=0.07cm] (n4a)--(n6); 
\draw[->, line width=0.07cm] (n3a)--(n4); \draw[->, line width=0.07cm] (n3a)--(n5a); \draw[->, line width=0.07cm] (n1c)--(n2a); 
\draw[->, line width=0.07cm] (n1c)--(n5a); \draw[->, line width=0.07cm] (n5a)--(n6a); \draw[->, line width=0.07cm] (n2a)--(n6a); 
\draw[->, line width=0.07cm] (n4a)--(n1a); \draw[->, line width=0.07cm] (n1a)--(n3); \draw[->, line width=0.07cm] (n0a)--(n4a); 

\path (e12) ++(1,0) coordinate (V0b); \node (n0b) at (V0b) {{\LARGE$0$}};
\draw[->, line width=0.07cm] (n3)--(n0b); \draw[->, line width=0.07cm] (n0b)--(n2); \draw[->, line width=0.07cm] (n0b)--(n1a);
\draw[<-, line width=0.07cm] (n6)--++(0,-1.2); \draw[<-, line width=0.07cm] (n0)--++(0,-1.2); \draw[->, line width=0.07cm] (n2)--++(0,1.2); 
\draw[->, line width=0.07cm] (n3a)--++(0,1.2); \draw[<-, line width=0.07cm] (n3a)--++(150:1.5cm);

\draw[line width=0.08cm, red]  (V1)--(V1a)--(V1b)--(V1c)--(V1) ; 
\end{tikzpicture}
} } ;  

\end{tikzpicture}
\end{center}
\end{example}

\begin{example}
\label{ex_semisteady} 
Next, we consider the square dimer model given in Figure~\ref{ex_quiver4a}. 
For simplicity, we denote the complete Jacobian algebra associated with this dimer model by $\sfA$. 
Then, the center of $\sfA$ is the $3$-dimensional Gorenstein toric singularity $R=k[[\sigma^\vee\cap\ZZ^3]]$ defined by the cone $\sigma$:
\[
\sigma=\mathrm{Cone}\{v_1=(1,0,1), v_2=(0,1,1), v_3=(-1,0,1), v_4=(0,-1,1) \}. 
\]
For this singularity, we have that $\Cl(R)\cong\ZZ\times\ZZ/2\ZZ$, 
and hence each divisorial ideal is represented by $T(a,b,0,0)$ where $a\in\ZZ, b\in\ZZ/2\ZZ$. 
By Theorem~\ref{main_thm}, $\sfA$ is a semi-steady NCCR of $R$ (that is not steady). 
More precisely, we have that 
\begin{center}
\begin{tabular}{ll}
$e_i\sfA\cong R\oplus T(0,1,0,0)\oplus T(1,1,0,0)\oplus T(-1,0,0,0),$ \\ 
$e_j\sfA\cong R\oplus T(1,0,0,0)\oplus T(1,1,0,0)\oplus T(2,1,0,0), $
\end{tabular}
\end{center}
for $i=0,2$ and $j=1,3$ (see \cite[subsection~5.2]{Nak}). 
Further, we have that $(e_i\sfA)^*\cong e_j\sfA$, and these give semi-steady NCCRs of $R$ that are not steady. 
However, there exists another consistent dimer model associated with $R$ written below, 
and this is not homotopy equivalent to a regular dimer model. Thus, this does not give semi-steady NCCRs. 
A similar example is also found in \cite[subsection~5.11]{Nak}. 

\medskip


\begin{center}
{\scalebox{0.3}{
\begin{tikzpicture}
\coordinate (B1) at (1,-1); \coordinate (B2) at (2,2); \coordinate (B3) at (5,1); \coordinate (B4) at (4,-2); 
\coordinate (B5) at (-1,1); \coordinate (B6) at (-4,2); \coordinate (B7) at (-5,-1); \coordinate (B8) at (-2,-2); 
\coordinate (B9) at (-5,5); \coordinate (B10) at (-2,4); \coordinate (B11) at (1,5); \coordinate (B12) at (4,4); 
\coordinate (B13) at (-4,-4); \coordinate (B14) at (-1,-5); \coordinate (B15) at (2,-4); \coordinate (B16) at (5,-5); 

\coordinate (W1) at (1,1); \coordinate (W2) at (4,2); \coordinate (W3) at (5,-1); \coordinate (W4) at (2,-2); 
\coordinate (W5) at (-2,2); \coordinate (W6) at (-5,1); \coordinate (W7) at (-4,-2); \coordinate (W8) at (-1,-1); 
\coordinate (W9) at (-4,4); \coordinate (W10) at (-1,5); \coordinate (W11) at (2,4); \coordinate (W12) at (5,5); 
\coordinate (W13) at (-5,-5); \coordinate (W14) at (-2,-4); \coordinate (W15) at (1,-5); \coordinate (W16) at (4,-4); 

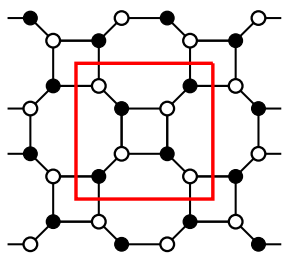
\draw[line width=0.09cm]  (B1)--(W1)--(B5)--(W8)--(B1); 
\draw[line width=0.09cm]  (W1)--(B2)--(W2)--(B3)--(W3)--(B4)--(W4)--(B1)--(W1);
\draw[line width=0.09cm]  (W8)--(B5)--(W5)--(B6)--(W6)--(B7)--(W7)--(B8)--(W8);
\draw[line width=0.09cm]  (B9)--(W9)--(B10)--(W10)--(B11)--(W11)--(B12)--(W12);
\draw[line width=0.09cm]  (W13)--(B13)--(W14)--(B14)--(W15)--(B15)--(W16)--(B16); 
\draw[line width=0.09cm]  (B6)--(W9)--(B10)--(W5)--(B6) ; \draw[line width=0.09cm]  (B2)--(W11)--(B12)--(W2)--(B2) ;
\draw[line width=0.09cm]  (B8)--(W7)--(B13)--(W14)--(B8) ; \draw[line width=0.09cm]  (B4)--(W4)--(B15)--(W16)--(B4) ;

\draw[line width=0.09cm]  (B3)--++(1,0); \draw[line width=0.09cm]  (W3)--++(1,0); \draw[line width=0.09cm]  (W12)--++(1,0); \draw[line width=0.09cm]  (B16)--++(1,0); 
\draw[line width=0.09cm]  (W6)--++(-1,0); \draw[line width=0.09cm]  (B7)--++(-1,0); \draw[line width=0.09cm]  (B9)--++(-1,0); \draw[line width=0.09cm]  (W13)--++(-1,0); 

\filldraw  [ultra thick, fill=black] (B1) circle [radius=0.3] ; \filldraw  [ultra thick, fill=black] (B2) circle [radius=0.3] ;
\filldraw  [ultra thick, fill=black] (B3) circle [radius=0.3] ; \filldraw  [ultra thick, fill=black] (B4) circle [radius=0.3] ;
\filldraw  [ultra thick, fill=black] (B5) circle [radius=0.3] ; \filldraw  [ultra thick, fill=black] (B6) circle [radius=0.3] ;
\filldraw  [ultra thick, fill=black] (B7) circle [radius=0.3] ; \filldraw  [ultra thick, fill=black] (B8) circle [radius=0.3] ;
\filldraw  [ultra thick, fill=black] (B9) circle [radius=0.3] ; \filldraw  [ultra thick, fill=black] (B10) circle [radius=0.3] ;
\filldraw  [ultra thick, fill=black] (B11) circle [radius=0.3] ; \filldraw  [ultra thick, fill=black] (B12) circle [radius=0.3] ;
\filldraw  [ultra thick, fill=black] (B13) circle [radius=0.3] ; \filldraw  [ultra thick, fill=black] (B14) circle [radius=0.3] ;
\filldraw  [ultra thick, fill=black] (B15) circle [radius=0.3] ; \filldraw  [ultra thick, fill=black] (B16) circle [radius=0.3] ;

\filldraw  [line width=0.1cm, fill=white] (W1) circle [radius=0.3] ; \filldraw  [line width=0.1cm, fill=white] (W2) circle [radius=0.3] ;
\filldraw  [line width=0.1cm, fill=white] (W3) circle [radius=0.3] ; \filldraw  [line width=0.1cm, fill=white] (W4) circle [radius=0.3] ;
\filldraw  [line width=0.1cm, fill=white] (W5) circle [radius=0.3] ; \filldraw  [line width=0.1cm, fill=white] (W6) circle [radius=0.3] ;
\filldraw  [line width=0.1cm, fill=white] (W7) circle [radius=0.3] ; \filldraw  [line width=0.1cm, fill=white] (W8) circle [radius=0.3] ;
\filldraw  [line width=0.1cm, fill=white] (W9) circle [radius=0.3] ; \filldraw  [line width=0.1cm, fill=white] (W10) circle [radius=0.3] ;
\filldraw  [line width=0.1cm, fill=white] (W11) circle [radius=0.3] ; \filldraw  [line width=0.1cm, fill=white] (W12) circle [radius=0.3] ;
\filldraw  [line width=0.1cm, fill=white] (W13) circle [radius=0.3] ; \filldraw  [line width=0.1cm, fill=white] (W14) circle [radius=0.3] ;
\filldraw  [line width=0.1cm, fill=white] (W15) circle [radius=0.3] ; \filldraw  [line width=0.1cm, fill=white] (W16) circle [radius=0.3] ;

\draw[line width=0.15cm, red] (3,3)--(-3,3)--(-3,-3)--(3,-3)--(3,3) ; 
\end{tikzpicture} }}
\end{center}

\end{example}

\begin{example} 
\label{ex3}
If $M$ is a semi-steady module, we have that $\add_R\End_R(M)=\add_R(M\oplus M^*)$ (see Lemma~\ref{basic_lem}(a)), 
but the converse is not true as follows. 

Let $R$ be the $3$-dimensional complete local Gorenstein toric singularity defined by the cone $\sigma$: 
\[
\sigma=\mathrm{Cone}\{v_1=(0,1,1), v_2=(-1,0,1), v_3=(0,-1,1), v_4=(1,-1,1) \}. 
\]
In this situation, we have that $\Cl(R)\cong\ZZ$, and each divisorial ideal is represented by $T(a,0,0,0)$ where $a\in\ZZ$. 
By the results in \cite[subsection~5.3]{Nak}, we see that 
$M=R\oplus T(1,0,0,0)\oplus T(2,0,0,0)\oplus T(3,0,0,0)$ gives an NCCR of $R$. 
Furthermore, we have that $\add_R\End_R(M)=\add_R(M\oplus M^*)$, but we can check $M$ is not semi-steady. 
\end{example}

\subsection*{Acknowledgements}
The author is supported by World Premier International Research Center Initiative (WPI initiative), MEXT, Japan, 
and JSPS Grant-in-Aid for Young Scientists (B) 17K14159. 

The author would like to thank the anonymous referees for valuable comments on this paper. 


\end{document}